 \documentclass{article}

\usepackage{amsmath,amsfonts,amsthm,amssymb,amscd,color}
\usepackage{enumitem}
\usepackage{verbatim}
\usepackage{xcolor}
\usepackage[dvipdfmx]{graphicx}
\usepackage{ulem}

\renewcommand{\Re}{\mathop{\rm Re}\nolimits}
\renewcommand{\Im}{\mathop{\rm Im}\nolimits}
\def\S{\mathhexbox278}

\theoremstyle{plain}
\newtheorem{theorem}{Theorem}[section]
\newtheorem{lemma}[theorem]{Lemma}
\newtheorem{proposition}[theorem]{Proposition}

\theoremstyle{definition}
\newtheorem{definition}[theorem]{Definition}
\theoremstyle{remark}
\newtheorem{remark}[theorem]{Remark}

\newtheorem{assumption}[theorem]{Assumption}
\newtheorem{claim}[theorem]{Claim}

\newcommand{\R}{{\mathbb R}}

\newcommand{\N}{{\mathbb N}}

\def\im{{\rm i}}

\newcommand{\C}{\mathbb{C}}

\newcommand{\sign}{{\mathrm{sign}}}
\newcommand{\even}{{\mathrm{even}}}
\newcommand{\odd}{{\mathrm{odd}}}
\newcommand{\sech}{{\mathrm{sech}}}
\def\({\left(}
\def\){\right)}
\def\<{\left\langle}
\def\>{\right\rangle}

\numberwithin{equation}{section}

\begin{document}

\title{Asymptotic stability of kink with internal modes \\under odd perturbation}

\author{Scipio Cuccagna, Masaya Maeda}
\maketitle

\begin{abstract}
	We give a sufficient condition, in the spirit of Kowalczyk-Martel-Munoz-Van Den Bosch \cite{KMMvdB21AnnPDE}, for the local asymptotic stability of kinks under odd perturbations.
	In particular, we  allow the existence of quite general configuration of  internal modes.
	The extension  of our result   to   moving kinks remains an open problem.
\end{abstract}

\section{Introduction}
In this paper, we consider  the problem of the asymptotic stability of kink solutions of the  $(1+1)$ dimensional nonlinear scalar field model
\begin{equation}\label{nwe1}
 \Box u_1 +  W'(u_1) =0      ,\quad (t,x)\in \R^{1+1} \text{  ,   where }  \Box = \partial_t ^2     -\partial _x ^2.
\end{equation}
We can write the above problem as
\begin{align}\label{nlkg}
\partial_t \begin{pmatrix}
u_1\\ u_2
\end{pmatrix}
=
\mathbf{J}
\begin{pmatrix}
-\partial_x^2 u_1 + W'(u_1)\\ u_2
\end{pmatrix},\ u_1,u_2:\R^{1+1}\to\R, \text{ where }\mathbf{J}:=\begin{pmatrix}
0 & 1 \\ -1 & 0
\end{pmatrix}.
\end{align}
Our    nonlinear potential   $W$  is an even $C^\infty$ function  such that
\begin{align}\label{ass:W}
\exists\zeta>0\ \mathrm{s.t.}\ W(\zeta)=W'(\zeta)=0,\ \omega^2:=W''(\zeta)>0\ \text{and}\  \forall h\in (-\zeta,\zeta),\ W(h)>0.
\end{align}
Under   assumption \eqref{ass:W}, it is well known that an odd kink solution exists, see Lemma 1.1 of Kowalczyk et al. \cite{KMMvdB21AnnPDE}.
\begin{proposition}\label{prop:kink}
There exists odd
$H\in C^\infty(\R)$ satisfying
$H''=W'(H)$.
Furthermore, we have
$H'(x)>0$, $\lim_{x\to \infty}H(x)=\zeta$, $|H(x)-\zeta|\lesssim e^{-\omega  |x|}$
and
\begin{align*}
\forall k\geq 1,\ |H^{(k)}(x)|\lesssim_k e^{-\omega  |x|}.
\end{align*}
\end{proposition}\qed

\begin{remark}
	By $A\lesssim B$, we mean that there exists $C>0$ s.t.\ $A\leq B$.
	The implicit constant $C$ is independent of important parameters (e.g.\ in the claim of Proposition \ref{prop:kink}, the implicit constants are independent of $x$ but depends on $k$).
\end{remark}

The purpose of this paper is to study the  case when  the kink  has internal modes, but only in the context of odd solutions of \eqref{nlkg}.
We set $\mathbf{H}=(H,0)$.  We   denote by $\boldsymbol{\Phi} [  \mathbf{z} ]  $ the \textit{refined profile},  introduced later in Sect. \ref{sec:refprof}, where
\begin{align}\label{eq:discr_mores}
		 \mathbf{z}= (z_1,...,z_{ {N}}  ),
	\end{align}
encodes the discrete modes and where $\boldsymbol{\Phi} [  \mathbf{0} ]= \mathbf{H}$. In analogy to Kowalczyk et al. \cite{KMMvdB21AnnPDE} we set \begin{align}&\label{def:energy_space}
\mathbf{E}_{\odd} =\{    \mathbf{u} \in L ^{1}_{loc}(\R, \R^2): \quad  u_1'\in  L ^{2}_{\even} (\R  ) , \quad  u_2\in  L ^{2}_{\odd} (\R  ) , \quad  \sqrt{W(u_1)}  \in L ^{2}_{\even} (\R  )\}  \text{  and}\\&
\mathbf{E}_{\mathbf{H}} =\{    \mathbf{u} \in \mathbf{E}_{\odd} : \quad  u_1'\in  L ^{2} _{\even} (\R  ) , \quad  u_2 \in  L ^{2}_{\odd}  (\R  ) , \quad  u_1-H\in   L ^{2}_{\odd}  (\R  )\} .\label{eq:restr_energy_sp}
\end{align}
For any $ \mathbf{u}$, there is a natural identification,  a natural \textit{trivialization} in fact,  of the tangent space
\begin{align}&
 T _{\mathbf{u}}\mathbf{E}_{\mathbf{H}} =  \boldsymbol{\mathcal{H}}^{1} \text{ where }  \boldsymbol{\mathcal{H}}^{s}:= H ^{s}_{\odd}(\R, \R) \times H ^{s-1}_{\odd}(\R, \R) .\label{eq:restr_energy_sptang}
\end{align}
There is a   natural distance   in $\mathbf{E}_{\mathbf{H}}$, given by $
  \|   \mathbf{u} -  \mathbf{v}    \| _{ \boldsymbol{\mathcal{H}}^{1}}     .$
Our main result is the following.
\begin{theorem}\label{thm:main}
	Under Assumptions  \ref{ass:repuls}, \ref{ass:generic} and \ref{ass:FGR} given below, for any $\epsilon>0$ and $a>0$, there exists a  $\delta _0>0$    s.t.,  for all odd functions
\begin{align}\label{eq:thm:main0}
		 \mathbf{u}(0)\in \mathbf{E}_{\mathbf{H}}\ \text{satisfying}\ \delta :=\| \mathbf{u}(0)-\mathbf{H}\|_{ \boldsymbol{\mathcal{H}}^{1}}<\delta _0
	\end{align}
and  for the corresponding solution  $\mathbf{u}$ of \eqref{nlkg},  we have
 \begin{align}
 \mathbf{u}(t) =     {\boldsymbol{\Phi}}[ \mathbf{z}(t)]   + \boldsymbol{\eta }    (t) \text{ for appropriate $\mathbf{z}\in C (\R , \C ^{N})$ and $ \boldsymbol{\eta }\in C (\R , \boldsymbol{\mathcal{H}}^{1})$,}
 \label{eq:main1}
\end{align}
and , for  $I=\R$ and  $\<x\> := (1+x^2)^{1/2}$,
\begin{align}
 \int _{I }  \|  e^{-a\<x\>}  \boldsymbol{\eta }  (t) \| ^2 _{\boldsymbol{\mathcal{H}}^{1} (\R )}    \le \epsilon,     \label{eq:main2}
\end{align}
   and
\begin{align}&
\lim_{t\to  \infty} \mathbf{z} (t) =0     \text{  .  }   \label{eq:main3}
\end{align}

\end{theorem}

In Sect.\ \ref{sec:phi8} we show that the $\phi ^8$ model, when near the $\phi ^4$ model, satisfies our repulsivity hypothesis.

\begin{remark}
	\begin{itemize}
		\item
		The local well-posedness of \eqref{nlkg} in $\mathbf{E}_{\mathbf{H}}$ in a small neighborhood of the kink is known, see section 3.1 of Kowalczyk et al. \cite{KMMvdB21AnnPDE}.
		\item
		Under the assumption that $\mathbf{u}(0)$ is odd,  then $\mathbf{u}(t)$ is odd for all $t$, by uniqueness.
		For general perturbations and also for  boosted kinks, Kowalczyk et al.\cite{KMMvdB21AnnPDE} gave a sufficient repulsivity  hypothesis for   asymptotic stability of the kinks.
		The  repulsivity  hypothesis in Kowalczyk et al.\cite{KMMvdB21AnnPDE} implies  the absence of internal modes.
		The purpose of this paper is to give  a somewhat related   repulsivity  hypothesis in Assumption \ref{ass:repuls}, which  allows the existence of internal modes. Unfortunately, at this time  we are
unable to treat moving kinks, which were  the main focus of Kowalczyk et al. \cite{KMMvdB21AnnPDE}.
\item The restriction to odd solutions  can  simplify considerably dispersion problems, as for example is shown in Kowalczyk et al.\cite{KMM1}. In Germain  et al. \cite{GP20,GPZ},  restricting to a  smaller class of odd solutions   of \eqref{nlkg}  and assuming the absence of internal modes, but avoiding any explicit repulsivity   hypothesis,     there is a different kind of proof of asymptotic stability, involving rates of decay of remainder terms.  Such  rates of decay  cannot hold with data   like in   \eqref{eq:thm:main0}, where the estimates and the asymptotic behavior  on the solutions need to be invariant by time translation, like in \eqref{eq:main2}--\eqref{eq:main3}. Notice, also, that  the literature which uses dispersive estimates such as also \cite{DM20,LP} as many others, has   been able so far to  treat only rather simple configurations of internal modes, usually at most a single  one, usually nicely    positioned.
    \item Just before completion of this work we learned about   Kowalczyk and Martel \cite{KM22}, which   in the case of a favorably   placed internal mode, encompassing the $\phi ^4$ model, simplifies  and generalizes the proof in \cite{KMM2}. Our paper is independent from Kowalczyk and Martel \cite{KM22}.
        Kowalczyk and Martel \cite{KM22} make a more efficient use of dispersivity, adapting this notion to the context of odd solutions of our problem.

        By and large, it should be possible to combine \cite{KM22} with our framework   in  the  context of more general eigenvalues. Somewhat  delicate, at least in our context, should be the case when $L_D$ has a resonance at $\omega ^2$, as for example when $L_D=-\partial ^2_x+\omega ^2$, for instance in the case of the $\phi ^4$ problem. In that case, our proof of Lemma \ref{lem:LAP} does not work, and perhaps the statement is wrong. On the other hand, this lemma for us is important     when we bound the auxiliary variable $\mathbf{g}$ in Sect. \ref{sec:smoothing}, needed in our proof of the Fermi Golden Rule.

        Our proof of the Fermi Golden Rule is different  from Kowalczyk and Martel \cite{KM22}. From a combinatorial view point, we are dealing with a more complicated problem. In our   framework, we need an expansion on the transformed variable $\mathbf{v}$, see   \eqref{eq:expan_v}, rather than the original variable   $\mathbf{u}$,  to take advantage of   the cutoff  factor $\chi _{B^2}$ in front of the nonlinear terms in the equation of $\mathbf{g}$. The cutoff offsets the long range nature of the  nonlinearity.
         Unfortunately, the commutator of cutoff and linear  part of the equation,  generates an extra term, for example in \eqref{eq:equation_g13}, which is delicate. Kowalczyk and Martel \cite{KM22} study the Fermi Golden Rule in the initial variable $\mathbf{u}$, so they do not have this commutator.   We  are neither able  to generalize their FGR argument     nor, exactly because of the long range nonlinearity,    to perform smoothing estimates directly  in the original variable  $\mathbf{u}$.

        \item Early work on kinks by Komech and Kopylova \cite{KomKop111,KomKop112}, treated special cases with  short range nonlinearities and within the framework of Buslaev and Perelman \cite{BP2}. We utilize  Komech and Kopylova \cite{KomKop10,KomKop111} when we consider some smoothing estimates in Sections \ref{sec:Preliminary} and \ref{sec:smoothing}.

            \item For a treatment of
            the sine--Gordon model, using he  Nonlinear Steepest Descent method and   techniques of Integrable Systems, see \cite{Chen} and therein.
	\end{itemize}
\end{remark}
This paper is very similar to   our previous \cite{CM2109.08108},  which in turn used    the dispersion theory of Kowalczyk et al.\ \cite{KMM3,KMMvdB21AnnPDE}  (see also \cite{LiLu} for a very recent paper related to \cite{KMM3})  combined with our theory on the Fermi Golden Rule (FGR).  In \cite{CM2109.08108} and in a number of other papers, like  \cite{CM2111.02681}, we have produced a framework very  effective  in sorting out effortlessly  the complexities  of the transient  patterns of the internal discrete modes that can occur  in stability problems, and exploiting dissipation induced on the discrete modes, due to their hemorrhaging energy which, by nonlinear interaction,  leaks in the radiation component  where it escapes to infinity. Since Bambusi and Cuccagna \cite{BC11AJM}, which can be considered the first paper in our series, we avoided the  decay rates analysis of the early PDE papers on this subject,  the earliest involving internal modes being Buslaev and Perelman\cite{BP2}.  Notice that  the  rarely cited  paper \cite{CMP2016} extends considerably the result in  \cite{BC11AJM}.  As mentioned above, decay rates cannot exist in the Energy space.  Kowalczyk et al. \cite{KMM2,KMM3,KMMvdB21AnnPDE} as well as  KdV papers by Martel and Merle  such as  \cite{MM05,MM081,MM082}  work in energy space, presumably to achieve a maximum of generality. In fact,  in the presence of discrete modes, the Energy Space framework tends to be   conductive to  rather simple sorting out of the discrete modes.  Early in the literature,   for example in \cite{BP2,SW3} or   Komech and Kopylova \cite{KomKop111,KomKop112}, as well as in many others such as, for example \cite{AS20DCDS}, there was use of dispersive estimates. This is also related to the fact that  the earliest papers   predate Keel and Tao  \cite{Kl-Tao}, whose endpoint Strichartz estimate has played an important role in the theory, even though  it can be replaced by smoothing estimates.
 Ultimately, the literature using dispersive estimates,  so far  has not dealt with discrete mode configurations which  are    not simple.

In the presence of short range nonlinearities, where it is possible to prove dispersion using Strichartz estimates, papers such  as \cite{CM2111.02681} provide proofs of asymptotic stability and scattering in the presence of very general discrete modes configurations. In the case of possibly  long range nonlinearities, like here and  \cite{CM2109.08108}, we use the Virial Inequalities framework originating in Merle's school,  in the  particular elaboration of  Kowalczyk et al. We need that the linearization $L_1$ be, not directly dispersive, but, rather,   dispersive after a sufficient number of  Darboux transformations, and not just   a  single one  like  in  Kowalczyk et al. \cite{KMM2} or Martel \cite{Martel2110.01492}. Darboux transformations are beautifully discussed by Deift and Trubowitz    \cite {DT79CPAM}, although their theory would be not sufficient in the context of more general kinks than the ones discussed here.

Kowalczyk et al.  \cite{KMM2} have been able to prove an asymptotic stability result in the absence of dissipative operators, see also   \cite{Snelson}.    Furthermore, as we remarked above,  Kowalczyk and Martel \cite{KM22} gave a new more general proof. We refer also to \cite{Snelson,AS1,AS2}.

  In the context of the dispersion theory   of papers such as  \cite{LS1}--\cite{LS4} and \cite{Sterbenz},
  and the framework in \cite{DM20,GP20,LuSchlag,GPZ,LP}, if it works in the absence of dissipative potentials, can obviously prove very useful.  A natural problem would be to prove some form of scattering of the remainders for solutions in Energy space.

\subsection{Internal modes, Darboux transform and repulsivity assumption}
\label{sec:intmodes}

We consider the  Schr\"odinger operator
\begin{align}\label{eq:L1}
L_1=-\partial_x^2 + W''(H).
\end{align}
By differentiating $H''=W'(H)$, we obtain $H'\in \mathrm{ker}L_1$.
Since $H'>0$, we have $\mathrm{ker}L_1=\mathrm{span}\{H'\}$.
By Proposition \ref{prop:kink}, $W''(H)-\omega^2$ decays exponentially.
Thus, $L_1$ will have at most finitely many eigenvalues, which, since we are in 1D,  are all simple.
We label the eigenvalues corresponding to the odd eigenfunctions as follows:
\begin{itemize}
\item $\sigma_{\mathrm{d}}(\left.L_1\right|_{L_{\mathrm{odd}}^2})=\{\lambda_j^2\ |\ j=1,\cdots, N\}$ with
$0<\lambda_1<\cdots<\lambda_N <\omega$.
\end{itemize}
We set $\phi_j \in L_{\mathrm{odd}}^2$ to be the corresponding (odd, normalized and $\R$-valued) eigenfunctions, i.e. $L_1\phi_j=\lambda_j^2\phi_j$ and $\|\phi_j\|_{L^2}^2=\(2\lambda_j\)^{-1}$.
\begin{remark}\label{rem:repul}
	The repulsivity  condition in \cite{KMMvdB21AnnPDE} implies $\sigma_{\mathrm{p}}\(    L_1    \)=\{0\}$ and, therefore, $N=0$.
\end{remark}
In the following, we assume $N\geq 1$.
The case $N=0$ is contained in \cite{KMMvdB21AnnPDE}.

By the Sturm-Liouville theory $L_1$ will have a number $\widetilde{N} $, equal to      $2N$ or $2N+1$, of eigenvalues.
We  consider  $\widetilde{\lambda}_j>0$ so that we have $\sigma_{\mathrm{d}}(L_1)=\{\widetilde{\lambda}_j^2\}_{j=1}^{\widetilde{N}}$.
In this case we have $\widetilde{\lambda}_1=0$ and $\lambda_j=\widetilde{\lambda}_{2j}$.

We set
\begin{align}\label{eq:Lineariz1}
\mathbf{L}_1:=
\mathbf{J}
\begin{pmatrix}
L_1 & 0 \\ 0 & 1
\end{pmatrix}
=
\begin{pmatrix}
0 & 1\\ -L_1 & 0
\end{pmatrix}
\ \text{ and }\ \boldsymbol{\Phi}_j:=\begin{pmatrix}
\phi_j\\
-\im \lambda_j \phi_j
\end{pmatrix}.
\end{align}
This operator  $\mathbf{L}_1$ is relevant here  because it is obtained linearizing \eqref{nlkg}      at $ \mathbf{H} $.
Indeed, substituting $\mathbf{u}=\mathbf{H}+\mathbf{r}$ into \eqref{nlkg}, we   have
\begin{align*}
	\partial_t \mathbf{r} = \mathbf{L}_1 \mathbf{r} + O(\mathbf{r}^2).
\end{align*}
From now on, we will consider only odd functions. In particular $\mathbf{L}_1$ will act only on odd in $x$ functions.

\noindent By direct computation, we see that
\begin{align}\label{eq:eigfunL1}
\mathbf{L}_1\boldsymbol{\Phi}_j=-\im \lambda_j \boldsymbol{\Phi}_j\ \text{and}\  \mathbf{L}_1\overline{\boldsymbol{\Phi}_j}=\im \lambda_j \overline{\boldsymbol{\Phi}_j}.
\end{align}
We consider
 \begin{align}\label{eq:inner0} &
     \(
 \mathbf{f},  \mathbf{ g}\)  =\int _{\R} {^t\mathbf{f}}(x) \mathbf{ g}(x) dx ,\\&  \label{eq:inner1}    \<
 \mathbf{f},  \mathbf{ g}\> =\Re   \(
 \mathbf{f},  \overline{\mathbf{ g}}\)
\end{align}
and the   symplectic form \begin{align}\label{eq:symplect0}
 \< \mathbf{J }
 \mathbf{f},   \mathbf{ g}\>   .
\end{align}
Notice that $ \< \mathbf{J }
 \boldsymbol{\Phi}_j,   \overline{\im \boldsymbol{\Phi}}_j\> =1$.

\noindent It is easy to check
\begin{align}\label{eq:specL1}
 \sigma_{\mathrm{d}}( \mathbf{L}_ 1)
=\{\pm \im \lambda_j\ |\ j=1,\cdots, N\}
\ \text{and}\ \sigma_{\mathrm{ess}}(\mathbf{L}_ 1)=\im\( (-\infty,-\omega]\cup [\omega,\infty)  \).
\end{align}
Notice also that $\mathbf{L}_1$  leaves the following decomposition invariant,  \begin{align}& \label{eq:Linz2eig13}   L^2_{\odd}(\R , \C ^2 ) =   L^2_{discr} \oplus  L^2_{disp}  \text{  where }L^2_{discr}:= \oplus _{\lambda \in \sigma _p (\mathbf{L} _{1}) }  \ker \( \mathbf{L} _{1} -  \lambda \) ,
   \end{align}
   where $L^2_{disp} $   is   the $   \< \mathbf{J} \cdot , \cdot \>$--orthogonal  of $L^2_{discr}$.

 Thus, the linearized operator $\mathbf{L}_1$ has neutral eigenvalues, which will create   oscillating and non-decaying solutions in the linear level.
 Such oscillations will last for long time in the full nonlinear problem, they will loose energy  and  oscillations will eventually decay.
 The Fermi Golden Rule (FGR) non-degeneracy condition, which will be introduce in the next subsection, guarantees such  phenomenon, but it has to be combined  with dispersion of the continuous modes.  To prove dispersion  we use  virial estimates of  Kowalczyk et al.  \cite{KMMvdB21AnnPDE}. For this  we need to  assume that the potential $W''(H)$ is ``repulsive" after     a series of Darboux transforms   which eliminate the eigenvalues, as we explain now.
  The discussion is similar to  \cite{CM2109.08108}, which was based on \cite{DT79CPAM}.

\subsubsection{Darboux Transformations}\label{sec:darboux}
\noindent We inductively define the Schr\"odinger operator $L_j$ $(j=1,\cdots,\widetilde{N}+1)$ and a differential operator $A_j$ ($j=1,\cdots,\widetilde{N}$) as follows.
\begin{enumerate}
\item $L_1=-\partial_x^2+W''(H)$ and $A_1=(H')^{-1}\partial_x\(H' \cdot \)$.
In this case, we have
\begin{align}\label{eq:dar1}
L_1= A_1 A_1^* ,
\end{align}
and we define $L_2$ by
\begin{align*}
L_2:=A_1^*A_1 .
\end{align*}
\item
Inductively, given $L_k$ with  $\psi_k$   the   ground state of $L_k$,  we set $A_k:=\psi_k^{-1}\partial_x\(\psi_k \cdot \)$.
Then
\begin{align}\label{eq:dar3}
L_k=A_k A_k^*-\widetilde{\lambda}_k^2
\end{align}
and we define
\begin{align*}
L_{k+1}:=A_k^* A_k - \widetilde{\lambda}_k^2
\end{align*}
\item
In the last step,   $L_{\widetilde{N}+1}:=A_{\widetilde{N}}^* A_{\widetilde{N}}- \widetilde{\lambda}_{\widetilde{N}}^2$.
We set
\begin{align}\label{eq:dar5}
L_{D}=L_{\widetilde{N}+1}=-\partial_x^2 + V_D \text{  where, here, } V_D-\omega ^2\in \mathcal{S}(\R,\R).
\end{align}
\end{enumerate}

For the above we refer to Section 3 of \cite{DT79CPAM} and Proposition 1.9 of \cite{CM2109.08108}.
We set
\begin{align}\label{def:A}
\mathcal{A}:=A_1\cdots A_{\widetilde{N} }.
\end{align}
Then,  by simple computation we obtain.
\begin{align}\label{eq:DarConj2}
\mathcal{A}^* L_1 =  L_D  \mathcal{A}^*.
\end{align}


\noindent We  assume that    $V_D$ is repulsive, in the following sense:
\begin{assumption}\label{ass:repuls}
 $x V_D'(x)\leq 0$ for all $x\in \R$ and $V_D$ is not identically zero.
\end{assumption}

\begin{remark}\label{rem:ass:repuls}
  In Kowalczyk and Martel \cite{KM22}  the above assumption is eased into the following:  there exists a $\gamma >0$ such that the operator $-(1-\gamma ) \partial ^2 _x - 2^{-1/2}x V_D'(x)$  has at most one negative eigenvalue.  In order to prove their result, Kowalczyk and Martel \cite{KM22} modify the first virial inequality of \cite{KMM3}. This could  be arranged here as well, but there is an issue that we face, and which we describe in Remark \ref{rem:ass:repuls1} below.
  \end{remark}

\subsection{Refined profile and Fermi Golden Rule assumption}\label{sec:refprof}

As in the asymptotic stability of solitons for nonlinear Schr\"odinger equations \cite{CM2109.08108}, we introduce the notion of \textit{refined profile}.

We introduce some notation.
For $\mathbf{m}=(\mathbf{m}_+,\mathbf{m}_-)\in \N_0^{2N}$, where $\N_0:=\N\cup\{0\}$, we write $\overline{\mathbf{m}}=(\mathbf{m}_-,\mathbf{m}_+)$ and $|\mathbf{m}|=\sum_{j=1}^N(m_{+j}+m_{-j})$.
We set $\mathbf{e}^j=(\delta_{j1},\cdots,\delta_{jN},0,\cdots,0)$.
We set
\begin{align}
\boldsymbol{\lambda}:=(\lambda_1,\cdots,\lambda_N,-\lambda_1,\cdots,-\lambda_N),
\end{align}
and
\begin{align}\label{eq:ldotm}
\boldsymbol{\lambda}\cdot \mathbf{m}:=\sum_{j=1}^N \lambda_j(m_{+j}-m_{-j}).
\end{align}
We assume the following.
\begin{assumption}\label{ass:generic}For  $M  $ be the largest number in $\N$ such that $(M -1)\lambda _1< \omega$, then for a multi--index $ \mathbf{m} \in \N_0^{2N}$
\begin{align}&  \|  \mathbf{m}  \|  \le M   \Longrightarrow      \(   \mathbf{m} \cdot \boldsymbol{\lambda} \) ^2 \neq \omega   . \label{eq:generic12}
\end{align}
We also assume that    for $\mathbf{m}=( \mathbf{m}_+,  \mathbf{m}_-) \in \N_0^{2N}$ then
\begin{align}&  \| \mathbf{m} \|  \le 2M \text{  and    }  \mathbf{m}\cdot \boldsymbol{\lambda} = 0 \Longrightarrow       \mathbf{m}_+ =  \mathbf{m}_- . \label{eq:generic11}
\end{align}
\end{assumption}
As in \cite{CM2109.08108}, we  set
\begin{align*}
\mathbf{R}&:=\{\mathbf{m}\in \N_0^{2N}\ |\ |\boldsymbol{\lambda}\cdot \mathbf{m}|>\omega\},\\
\mathbf{R}_{\mathrm{min}}&:=\{\mathbf{m}\in \mathbf{R}\ |\  \not \exists \mathbf{n}\in \mathbf{R}\ \mathrm{s.t.}\ \mathbf{n}\prec \mathbf{m}\},\\
\mathbf{I}&:=\{\mathbf{m}\in \N_0^{2N}\ |\ \exists \mathbf{n}\in \mathbf{R}_{\mathrm{min}},\ \mathbf{n}\prec \mathbf{m}\}\\
\mathbf{NR}&:=\N_0^{2N}\setminus (\mathbf{I}\cup \mathbf{R}_{\mathrm{min}}),\\
\boldsymbol{\Lambda}_j&:=\{\mathbf{m} \in \mathbf{NR}\ |\ \boldsymbol{\lambda}\cdot \mathbf{m}=\lambda_j\} \\  \boldsymbol{\Lambda}_0&:=\{\mathbf{m} \in \mathbf{NR}\ |\ \boldsymbol{\lambda}\cdot \mathbf{m}=0\} ,
\end{align*}
where the partial order $\prec$ is defined by
\begin{align*}
\mathbf{n}\prec \mathbf{m}\ \Leftrightarrow\ \forall j, n_{+j}+n_{-j}\leq m_{+j}+m_{-j}\ \text{and}\ |\mathbf{n}|<|\mathbf{m}|.
\end{align*}
 \begin{lemma}\label{lem:combinat1}   The following facts hold.
 \begin{enumerate}
   \item  If $| \mathbf{m} |> M$, with $M$ the constant in Assumption (H2),   then  $\mathbf{m}\in \mathbf{I}$.
   \item  $\mathbf{R}_{\mathrm{min}}$ and  $\mathbf{NR}$ are finite sets.
   \item If $\mathbf{m}\in \mathbf{NR}$, then $|\boldsymbol{\lambda}\cdot \mathbf{m}|<\omega$ and if $\mathbf{m}\in \mathbf{R}_{\mathrm{min}}$, then $\mathbf{m}_+=0$ or $\mathbf{m}_-=0$.
   \item  If $\mathbf{m}\in \boldsymbol{\Lambda}_{j}$ then  there is   a $\mathbf{n} \in \boldsymbol{\Lambda} _{0}$ with $\mathbf{m}= \mathbf{e}^j+\mathbf{n}$.
\end{enumerate}
 \end{lemma}
\proof If
  $| \mathbf{m} |> M$,  we can write $\mathbf{m} =\boldsymbol{\alpha  } +\boldsymbol{ \beta}  $  with
   $|\boldsymbol{\alpha  } | = M $.  If $   \boldsymbol{\alpha  }=(\boldsymbol{\alpha  }_+,\boldsymbol{\alpha  }_-)$ and if we set $   \mathbf{n}=(\mathbf{n}_+,\mathbf{n}_-)$  with $\mathbf{n}_+= \boldsymbol{\alpha  }_++\boldsymbol{\alpha  }_-$ and
   $\mathbf{n}_-= 0$, then $\mathbf{n} \cdot \boldsymbol{\lambda} \ge M \lambda _1 >  {\omega}   $. This implies that $ \mathbf{{n  } }\in    \mathbf{R} $ and that there exists $\boldsymbol{\mathfrak{ a} }\in \mathbf{R}_{\mathrm{min} }$  with $\boldsymbol{\mathfrak{ a} } \preceq  \mathbf{n}$. From  $| \boldsymbol{\beta  }| \ge 1 $ it follows that   $\boldsymbol{\mathfrak{ a} } \prec \mathbf{m}$ and so $\mathbf{m}\in \mathbf{I}$.

   Obviously, from the 1st claim it follows that if  $\mathbf{m}\in  \mathbf{R}_{\mathrm{min}}\cup \mathbf{NR}$ then $| \mathbf{m} |\le  M$. Next we observe that $\mathbf{m}\in \mathbf{NR}$ implies $| \mathbf{m} |\le  M$ and $|\boldsymbol{\lambda}\cdot \mathbf{m}| \le \omega$ and, by Assumption \ref{ass:generic}, $|\boldsymbol{\lambda}\cdot \mathbf{m}| < \omega$.
    If $\mathbf{m}\in \mathbf{R}_{\mathrm{min}}$ with, say, $\mathbf{m} \cdot \boldsymbol{\lambda} >\omega$, then obviously from \eqref{eq:ldotm} we have
 $\mathbf{m} _+ \cdot \boldsymbol{\lambda} >\omega$ and it is elementary that $\mathbf{m} =(\mathbf{m} _+,0)$. Finally,  from the first claim we know that if $\mathbf{m}\in \boldsymbol{\Lambda}_{j}$ then   $\| \mathbf{m} \| \le  M$. From $  \mathbf{m}   \cdot \boldsymbol{\lambda  } -\lambda _j =0$ it follows from \eqref{eq:generic11}  that  we have the last claim.\qed

For $\mathbf{z}\in \C^N$ and $\mathbf{m}\in \N_0^2$, we write
 $\mathbf{z}^{\mathbf{m}}=\prod_{j=1}^N z_j^{m_{+j}}\overline{z_j}^{m_{-j}}
 $.

 For $f\in C^1(\C^N,X)$ (differentiability is taken in the real sense), we set
 $D_{\mathbf{z}}f(\mathbf{z})\mathbf{w}:=\frac{d}{d\epsilon}f(\mathbf{z}+\epsilon\mathbf{w})$.

 \begin{definition}\label{def:spaces}   We set   $\|\cdot\|_{\Sigma^s}:=\| \cdot \|_{H^s_{a_1}}:=\| e^{a_1\<x\>} \cdot\|_{H^s}$ where $a_1 = \frac{1}{2}\sqrt{\omega^2 - \lambda_N^2}$ and denote by $\Sigma^s$  the corresponding spaces. We write $\boldsymbol{\mathcal{H}}^s=\Sigma^s\times \Sigma^s$

 \noindent For $ b\in \R$ we write $\|\cdot\|_{ L ^2_{b}}:= \| e^{b\<x\>} \cdot\|_{L^2}$

 \noindent We write $\Sigma :=\Sigma ^{1}$ and denote by $\Sigma ^*$  its dual.

\noindent   For any $s,\sigma\in \R$, we will use also other weighted spaces, defined by the norm   $\|\cdot\|_{ L^{2,\sigma} }:= \|  \<x\> ^\sigma \cdot\|_{L^2}$ and spaces defined by the norm  $\|\cdot\|_{ \boldsymbol{\mathcal{H}}^{s,\sigma}}:= \|  \<x\> ^\sigma \cdot\|_{ \boldsymbol{\mathcal{H}}^{s}}$.

 \noindent We pick $a\in (0, a_1  )$ and consider  the following norm,
 \begin{align}\label{eq:norm_rho}
 \|f\|_{\widetilde{\Sigma} }^2=  \<\(-\partial_x^2+ \sech ^2\( \frac{a  x}{10}   \)   \)  f,f\> \sim \|f\|_{\dot{H}^1}^2+\|f\|_{L^2_{-\frac{a}{10}}}^2,
 \end{align}
  denoting by $\widetilde{\Sigma}$  the corresponding space. For $\mathbf{f}=(f_1,f_2)$, we will consider the norm
\begin{align}\label{eq:norm_rhovec}
 \|\mathbf{f}\|_{\boldsymbol{\widetilde{\Sigma}} } =  \|f_1\|_{\widetilde{\Sigma} }+\|f_2\|_{L^2_{-\frac{a}{10}}} .
 \end{align}
\end{definition}

    We observe that $H^{(n)}, \phi_j \in \Sigma^s$ for arbitrary $n\geq 1$, $s\in \R$ and $j=1,\cdots N$.

 The refined profile is an approximate solution of \eqref{nlkg} which encodes the  kink with its internal modes.

\begin{proposition}\label{prop:rp}
There exist $\alpha _0>0$,  functions $\{\boldsymbol{\phi}_{\mathbf{m}}: \mathbf{m}\in \mathbf{NR}  \} \subset \Sigma ^{\infty}$, $\widetilde{\mathbf{z}} _{  R} \in C ^{\infty}( \mathcal{B} _{\C ^N} (0, \alpha _0), \C^N)$  and $\{\lambda_{j\mathbf{m}}\}_{\mathbf{m}\in \Lambda_j}\subset\R$ for $j=1,\cdots,N$ with $\boldsymbol{\phi}_0=\mathbf{H}$, $\boldsymbol{\phi}_{\mathbf{e}^j}=\boldsymbol{\Phi}_j$, $\boldsymbol{\phi}_{\overline{\mathbf{m}}}=\overline{\boldsymbol{\phi}_{\mathbf{m}}}$ and $\lambda_{j\mathbf{e}^j}=\lambda_j$ s.t.
setting
\begin{align}\label{eq:refprof1}
&\boldsymbol{\phi}[\mathbf{z}] :=\begin{pmatrix}
\phi_1[\mathbf{z}]\\ \phi_2[\mathbf{z}]
\end{pmatrix}=  \boldsymbol{\phi}_0 +  \widetilde{\boldsymbol{\phi}}[\mathbf{z}] :=  \boldsymbol{\phi}_0+    \sum_{\mathbf{m}\in \mathbf{NR}, | \mathbf{{m}}|\ge 1}\mathbf{z}^{\mathbf{m}}\boldsymbol{\phi}_{\mathbf{m}},\\ &   \boldsymbol{\phi}_{\overline{\mathbf{m}}}=\overline{\boldsymbol{\phi}_{\mathbf{m}}}
    \label{eq:symmrp} \\
\label{eq:tildez} &\widetilde{z}_j :=-\im \sum_{\mathbf{m}\in \Lambda_j}\lambda_{j\mathbf{m}}\mathbf{z}^{\mathbf{m}} +\widetilde{z}_{jR}    ,\ \widetilde{\mathbf{z}}:=(\widetilde{z}_1,\cdots,\widetilde{z}_N) \text{  and }\widetilde{\mathbf{z}}_{R}=(\widetilde{z}_{1R},\cdots,\widetilde{z}_{NR}),\\& \label{eq:boundzjR}
|\widetilde{\mathbf{z}} _{  R} | \lesssim   \sum_{\mathbf{m}\in \mathbf{ R}_{\min} } |\mathbf{z}^{\mathbf{m}} |,  \\ & \boldsymbol{\lambda} _{ \overline{\mathbf{m}}} =  \boldsymbol{\lambda} _{ \mathbf{m}} \in \R ^{2N}\label{eq:lambdan}
\end{align}
where $\boldsymbol{\lambda} _{ \mathbf{m}} :=(\lambda_{1\mathbf{m}},..., \lambda_{N\mathbf{m}},-\lambda_{1\mathbf{m}},..., -\lambda_{N\mathbf{m}} ) $,
and
\begin{align}\label{eq:RF}
\boldsymbol{ \mathcal{R}}  [\mathbf{z}]:=\mathbf{J} \begin{pmatrix}
-\partial_x^2 \phi_1[\mathbf{z}]+W'(\phi_1[\mathbf{z}])\\ \phi_2[\mathbf{z}]
\end{pmatrix}
-D_{\mathbf{z}}\boldsymbol{\phi}[\mathbf{z}]\widetilde{\mathbf{z}} -D_{\mathbf{z}}\boldsymbol{\phi}[\mathbf{z}]\widetilde{\mathbf{z}}_{R},
\end{align}
the remainder function $\boldsymbol{ \mathcal{R}}[\mathbf{z}]$ can be expanded as
\begin{align}\label{eq:RFRemainder}
\boldsymbol{ \mathcal{R}}[\mathbf{z}]=\sum_{\mathbf{m}\in \mathbf{R}_{\mathrm{min}}}\mathbf{z}^{\mathbf{m}} \mathcal{R}_{\mathbf{m}}+\boldsymbol{ \mathcal{R}}_1[\mathbf{z}],
\end{align}
with   $ \mathcal{R}_{\overline{\mathbf{m}}}= \overline{\mathcal{R}_{\mathbf{m}}} \in \boldsymbol{\Sigma}^{\infty}$ and for any $l\in \N$
\begin{align}\label{eq:R1FRemainder}
\|\boldsymbol{ \mathcal{R}}_1[\mathbf{z}]\|_{\boldsymbol{\Sigma} ^{l}}\lesssim _{l }|\mathbf{z}| \sum_{\mathbf{m}\in \mathbf{R}_{\mathrm{min}}} |\mathbf{z}^{\mathbf{m}}|.
\end{align}
Furthermore,  \begin{align}\label{eq:R1FRemainder--}\< \mathbf{J} \boldsymbol{ \mathcal{R}}[\mathbf{z}],D_{\mathbf{z}}\boldsymbol{\phi}[\mathbf{z}] \boldsymbol{\zeta}\>=0  \text{ for any $\boldsymbol{\zeta}\in \C^N$}.\end{align}
\end{proposition}
\proof We insert \eqref{eq:refprof1} in \eqref{eq:RF}, using \eqref{eq:tildez}. We expand
\begin{align*} &  W' ( H  + \widetilde{{\phi  }}_1 [ \mathbf{z} ] )=  W' ( H   )+ W ^{\prime\prime} ( H  )  \widetilde{{\phi  }}_1 [  \mathbf{z}]   +  \sum_{\ell =2}^M\frac{W^{(1+\ell )}(H  )}{\ell !} \widetilde{{\phi  }}_1 ^{\ell}[  \mathbf{z}] +  O(\|  \mathbf{ z} \| ^{M+1}),
\end{align*}
where $\widetilde{\boldsymbol{\phi}}[\mathbf{z}] = ( \widetilde{{\phi  }}_1  [  \mathbf{z}] ,\widetilde{{\phi  }}_2 [  \mathbf{z}]  )$. Then, for $\overrightarrow{\mathbf{j}} ={ ^t(0,1)},$
\begin{align*} &     \sum_{\ell =2}^M\frac{W^{(1+\ell )}(H )}{\ell !} \widetilde{{\phi  }}_1 ^{\ell}[   \mathbf{z}]   \overrightarrow{ \mathbf{j}} =  \sum_{\mathbf{m} \in \mathbf{NR}  }\mathbf{z} ^{\mathbf{m} }  \mathbf{g}_\mathbf{m}  + \sum_{\substack{\mathbf{m} \in \mathbf{R} \cup \mathbf{I} \\  |\mathbf{m}|\le M} }\mathbf{z} ^{\mathbf{m} }  \mathbf{g}_\mathbf{m}   +  O(\|  \mathbf{ z} \| ^{M+1})
\end{align*}
 where,   for   $ {\boldsymbol{\phi  }}  _{\mathbf{m}}   ={ ( {{\phi  }}_{1\mathbf{m}}  ,  {{\phi  }}_{2\mathbf{m}}  )}$,
\begin{align}\label{eq:gm}&
  \mathbf{g}_\mathbf{m}  =  \sum_{\ell =2}^M\frac{W^{(1+\ell )}(H  )}{\ell !} \sum_{\substack{\mathbf{m}^1,\cdots,\mathbf{m}^\ell \in \mathbf{NR}\\ \mathbf{m}^1+\cdots+\mathbf{m}^{\ell}=\mathbf{m}}} {\phi}_{1 \mathbf{m}^1} \cdots {\phi}_{1\mathbf{m}^\ell}    \overrightarrow{\mathbf{j}} .
\end{align}
Using   \begin{align}\label{eq:difzm}
\( D_{\mathbf{z}}\mathbf{z}^{\mathbf{m}}\)(\im \boldsymbol{\lambda}\mathbf{z})=\im \mathbf{m}\cdot \boldsymbol{\lambda} \  \mathbf{z}^{\mathbf{m}}, \text{ where $\boldsymbol{\lambda}\mathbf{z} := ( \lambda _1z_1,..., \lambda _Nz_N)$,}
\end{align}
we obtain
\begin{align*} &    D _{\mathbf{z}} \boldsymbol{\phi  } [ \mathbf{z}]     \widetilde{\mathbf{z}}[ \mathbf{z}] = - \im
\sum_{\mathbf{m} \in \mathbf{NR}  } \mathbf{m}\cdot  \boldsymbol{\lambda  }    \mathbf{z} ^{\mathbf{m} } \boldsymbol{\phi  }_{\mathbf{m} } - \im
\sum_{\substack{  \mathbf{m}  \in \mathbf{NR }, \ \mathbf{n} \in \boldsymbol{\Lambda  }_0   }}  \mathbf{m}\cdot \boldsymbol{\lambda  } _{\mathbf{n}} \mathbf{z}^{\mathbf{n}}    \mathbf{z}^{\mathbf{m}}\boldsymbol{\phi  }_{\mathbf{m} } - D _{\mathbf{z}} \boldsymbol{\phi  } [ \mathbf{z}] \im \widetilde{\mathbf{z}} _{R} .
\end{align*}
Let us set  \begin{align}\label{eq:RFtilde}
\widehat{\boldsymbol{ \mathcal{R}}}  [\mathbf{z}]:=\mathbf{J} \begin{pmatrix}
-\partial_x^2 \phi_1[\mathbf{z}]+W'(\phi_1[\mathbf{z}])\\ \phi_2[\mathbf{z}]
\end{pmatrix}
-D_{\mathbf{z}}\boldsymbol{\phi}[\mathbf{z}]( \widetilde{\mathbf{z}}- \widetilde{\mathbf{z}}_R ) .
\end{align}
We expand now to get
\begin{align}\label{eq:RFtildeexpan} &  \widehat{\boldsymbol{ \mathcal{R}}}[\mathbf{z}] =  \sum_{\mathbf{m} \in \mathbf{NR}  }  \mathbf{z} ^{\mathbf{m}} \widehat{\mathcal{{R}}}_{\mathbf{m}} +   \sum_{\substack{\mathbf{m} \in \mathbf{R} \cup \mathbf{I} \\  |\mathbf{m}|\le M} }\mathbf{z} ^{\mathbf{m} }\widehat{ \mathcal{{R}}}_{\mathbf{m}}   +  O(\|  \mathbf{ z} \| ^{M+1}),
\end{align}
where
\begin{align*} &   \widehat{\mathcal{R}}_{\mathbf{m}} =    \(  \mathbf{L} _{1}   +  {\im}  \boldsymbol{\lambda  } \cdot \mathbf{m}  \)    \boldsymbol{\phi  }_{\mathbf{m} } - \mathcal{E}_\mathbf{m} \text{  where }  \\&   \mathcal{E}_\mathbf{m} =\mathbf{g}_\mathbf{m}   -   \sum_{\substack{\mathbf{m}' +\mathbf{n}'=\mathbf{m}\\ \mathbf{m}' \in \mathbf{NR }, \ \mathbf{n}'\in \boldsymbol{\Lambda  }_0   }}      \im  \boldsymbol{\lambda  }_{\mathbf{n}'} \cdot \mathbf{m}'   \boldsymbol{\phi  }_{\mathbf{m}'   }            .
\end{align*}
We seek $\widehat{\mathcal{R}}_{\mathbf{m}} \equiv 0$   for $\mathbf{m} \in \mathbf{NR} $. For $|\mathbf{m}|=1$ the equation reduces to
$\(  \mathbf{L} _{1}    +\im \boldsymbol{\lambda  } \cdot \mathbf{m}  \)    \boldsymbol{\phi  }_{\mathbf{m} }  =0$, so that we can set $\boldsymbol{\phi  }_{\mathbf{e}^j } = \boldsymbol{\Phi  }_{ j }   $   and  $\boldsymbol{\phi  }_{\overline{\mathbf{e} }^j } = \overline{\boldsymbol{\Phi  }}_{ j }   $.  Let us consider now $|\mathbf{m}|\ge 2 $   with $\mathbf{m}\not \in  \cup _{j=1}^{N}\(    \boldsymbol{\Lambda  }_j\cup   \overline{\boldsymbol{\Lambda  }}_j \) $. In this case let us assume by induction that  $\boldsymbol{\phi  }_{\mathbf{m} '}$ and $ \boldsymbol{\lambda  }_{\mathbf{m} '}$
have been defined for $|\mathbf{m}' | <| \mathbf{m} | $ and that they satisfy \eqref{eq:symmrp}--\eqref{eq:lambdan}.  Then, from \eqref{eq:gm} we obtain $g_{\overline{\mathbf{m}}}=\overline{g}_\mathbf{m}$ and $\mathcal{E}_{\overline{\mathbf{m}}}=\overline{\mathcal{E}}_\mathbf{m}$.
 We can solve  $\widehat{\mathcal{R}}_{\mathbf{m}}= 0$  writing  $\boldsymbol{\phi  }_{\mathbf{m}   }=\(  \mathbf{L} _{1}    + \im \boldsymbol{\lambda  } \cdot \mathbf{m}  \) ^{-1}\mathcal{E}_\mathbf{m}$.  By  $\boldsymbol{\lambda  } \cdot \overline{\mathbf{m}}=-\boldsymbol{\lambda  } \cdot \mathbf{m}$, we conclude
$\boldsymbol{\phi  }_{\overline{\mathbf{m}}}=\overline{\boldsymbol{\phi  }}_\mathbf{m}$.

Let us now consider $\mathbf{m}\in  \boldsymbol{\Lambda  }_j$. We assume by induction $\boldsymbol{\phi  }_{\mathbf{m} '} $ have been defined for $|\mathbf{m}' | <| \mathbf{m} | $ and  so too $ \boldsymbol{\lambda  }_{\mathbf{n} '}$
  for $\|\mathbf{n}' \| <\| \mathbf{m} \| -1$. Then, for $ \mathbf{m}=\mathbf{n}+\mathbf{e}^j$ where $\mathbf{n}\in  \boldsymbol{\Lambda  }_0$,  $\widehat{\mathcal{R}}_{\mathbf{m}}= 0$  becomes
\begin{align} &    \nonumber     \(  \mathbf{L} _{1}     + \im   {\lambda  } _j  \)    \boldsymbol{\phi  }_{\mathbf{m} }  =-\im  \boldsymbol{\lambda  }_{\mathbf{n} } \cdot \mathbf{e}^{j} \boldsymbol{\Phi  }_{j }  -  \mathcal{K}_\mathbf{m}  \text{  with } \\&  \mathcal{K}_\mathbf{m}:=  g_\mathbf{m}- \sum_{\substack{\mathbf{m}' +\mathbf{n}'=\mathbf{m}\\ \mathbf{m}' \in \mathbf{NR  }, |\mathbf{m}'|\ge 2, \ \mathbf{n}'\in \boldsymbol{\Lambda  }_0   }}     {\im}  \boldsymbol{\lambda  }_{\mathbf{n}'} \cdot \mathbf{m}'   \boldsymbol{\phi  }_{\mathbf{m}'   }   .\label{eq:lambdaj}
\end{align}
This equation can be solved if we impose $ \( \mathbf{J}  {\mathcal{E}}_\mathbf{m} ,\overline{\boldsymbol{\Phi  }}_{j } \) =0 $, that is, for  $ {\lambda   }_{\mathbf{n} j } :=\boldsymbol{\lambda  }_{\mathbf{n} } \cdot \mathbf{e}^{j}$,
\begin{align*} & - {\im}  {\lambda   }_{\mathbf{n} j }   \( \mathbf{J}  \boldsymbol{\Phi  }_{j } ,\overline{\boldsymbol{\Phi  }}_{j }   \) =- {\im}  {\lambda   }_{\mathbf{n} j }    (-\im ) =  -    {\lambda   }_{\mathbf{n} j }        =      \( \mathbf{J}  {\mathcal{K}}_\mathbf{m} ,\overline{\boldsymbol{\Phi  }}_{j }\)   ,
\end{align*}
 which is true for ${\lambda   }_{\mathbf{n} j }        =    -  \( \mathbf{J}  {\mathcal{K}}_\mathbf{m} ,\overline{\boldsymbol{\Phi  }}_{j }\)  $.   Then we can solve for $ \boldsymbol{\phi  }_{\mathbf{m} } = -\(   \mathbf{L} _{1}   + \im  {\lambda  } _j  \) ^{-1} \mathcal{K}_\mathbf{m} $     in the complement, in \eqref{eq:Linz2eig13}, of $\ker (   \mathbf{L} _{1} -\im  \lambda _j) $.

\noindent We want to show that ${\lambda   }_{\mathbf{n} j } \in \R $. For the corresponding $\overline{\mathbf{m}} \in \overline{\boldsymbol{\Lambda  }}_{j }$,
we have
\begin{align} &        \(  \mathbf{L} _{1}  - \im  {\lambda  } _j  \)    \boldsymbol{\phi  }_{\overline{\mathbf{m}} }  =-\im  \boldsymbol{\lambda  }_{\mathbf{n} } \cdot \overline{\mathbf{e}}^{j} \overline{\boldsymbol{\Phi  }}_{j }  -  \mathcal{K}  _{ \overline{\mathbf{m}} }     \text{  with } \nonumber \\&  \mathcal{K}_  { \overline{\mathbf{ {m}}} }:=  g_ { \overline{\mathbf{ {m}}} }- \sum_{\substack{\overline{\mathbf{m}}' +\mathbf{n}'=\overline{\mathbf{m}}\\ \mathbf{m}' \in \mathbf{NR_2 }, \ \mathbf{n}'\in \boldsymbol{\Lambda  }_0   }}   \im  \boldsymbol{\lambda  }_{\mathbf{n}'} \cdot \overline{\mathbf{m}}'   \boldsymbol{\phi  }_{\overline{\mathbf{m}}'   }     . \label{eq:barlambdaj}
\end{align}
Notice that by induction  $\mathcal{K}  _{ \overline{\mathbf{m}} }=\overline{\mathcal{K} } _{ \mathbf{m} }$. Since  $\boldsymbol{\lambda  }_{\mathbf{n} } \cdot \overline{\mathbf{e}}^{j} = - {\lambda  }_{\mathbf{n}j }$,  taking the complex conjugate of  \eqref{eq:lambdaj} we obtain
\begin{equation}\label{eq:bothlambdaj}
\begin{aligned} &  \(  \mathbf{L} _{1}  -  \im  {\lambda  } _j  \)    \boldsymbol{\phi  }_{\overline{\mathbf{m}} }  = \im   {\lambda  }_{\mathbf{n} j}   \overline{\boldsymbol{\Phi  }}_{j }  -  \overline{\mathcal{K}}  _{ \mathbf{m}  } \text{  and  }\\& \(  \mathbf{L} _{1}   -  \im   {\lambda  } _j  \)    \overline{\boldsymbol{\phi  }}_{ \mathbf{m}  }  =  \im   \overline{{\lambda  }}_{\mathbf{n} j}  \overline{ \boldsymbol{\Phi  }} _{j }  -  \overline{\mathcal{K}}  _{ \mathbf{m}  } .
\end{aligned}
\end{equation}
Applying $  \( \mathbf{J} \cdot ,    \boldsymbol{\Phi  } _{j }\) $  on both the last two equations,  we obtain
\begin{align*} &      \text{$\im  {\lambda  }_{\mathbf{n} j}  \( \mathbf{J}   \overline{\boldsymbol{\Phi   }}_{j } ,   \boldsymbol{\Phi   }_{j } \) =  \( \mathbf{J}  \overline{\mathcal{K}}  _{ \mathbf{m}  } ,    \boldsymbol{\Phi   }_{j }\)  $   and   $\im   \overline{{\lambda  }}_{\mathbf{n} j}    \( \mathbf{J}  \overline{ \boldsymbol{\Phi  }}_{j } ,    \boldsymbol{\Phi   }_{j }\)  =  \( \mathbf{J}    \overline{\mathcal{K}}  _{ \mathbf{m}  } ,    \boldsymbol{\Phi   }_{j }\) $.}
\end{align*}
Hence ${\lambda  }_{\mathbf{n} j}=  \overline{{\lambda  }}_{\mathbf{n} j} $ and we have proved that ${\lambda  }_{\mathbf{n} j} \in \R$.

\noindent Since the equations in \eqref{eq:bothlambdaj} are the same, we conclude $ \boldsymbol{\phi  }_{\overline{\mathbf{m}} } =  \overline{\boldsymbol{\phi  }}_{ \mathbf{m}  } $.

\noindent  We consider now  $ \boldsymbol{ \mathcal{R}} [\mathbf{z}] = \widehat{\boldsymbol{ \mathcal{R}}}[\mathbf{z}] -D_{\mathbf{z}}\boldsymbol{\phi}[\mathbf{z}]\widetilde{\mathbf{z}}_{R}$ where we seek $\widetilde{\mathbf{z}}_{R}$  so that \eqref{eq:R1FRemainder--} is true.
This will follow from
\begin{align*} \<\mathbf{J} \widetilde{\boldsymbol{ \mathcal{R}}}[\mathbf{z}],D_{\mathbf{z}}\boldsymbol{\phi}[\mathbf{z}] \boldsymbol{\zeta}\>-\<\mathbf{J} D_{\mathbf{z}}\boldsymbol{\phi}[\mathbf{z}]\widetilde{\mathbf{z}}_{R},D_{\mathbf{z}}\boldsymbol{\phi}[\mathbf{z}] \boldsymbol{\zeta}\>=0  \text{ for the standard basis $\boldsymbol{\zeta} = e_1, \im e_1,...,e_N, \im e_N$}.\end{align*}
Since the restriction of  $\< \mathbf{J}\cdot , \cdot \>$ in $L^2_{discr}$ is a non--degenerate symplectic form and from  $\boldsymbol{\phi  }_{\mathbf{e}^j } = \boldsymbol{\Phi  }_{ j }   $   and  $\boldsymbol{\phi  }_{\overline{\mathbf{e} }^j } = \overline{\boldsymbol{\Phi  }}_{ j }
$,
the Implicit Function  Theorem guarantees the existence of $\widetilde{\mathbf{z}} _{  R} \in C ^{\infty}( \mathcal{B} _{\C ^N} (0, \alpha _0), \C^N)$  with $\widetilde{\mathbf{z}} _{  R}(\mathbf{0})=\mathbf{0}$ for a sufficiently small $\alpha _0>0$. Furthermore, from the last formula and from the fact that in the expansion \eqref{eq:RFtildeexpan} we have  $\widehat{\mathcal{{R}}}_{\mathbf{m}}=0$  for all $\mathbf{m} \in \mathbf{NR}$, we obtain  the bound \eqref{eq:boundzjR}. This in turn implies  expansion \eqref{eq:RFRemainder} and bound \eqref{eq:R1FRemainder}.
 \qed

Let us consider now the expansion \eqref{eq:RFRemainder}.
An important assumption, related to the  Fermi Golden Rule (FGR), is the following.
\begin{assumption}\label{ass:FGR}
We assume that for all $\mathbf{m}\in \mathbf{R}_{\mathrm{min}}$,
\begin{align*}
\sum_{\sigma=\pm 1}  \left |  \left [-\im \sqrt{(\boldsymbol{\lambda} \cdot \mathbf{m})^2-\omega ^2} \ \widehat{( \mathcal{R} _{\mathbf{m}})} _{1}( \sigma     \sqrt[4]{ {(\boldsymbol{\lambda} \cdot \mathbf{m})^2-\omega ^2}} )+\widehat{( \mathcal{R}_{\mathbf{m}})} _{2}( \sigma     \sqrt[4]{ {(\boldsymbol{\lambda} \cdot \mathbf{m})^2-\omega ^2}} ) \right ]     \right |>0,
\end{align*}
where $( \mathcal{R}_{\mathbf{m}}) _{j}$ are the two components of  $ \mathcal{R}_{\mathbf{m}}$ for $j=1,2$ and we are taking the distorted Fourier transform associated to operator $L_1$, for which we refer to Weder \cite{weder}.
\end{assumption}

\section{Modulation and transformed equations}\label{sec:dar}

  For small
    $\alpha \in (0, 1)$    we set
\begin{align*}
 \mathcal{M}_\alpha (c_*)=\{     \boldsymbol{\phi} [    \mathbf{z} ] \   |\    \mathbf{z}\in \mathcal{B}_{\C^N}(0,\alpha)\},\ \text{where}\ \mathcal{B}_{\C^N}(0,\alpha):=\{\mathbf{w}\in \C^N\ |\ |\mathbf{w}|<\alpha \}
\end{align*}
We first observe the following.
\begin{lemma}\label{lem:symplect_man0}
There is an $\alpha  _0\in (0, 1)$  such that for $\alpha \in (0,\alpha  _0)$ the map $  \mathbf{z}  \to     \boldsymbol{\phi} [    \mathbf{z} ]$ in an embedding   $\mathcal{B}_{\C^N}(0,\alpha)   \hookrightarrow \mathbf{E _{H}}$.
\end{lemma}
\proof  It is clear that the map is smooth. Next we observe that for $ \alpha >0$ sufficiently small,   the   above map
  $\mathcal{B}_{\C^N}(0,\alpha) \to \mathbf{E _{H}}$ is an embedding. This follows from the fact that the partial derivatives computed  at  $\mathbf{z}=0$    span  $L^2_{discr}$, which is symplectic with respect to the   form $\< \mathbf{J} \cdot, \cdot \>$.\qed

We set
\begin{align}\label{eq:RForth}
\mathcal{H}_c[\mathbf{z}]:=\{\mathbf{u}\in \mathbf{H} +\Sigma^*\ |\ \forall  \boldsymbol{\zeta}\in \C^N,  \<\mathbf{J}\mathbf{u},D_{\mathbf{z}}\boldsymbol{\phi}[\mathbf{z}] \boldsymbol{\zeta}\>=0\}.
\end{align}
\begin{lemma}[Modulation]\label{lem:modulat}
There exists $\delta>0$ s.t. there exists $\mathbf{z}(\cdot)\in C^\infty(B_{\mathbf{E_H}}(0,\delta),\C^N)$ s.t.
\begin{align}
 \boldsymbol{\eta}(\mathbf{u}):=\mathbf{u}-\boldsymbol{\phi}[\mathbf{z}(\mathbf{u})]\in \mathcal{H}_c[\mathbf{z}(\mathbf{u})]  \label{eq:lem:mod-1}
\end{align}  and, leaving implicit the dependence of $\mathbf{z}$ and $   \boldsymbol{\eta}$   on $\mathbf{u}$,
\begin{align}
|\mathbf{z}| + \|\boldsymbol{\eta}\|_{ \boldsymbol{\mathcal{H}}^1}\sim \|\mathbf{u}-\mathbf{H}\|_{\boldsymbol{\mathcal{H}}^{1}}. \label{eq:lem:mod-2}
\end{align}
\end{lemma}

\begin{proof} For $z_{jR}=\Re (z_j)$ and $z_{jI}=\Im (z_j)$,  consider a function
$F(\mathbf{u} ,\mathbf{z})$ with components
\begin{align}&    \label{eq:modulation1}
	 \< \mathbf{J} \(  \mathbf{ u}   -    \boldsymbol{\phi}   [  \mathbf{z} ]  \)  ,  \partial  _{z_{jJ}}  \boldsymbol{\phi}   [  \mathbf{z} ]      \>    \text{   for   } j=1,...,\widetilde{N} \text{  and } J=R,I.
	\end{align}
Then $F\in C^\infty \( \mathbf{E} _\mathbf{H}    \times   \C ^{ {N}}, \R ^{2 {N} }\) $, trivially we have  $F( \mathbf{H} ,0)=0$  and  the Jacobian matrix
$
   \frac{\partial F  }{\partial    \mathbf{z} }  { (\mathbf{H} ,  0)}
$
a non--degenerate   $ {N}\times  {N}$  matrix,  exactly  because,    for the space in $ L^2_{discr}$ in \eqref{eq:Linz2eig13},  the form $\< \mathbf{J} \cdot , \cdot \>$  is symplectic.  Then, by Implicit Function Theorem,  there exists the desired function $\mathbf{z}(\mathbf{u})$ such that $F(\mathbf{u} , \mathbf{z}(\mathbf{u}))=0$, i.e.  which satisfies \eqref{eq:lem:mod-1},
with $\mathbf{z}(\mathbf{H})=0$.  The fact that $|\mathbf{z}| + \|\boldsymbol{\eta}\|_{ \boldsymbol{\mathcal{H}}^1}\lesssim \|\mathbf{u}-\mathbf{H}\|_{ \boldsymbol{\mathcal{H}}^{1}}$ follows from the Lipschitz regularity of $\mathbf{u}\to (  \mathbf{{z}}(\mathbf{u}), \boldsymbol{\eta}(\mathbf{u}))$  at $\mathbf{H}$, while we have
\begin{align*}  \|\mathbf{u}-\mathbf{H}\|_{ \boldsymbol{\mathcal{H}}^{1}} =  \|  \boldsymbol{\phi}   [  \mathbf{z} ]  -\mathbf{H}    +\boldsymbol{\eta} \|_{ \boldsymbol{\mathcal{H}}^{1}} \le \|  \boldsymbol{\phi}   [  \mathbf{z} ]  -\mathbf{H} \|_{ \boldsymbol{\mathcal{H}}^{1}}    +\|  \boldsymbol{\eta} \|_{ \boldsymbol{\mathcal{H}}^{1}} \lesssim |\mathbf{z}| + \|\boldsymbol{\eta}\|_{ \boldsymbol{\mathcal{H}}^{1}}.
\end{align*}

\end{proof}

Substituting $\mathbf{u}=\boldsymbol{\phi}[\mathbf{z}]+\boldsymbol{\eta}$, we obtain
\begin{align}\label{eq:eta}
\partial_t \boldsymbol{\eta} + D\boldsymbol{\phi}[\mathbf{z}](\dot{\mathbf{z}}-\widetilde{\mathbf{z}})= \mathbf{L}_1\boldsymbol{\eta}   + (\mathbf{L}[\mathbf{z}]- \mathbf{L}_1)\boldsymbol{\eta}  + \mathbf{J} \mathbf{F}[\mathbf{z},\boldsymbol{\eta}] + \mathbf{R}[\mathbf{z}],
\end{align}
where

\begin{align}& \mathbf{L}[\mathbf{z}] :=\mathbf{J}\mathbf{H}[\mathbf{z}],\ \mathbf{H}[\mathbf{z}]=\begin{pmatrix}
-\partial_x^2 + W''(\phi_1[\mathbf{z}]) & 0 \\ 0 & 1
\end{pmatrix}    \label{eq:defLz} \\&
\mathbf{F}[\mathbf{z},\boldsymbol{\eta}] :=\begin{pmatrix}
F_1[\mathbf{z},\eta_1]\\ 0
\end{pmatrix} \text{,  where } F_1[\mathbf{z},\eta_1]:=  W'(\phi_1[\mathbf{z}]+\eta_1)-W'(\phi_1[\mathbf{z}])-W''(\phi_1[\mathbf{z}])\eta_1  .  \label{eq:defF}
\end{align}

We denote by $P_{\mathrm{c}}$ the projection onto $L^2_{disp}$ associated to the splitting \eqref{eq:Linz2eig13} and let
$P_d=1-P_{\mathrm{c}}$.
\begin{remark} \label{rem:comm}    Notice that $L_1=\text{diag}(L_1,L_1)$  commutes with  $\mathbf{L}_1$ and  with the resolvent  $ R _{ \mathbf{L}_1 } ( \varsigma ):= (\mathbf{L}_1 - \varsigma)^{-1}    $ for $\varsigma$  in the resolvent set of $\mathbf{L}_1$. It then follows that
$L_1$ commutes with the projections $P_d$ and $P_c$.
   \end{remark}

\begin{lemma}[Inverse of $P_{\mathrm{c}}$]\label{lem:R} There exists an $\alpha _0>0$  and $R[\mathbf{z}] \in C ^{\infty}  \( \mathcal{B} _{\C ^N}(0,\alpha _0), \mathcal{L} ( L^2) \)$  s.t.\ $\left.R[\mathbf{z}]P_{c}\right|_{\mathcal{H}_c[\mathbf{z}]}=\left.1\right|_{\mathcal{H}_c[\mathbf{z}]}$, $\left.P_c R[\mathbf{z}]\right|_{P_c (\boldsymbol{\Sigma}^{ l})^*}=\left.1\right|_{P_c (\boldsymbol{\Sigma}^{ l})^*}$  for all $ l \in \N _0$  and
\begin{align}
\|R[\mathbf{z}]-1\|_{(\boldsymbol{\Sigma}^{ l})^*\to \boldsymbol{\Sigma}^{ l}}\lesssim _{ l}|\mathbf{z}|.\label{eq:lem:R}
\end{align}
\end{lemma}

\begin{proof}
Let us write, summing on repeated indexes,
\begin{align*}
 R[\mathbf{z}] =1+ \< \mathbf{J} \cdot , C _{j\mathcal{A}}[\mathbf{z}] \> \partial _{z_{j\mathcal{A}}}\boldsymbol{\phi}[\mathbf{0}] , \text{ with $j=1,...,N$, $\mathcal{A}=R,I$, $z_{jR} :=\Re z_j$ and $z_{jI} :=\Im z_j$}.
\end{align*}
Then $R[\mathbf{z}] \boldsymbol{\theta} \in \mathcal{H}_c[\mathbf{z}] $ for all $\boldsymbol{\theta} \in L^{2}_{\odd}(\R , \C^2)$ is equivalent to
\begin{align*}
 \<  \mathbf{J} \boldsymbol{\theta} + \mathbf{J}\< \mathbf{J} \boldsymbol{\theta}  , C _{j\mathcal{A}}[\mathbf{z}] \> \partial _{z_{j\mathcal{A}}}\boldsymbol{\phi}[\mathbf{0}]  , \partial _{z_{j'\mathcal{A}'}}\boldsymbol{\phi}[\mathbf{z}] \> =0 \text{  for all  $j'=1,...,N$, $\mathcal{A}'=R,I$}
\end{align*}
or, equivalently, for all $\boldsymbol{\theta} \in L^{2}_{\odd}(\R , \C^2)$
\begin{align*}
 \<   \<\mathbf{J} \boldsymbol{\theta}  , C _{j\mathcal{A}}[\mathbf{z}] \>\mathbf{J} \partial _{z_{j\mathcal{A}}}\boldsymbol{\phi}[\mathbf{0}]  , \partial _{z_{j'\mathcal{A}'}}\boldsymbol{\phi}[\mathbf{z}] \>   =  \< \mathbf{J} \boldsymbol{\theta}  ,      \< \mathbf{J}    \partial _{z_{j\mathcal{A}}}\boldsymbol{\phi}[\mathbf{0}]  , \partial _{z_{j'\mathcal{A}'}}\boldsymbol{\phi}[\mathbf{z}] \>  C _{j\mathcal{A}}[\mathbf{z}] \>  = -  \< J \boldsymbol{\theta} , \partial _{z_{j'\mathcal{A}'}}\boldsymbol{\phi}[\mathbf{z}] \> ,
\end{align*}
that is, still summing on  the repeated indexes $j=1,...,N$, $\mathcal{A}=R,I$,
\begin{align*}
    \< \mathbf{J}    \partial _{z_{j\mathcal{A}}}\boldsymbol{\phi}[\mathbf{0}]  , \partial _{z_{j'\mathcal{A}'}}\boldsymbol{\phi}[\mathbf{z}] \>   C _{j\mathcal{A}}[\mathbf{z}]   = -   \partial _{z_{j'\mathcal{A}'}}\boldsymbol{\phi}[\mathbf{z}]  \text{  for all  $j'=1,...,N$, $\mathcal{A}'=R,I$} .
\end{align*}
By the invertibility of the matrix $\left \{  \<   \mathbf{J}  \partial _{z_{j\mathcal{A}}}\boldsymbol{\phi}[\mathbf{0}]  , \partial _{z_{j'\mathcal{A}'}}\boldsymbol{\phi}[\mathbf{z}] \>  \right \}  $, this equation has a solution for $|\mathbf{z}|$ small, which is unique. So $\mathbf{z}\to  C _{j\mathcal{A}}[\mathbf{z}]$ is smooth near $\mathbf{0}$  with values in $\boldsymbol{\Sigma} ^{l}$ for all $l\in \N _0$.
We conclude   $R[\mathbf{z}]\in \mathcal{L}\( (\boldsymbol{\Sigma} ^{l})^*, \mathcal{H}_c[\mathbf{z}]\)$.  Now
\begin{align*}
     P_c R[\mathbf{z}] =P_c  +  \< \mathbf{J} \cdot , C _{j\mathcal{A}}[\mathbf{z}] \> P_c\partial _{z_{j\mathcal{A}}}\boldsymbol{\phi}[\mathbf{0}] =P_c
\end{align*}
so it equals $\left.1\right|_{P_c (\boldsymbol{\Sigma} ^{l})^*}$  when restricted in $P_c (\boldsymbol{\Sigma} ^{l})^*$, and in particular for $l=1$.

Next,  notice that for $\boldsymbol{\theta}\in \mathcal{H}_c[\mathbf{z}]$, we have $ R[\mathbf{z}]P_c\boldsymbol{\theta}\in \mathcal{H}_c[\mathbf{z}]$ with $P_c\boldsymbol{\theta}= P_cR[\mathbf{z}]P_c\boldsymbol{\theta} $. Since,  for $|\mathbf{z}|$ small, $P_c$ is an isomorphism
from $\mathcal{H}_c[\mathbf{z}]$ to $ L ^{2}_{disp}$, we have $\left.R[\mathbf{z}]P_{c}\right|_{\mathcal{H}_c[\mathbf{z}]}=\left.1\right|_{\mathcal{H}_c[\mathbf{z}]}$.

\end{proof}

We set $\widetilde{\boldsymbol{\eta}}=P_c \boldsymbol{\eta}$ (and thus $\boldsymbol{\eta}=R[\mathbf{z}]\widetilde{\boldsymbol{\eta}}$).
Then, $\widetilde{\boldsymbol{\eta}}$ satisfies
\begin{align}\label{eq:tildeeta}
\partial_t \widetilde{\boldsymbol{\eta}} = \mathbf{L}_1\widetilde{\boldsymbol{\eta}} + \sum_{\mathbf{m}\in \mathbf{R}_{\mathrm{min}}}\mathbf{z}^{\mathbf{m}}P_c\mathcal{R}_{\mathbf{m}} + \mathbf{R}_{\widetilde{\boldsymbol{\eta}}},
\end{align}
where
\begin{align}\label{eq:Rtildeeta}
\mathbf{R}_{\widetilde{\boldsymbol{\eta}}}=P_c \boldsymbol{\mathcal{R}}_1[\mathbf{z}] + P_c \mathbf{J}\mathbf{F}+ P_c (\mathbf{L}[\mathbf{z}]- \mathbf{L}_1)\boldsymbol{\eta}  - P_cD\boldsymbol{\phi}[\mathbf{z}](\dot{\mathbf{z}}-\widetilde{\mathbf{z}})
+P_c \mathbf{L}_1(R[\mathbf{z}]-1)\widetilde{\boldsymbol{\eta}}   .
\end{align}
  We set
\begin{align} \label{def:Tg} &
\mathcal{T}:=\<\im \varepsilon \partial_x\>^{-\widetilde{N} }\mathcal{A}^* \text{   and }
\\& \label{def:vBg}
\mathbf{v}:=\mathcal{T}\chi_{B^2}\widetilde{\boldsymbol{\eta}}.
\end{align}
Then, also multiplying by the imaginary unit $\im$, we obtain
\begin{align}\label{eq:vBg}
\im \partial_t \mathbf{v}=\im \mathbf{L}_{D}\mathbf{v} + \sum_{\mathbf{m}\in \mathbf{R}_{\mathrm{min}}} \mathbf{z}^{\mathbf{m}}   \im \widetilde{ \mathcal{R}}_{\mathbf{m}}+\im \mathbf{R}_{\mathbf{v}},
\end{align}
where
\begin{align}  &  \mathbf{L}_{D} := \begin{pmatrix} 0  & 1 \\  -L_D  & 0
\end{pmatrix} \label{eq:LTR}
\\& \widetilde{ \mathcal{R}}_{\mathbf{m}}:= \mathcal{T}\chi_{B^2}P_c \mathcal{R}_{\mathbf{m}}  \text{  and}    \label{def:Gbgm}\\&
\mathbf{R}_{\mathbf{v}}=\mathcal{T}\chi_{B^2}P_c\mathbf{R}_{\widetilde{\boldsymbol{\eta}}}+ \begin{pmatrix} 0 \\ -\mathcal{T}(2\chi_{B^2}'\partial_x +\chi_{B^2}'')\widetilde{\eta}_1+\<\im \varepsilon\partial_x\>^{-\widetilde{N} }\left [V_{D},\<\im \varepsilon\partial_x\>^{\widetilde{N} }\right ]v_1
\end{pmatrix}. \label{def:RvBg}
\end{align}
Before stating our  main estimates, we state the following orbital stability result, which follows from   the Modulation Lemma \ref{lem:modulat} and the Orbital Stability Theorem in \cite{KMMvdB21AnnPDE}, in fact also the classical \cite{HPW82}.
\begin{proposition}[Orbital stability]\label{prop:OrbStab}
 There exist    $C>0$ and $\delta _0>0$ such that, for $\delta :=\| \mathbf{u}(0)-\mathbf{H}\|_{ \boldsymbol{\mathcal{H}}^{1}}<\delta _0$ the claim in line \eqref{eq:main1} is true for all $t\in \R$ and
 we have
\begin{align}\label{eq:lincor1}
\|\mathbf{z}\|_{L^\infty(\R )}+\|\boldsymbol{\eta}\|_{L^\infty (\R,  H^1)}\le C_0 \delta.
\end{align}
\end{proposition}
\qed

Notice that the above result and \eqref{eq:main2}, which will be proved below, along with  Lemma \ref{lem:combinat1} guarantee by elementary arguments the limit
\eqref{eq:main3}.  So here the main point is \eqref{eq:main2}
 for $I=\R_ +$.

\section{Main estimates and proof of Theorem \ref{thm:main}}



The proof of \eqref{eq:main2} in Theorem \ref{thm:main} is by means of a continuation argument. In particular, we will show the following.
\begin{proposition}\label{prop:continuation}  Assumptions  \ref{ass:repuls}, \ref{ass:generic} and \ref{ass:FGR} are given.
Then  for any small $\epsilon>0 $ there exists  a    $\delta _0= \delta _0(\epsilon )   $ s.t.\  if    \eqref{eq:main2}
holds  for $I=[0,T]$ for some $T>0$ and for $\delta  \in (0, \delta _0)$
then in fact for $I=[0,T]$    inequality   \eqref{eq:main2} holds   for   $\epsilon$ replaced by $\epsilon/2$.
\end{proposition}

Theorem \ref{thm:main} is a corollary of Proposition \ref{prop:continuation}.
\begin{proof}
By completely routine arguments, which we skip,
it is possible to show that {Proposition} \ref{prop:continuation}  implies   \eqref{eq:main2} with $I=[0,\infty)$. The time reversibility  of the system, yields immediately \eqref{eq:main2} for $I=\R$.
Finally, \eqref{eq:main3} follows from the integrabililty of $|z_j|^{2m_j}$ where $m_j$ is the smallest integer satisfying $\omega<m_j \lambda_j$, which follows from the FGR estimate given in Proposition \ref{prop:FGR} below and the boundedness of $\dot{\mathbf{z}}$ which can be easily obtained from the modulation equation and orbital stability.
\end{proof}

We set $\chi \in C_{\mathrm{even}}^\infty(\R)$ to satisfy $1_{|x|\leq 1}\leq \chi \leq 1_{|x|\leq 2}$ and $\chi'(x)\leq 0$ for $x>0$.
For $C>0$,
\begin{align}\label{def:zetaphi}
	\zeta_C(x):=\exp\(-\frac{|x|}{C}\(1-\chi(x)\)\),\ \varphi_C(x):=\int_0^x \zeta_C^2(y)\,dy .
\end{align}
We will consider constants  $A, B,\varepsilon >0$ satisfying
\begin{align}\label{eq:relABg}
	\log(\delta ^{-1})\gg\log(\epsilon ^{-1})\gg  A\gg    B^2\gg B \gg  \exp \( \varepsilon ^{-1} \) \gg 1.
\end{align}
We will denote by    $o_{\varepsilon}(1)$  constants depending on $\varepsilon$ such that
\begin{align}\label{eq:smallo}
	\text{ $o_{\varepsilon}(1) \xrightarrow {\varepsilon  \to 0^+   }0.$}
\end{align}
We set
\begin{align}\label{def:wAxiB}&
	\mathbf{w}=\zeta_A \widetilde{\boldsymbol{\eta}},\quad  \boldsymbol{\xi}:=\chi_{B^2} \zeta_B \mathbf{v}.
\end{align}
 We will prove the following  continuation argument.
 \begin{proposition}\label{prop:contreform}   Assumptions  \ref{ass:repuls}, \ref{ass:generic} and \ref{ass:FGR} are given. Then
  for any small $\epsilon>0 $
there exists  a    $\delta _0 = \delta _0(\epsilon )   $ s.t.\  if    in  $I=[0,T]$ we have
\begin{align}&
\|\dot {\mathbf{z}} - \widetilde{\mathbf{z}}\|_{L^2(I)}+\sum_{\mathbf{m}\in \mathbf{R}_{\mathrm{min}}}\|\mathbf{z}^\mathbf{m}\|_{L^2(I)}+ \| \boldsymbol{\xi}    \|_{L^2(I,\boldsymbol{\widetilde{\Sigma}} )}+\|  \mathbf{w} \|_{L^2(I, \boldsymbol{\widetilde{\Sigma}})}\le   \epsilon   \label{eq:main11}
\end{align}
then  for $\delta  \in (0, \delta _0)$
     inequality   \eqref{eq:main11} holds   for   $\epsilon$ replaced by $ o_{\varepsilon}(1)   \epsilon $    where $o_{\varepsilon}(1) \xrightarrow {\varepsilon  \to 0^+   }0 $.
\end{proposition}
Notice that Proposition \ref{prop:contreform} implies Proposition \ref{prop:continuation}.
In the following, we always assume the assumptions of the claim of Proposition \ref{prop:contreform}, which are true for $T>0$ small enough.

The following is proved is   Proposition 9.1 of \cite{CM2109.08108}.
\begin{proposition}[Coercivity]\label{prop:coerc}
We have
\begin{align}
\|w_1\|_{L^2_{-\frac{a}{10}}  }\lesssim \|\xi_1\|_{\widetilde{\Sigma}}+e^{-\frac{B}{20}}\|w_1'\|_{L^2}. \label{eq:coerc1}
\end{align}
\end{proposition}
\qed

In   analogy to \cite{CM2109.08108}, we now   consider essentially two virial estimates, one for $\mathbf{w}$ and the other for $\boldsymbol{\xi}$.
 The first is based directly on the equation for $\boldsymbol{\widetilde{\eta}}$,  \eqref{eq:tildeeta}.

\begin{proposition}[1st virial estimate]\label{prop:1stvirial}
We have
\begin{align}
\|w_1'\|_{L^2(I,L^2)}+\|w_2\|_{L^2(I, L^2_{-\frac{a}{10}})}&\lesssim  o_{\varepsilon}(1) \epsilon  \nonumber \\&+ \|w_1\|_{L^2(I,L^2_{-\frac{a}{10}})} + \sum_{\mathbf{m}\in \mathbf{R}_{\mathrm{min}}} \| \mathbf{z}^{\mathbf{m}}\|_{L^2(I)} + \delta \|\dot{\mathbf{z}}-\widetilde{\mathbf{z}}\|_{L^2(I)}.\label{eq:1stvInt}
\end{align}
\end{proposition}

The second  virial estimate,  involves  the transformed problem \eqref{eq:vBg}.

\begin{proposition}[2nd virial]\label{prop:2ndvirial}
We have
\begin{align}\label{eq:2ndv}
\|    \boldsymbol{\xi}   \|_{L^2(I,\boldsymbol{\widetilde{\Sigma}} )} \lesssim  o_{\varepsilon}(1) \epsilon +\sum_{\mathbf{m}\in \mathbf{R}_{\mathrm{min}}}\|\mathbf{z}^{\mathbf{m}}\|_{L^2(I)}+o_{\varepsilon}(1) \|\dot{\mathbf{z}}-\widetilde{\mathbf{z}}\|_{L^(I)2}+o_{\varepsilon}(1)  \| \mathbf{w} \|_{L^2,\boldsymbol{\widetilde{\Sigma}} )}
\end{align}
\end{proposition}

We will also need a control of modulation parameters.

\begin{proposition}\label{prop:modp}
We have
\begin{align}
\|\dot{\mathbf{z}}-\widetilde{\mathbf{z}}\|_{L^2(I)}=
o_{\varepsilon}(1) \|\mathbf{w}\|_{L^2(I,L^2_{-\frac{a}{10}})}. \label{eq:lem:estdtz}
\end{align}
\end{proposition}

The last ingredient is the FGR estimate.
 \begin{proposition}[FGR estimate]\label{prop:FGR}
 We have
 \begin{align}\label{eq:FGRint}
 \sum_{\mathbf{m}\in \mathbf{R}_{\mathrm{min}}}\|\mathbf{z}^{\mathbf{m}}\|_{L^2(I)}=  o_{\varepsilon}(1)   \epsilon   .
 \end{align}
 \end{proposition}

\subsubsection{Proof of Proposition \ref{prop:contreform} assuming Propositions
  \ref{prop:1stvirial}--\ref{prop:FGR}.}\label{sec:profcont}

By \eqref{eq:lem:estdtz}--\eqref{eq:FGRint} and by  the relation between $A,B,\varepsilon , \epsilon $ and $\delta$ in \eqref{eq:relABg},  we have
\begin{align}
\|\dot{\mathbf{z}}-\widetilde{\mathbf{z}}\|_{L^2(I)} +\sum_{\mathbf{m}\in \mathbf{R}_{\mathrm{min}} }\|\mathbf{z}^{\mathbf{m}}\|_{L^2(I)}\le  o_\varepsilon(1) \epsilon .\label{FGReqfinal}
\end{align}
Entering this in \eqref{eq:2ndv} we get
\begin{align}\label{eq:2ndvfin}
\|  \boldsymbol{\xi}   \|_{L^2(I, \boldsymbol{\widetilde{\Sigma}})}\le     o_\varepsilon(1) \epsilon
\end{align}
which, fed in \eqref{eq:coerc1},  yields
\begin{align*} &
 \|  w  _1\|_{L^2(I, L^2_{-\frac{a}{10}})}\le   o_\varepsilon(1) \epsilon   .
 \end{align*}
Using this in \eqref{eq:1stvInt},  we obtain
\begin{align*}
\|w_1'\|_{L^2(I,L^2)}+\|w_2\|_{L^2(I, L^2_{-\frac{a}{10}})}&\le   o_\varepsilon(1) \epsilon .
\end{align*}
This and   the previous one  together, yield
\begin{align}\label{eq:1stvIntfin}
\| \mathbf{w} \|_{L^2(I, \boldsymbol{\widetilde{\Sigma}})}\le     o_\varepsilon(1) \epsilon .
\end{align}
Taken together, \eqref{FGReqfinal}--\eqref{eq:1stvIntfin} yield the improvement $o_\varepsilon(1) \epsilon $  of the statement of Proposition \ref{prop:contreform}, concluding the proof. \qed

 We now turn to  the proofs of  Propositions
  \ref{prop:1stvirial}--\ref{prop:FGR}. The structure of the proofs is very similar to the analogous  ones in \cite{CM2109.08108}. In particular, 
Propositions
  \ref{prop:1stvirial}--\ref{prop:2ndvirial}  are very close to Kowalczyk   et al.  \cite{KMM3}. The proof of Proposition \ref{prop:FGR} requires the introduction of an additional variable $\mathbf{g}$, which, like in \cite{CM2109.08108}, is bounded using smoothing estimates: in particular here we borrow from Komech--Kapytula \cite{KomKop10,KomKop111}.

\section{First virial estimate, for $\mathbf{w}$}
\label{sec:1virial}

Recall $\< f,g\> =Re (f,\overline{g})$, see \eqref{eq:inner1}.
For $A\gg 1$ to be determined, we set
\begin{align*}
\mathcal{I}_1(  \widetilde{\boldsymbol{\eta}} ):=\frac{1}{2}\<\mathbf{J}  \widetilde{\boldsymbol{\eta}},S_A  \widetilde{\boldsymbol{\eta}}\>,\quad
\mathcal{I}_2(  \widetilde{\boldsymbol{\eta}}):=\frac{1}{2}\<\mathbf{J}  \widetilde{\boldsymbol{\eta}},\zeta_{\widetilde{A}}^2 \sigma_3   \widetilde{\boldsymbol{\eta}}\>, \text{ with }\sigma_3 = \begin{pmatrix}
1 & 0 \\ 0 & -1
\end{pmatrix} ,
\end{align*}
where $\widetilde{A}^{-1}=A^{-1}+\frac{a}{10}$ and
\begin{align*}
S_A:=\frac{1}{2}\varphi_A' +\varphi_A \partial_x.
\end{align*}
\begin{remark}
By the definition of $\widetilde{A}$ and \eqref{def:zetaphi}, we have $\zeta_{\widetilde{A}}=\zeta_{\frac{10}{a}}\zeta_A$.
\end{remark}

\begin{lemma}\label{lem:1stvdiff}
For any $c\in (0,1)$, we have
\begin{align}
\frac{d}{dt}\mathcal{I}_1(  \widetilde{\boldsymbol{\eta}})+\frac{1}{2}\|w_1'\|_{L^2}^2 &\lesssim \|w_1\|_{L^2_{-\frac{a}{10}}}^2+ c\|w_2\|_{L^2_{-\frac{a}{10}}}^2 + c^{-1}\sum_{\mathbf{m}\in \mathbf{R}_{\mathrm{min}}} |\mathbf{z}^{\mathbf{m}}|^2+\delta |\dot{\mathbf{z}}-\widetilde{\mathbf{z}}|^2,\label{eq:1stv1}\\
-\frac{d}{dt}\mathcal{I}_2(  \widetilde{\boldsymbol{\eta}})+\frac{1}{2}\|w_2\|_{L^2_{-\frac{a}{10}}}^2
&\lesssim \|w_1\|_{L^2_{-\frac{a}{10}}}^2 +  \|w_1'\|_{L^2}^2+ \sum_{\mathbf{m}\in \mathbf{R}_{\mathrm{min}}} |\mathbf{z}^{\mathbf{m}}|^2+\delta |\dot{\mathbf{z}}-\widetilde{\mathbf{z}}|^2,\label{eq:1stv2}
\end{align}
where the implicit constants are independent of $c$.
\end{lemma}

\begin{proof}[Proof of Proposition \ref{prop:1stvirial}]
From the orbital stability bound Proposition \ref{prop:OrbStab}, we have $|\mathcal{I}_1(  \widetilde{\boldsymbol{\eta}})|\lesssim A \delta^2$ and $|\mathcal{I}_2( \widetilde{\boldsymbol{\eta}})|\lesssim \delta^2$.
Thus, integrating \eqref{eq:1stv1} and \eqref{eq:1stv2}, we have
\begin{align}
&\frac{1}{2}\|w_1'\|_{L^2(I,L^2)}^2 \label{eq:proofprop1stv1}\\& \leq C_1\(A\delta^2+\|w_1\|_{L^2(I,L^2_{-\frac{a}{10}})}^2 +c\|w_2\|_{L^2(I,L^2(I))}^2+ c^{-1}\sum_{\mathbf{m}\in \mathbf{R}_{\mathrm{min}}} \|\mathbf{z}^{\mathbf{m}}\|_{L^2(I)}^2+\delta^2 \|\dot{\mathbf{z}}-\widetilde{\mathbf{z}}\|_{L^2(I)}^2\), \nonumber\\
&\frac{1}{2}\|w_2\|_{L^2(I,L^2)}^2 \label{eq:proofprop1stv2}\\&\leq C_2\( \delta^2+\|w_1\|_{L^2(I,L^2_{-\frac{a}{10}})}^2 +\delta\|w_1'\|_{L^2(I,L^2(I))}^2+ \sum_{\mathbf{m}\in \mathbf{R}_{\mathrm{min}}} \|\mathbf{z}^{\mathbf{m}}\|_{L^2(I)}^2+\delta^2 \|\dot{\mathbf{z}}-\widetilde{\mathbf{z}}\|_{L^2(I)}^2\). \nonumber
\end{align}
Taking $c$ sufficiently small so that $4cC_1C_2\leq 1$ and substituting \eqref{eq:proofprop1stv2} into \eqref{eq:proofprop1stv1}, we can bound $\|w_1'\|_{L^2(I,L^2)}^2\lesssim \(\mathrm{r.h.s.\ of\ }\eqref{eq:1stvInt}\)^2$.
Finally, using \eqref{eq:proofprop1stv2} again, we have the conclusion.
\end{proof}

The remainder of this section is devoted to the proof of Lemma \ref{lem:1stvdiff}.
First, since both $\<\mathbf{J} \cdot, S_A\cdot\>$ and $\<\mathbf{J} \cdot, \zeta_A^2\sigma_3 \cdot\>$ are symmetric, we have
\begin{align}
\frac{d}{dt} \mathcal{I}_1( \widetilde{\boldsymbol{\eta}})
&=\<\mathbf{J}\(  \mathbf{L}_1 \widetilde{\boldsymbol{\eta}} + \sum_{\mathbf{m}\in \mathbf{R}_{\mathrm{min}}}\mathbf{z}^{\mathbf{m}}G_{\mathbf{m}} + \mathbf{R}_{ \widetilde{\boldsymbol{\eta}}} \),S_A \widetilde{\boldsymbol{\eta}} \>,\label{eq:diffI1}\\
\frac{d}{dt} \mathcal{I}_2(\widetilde{\boldsymbol{\eta}})
&=\<\mathbf{J}\(  \mathbf{L}_1 \widetilde{\boldsymbol{\eta}} + \sum_{\mathbf{m}\in \mathbf{R}_{\mathrm{min}}}\mathbf{z}^{\mathbf{m}}G_{\mathbf{m}} + \mathbf{R}_{ \widetilde{\boldsymbol{\eta}}} \),\zeta_{\widetilde{A}}^2 \sigma_3\widetilde{\boldsymbol{\eta}} \>.\label{eq:diffI2}
\end{align}
We will investigate each terms in the r.h.s.\ of \eqref{eq:diffI1} and \eqref{eq:diffI2}.

\begin{lemma}\label{lem:v11main}
We have
\begin{align*}
\<\mathbf{J}  \mathbf{L}_1 \widetilde{\boldsymbol{\eta}} ,S_A \widetilde{\boldsymbol{\eta}} \>=-\|w_1'\|_{L^2}^2 + O\( \|w_1\|_{L^2_{-\frac{a}{10}}}^2 \) .
\end{align*}
\end{lemma}

\begin{proof}
First, we have
\begin{align}
\<\mathbf{J}  \mathbf{L}_1  \widetilde{\boldsymbol{\eta}} ,S_A   \widetilde{\boldsymbol{\eta}} \>=-\<L_1 \widetilde{{\eta}}_1,S_A \widetilde{{\eta}}_1\>-\< \widetilde{{\eta}}_2,S_A \widetilde{{\eta}}_2\>=-\<L_1 \widetilde{{\eta}}_1,S_A \widetilde{{\eta}}_1\>.
\end{align}
From \cite{CM2109.08108} Lemma 4.2, we have
\begin{align}
\<L_1 \widetilde{{\eta}}_1,S_A \widetilde{{\eta}}_1\>=\|w_1'\|_{L^2}^2 - \frac{1}{2}\int_{\R}\frac{\varphi_A}{\zeta_A^2}V' |w_1|^2\,dx + \frac{1}{2A}\int \(\chi'' |x| +2 \chi' \frac{x}{|x|}\)|w_1|^2\,dx,
\end{align}
where $V=W''(H)$.
Since $|\varphi_A \zeta_A^{-2}V'|+A^{-1}|\chi''x|+2|\chi'|\lesssim e^{-\frac{2a}{10}|x|}$, we have the conclusion.
\end{proof}

\begin{lemma}\label{lem:v11remainder}
For arbitrary $c\in (0,1)$, we have
\begin{align}
\<\mathbf{J}\(   \sum_{\mathbf{m}\in \mathbf{R}_{\mathrm{min}}}\mathbf{z}^{\mathbf{m}}P_c\mathcal{R}_{\mathbf{m}} + \mathbf{R}_{\widetilde{\boldsymbol{\eta}}} \),S_A \widetilde{\boldsymbol{\eta}} \>\label{eq:v11remainder} \lesssim \delta^{1/3}\|w_1'\|_{L^2} + c\|\mathbf{w}\|_{L^2_{-\frac{a}{10}}}^2 + c^{-1}\sum_{\mathbf{m}\in \mathbf{R}_{\mathrm{min}}} |\mathbf{z}^{\mathbf{m}}|^2+\delta |\dot{\mathbf{z}}-\widetilde{\mathbf{z}}|^2.
\end{align}
Here, the implicit constant is independent of $c$.
\end{lemma}

\begin{proof}
Recall $\mathbf{R}_{  \widetilde{\boldsymbol{\eta}} }$ is given in \eqref{eq:Rtildeeta}.
First,
\begin{align*}
|\<\mathbf{J} P_c\boldsymbol{ \mathcal{R}}   [\mathbf{z}],S_A {\boldsymbol{\eta}} \>|\leq \| \zeta_A^{-1}S_A \mathbf{J} P_c \boldsymbol{ \mathcal{R}}[\mathbf{z}]\|_{L^2_{\frac{a}{10}}}\|\mathbf{w}\|_{L^2_{-\frac{a}{10}}}
\end{align*}
We have $\|\zeta_A^{-1}S_A\mathbf{J}P_c\|_{  \Sigma \to L^2_{\frac{a}{10}}}\lesssim 1$.
Therefore, by Proposition \ref{prop:rp} and Proposition \ref{prop:OrbStab},
\begin{align}\label{eq:v11remain01}
|\<\mathbf{J} P_c\boldsymbol{ \mathcal{R}}[\mathbf{z}],S_A \widetilde{\boldsymbol{\eta}} \>|\lesssim \sum_{\mathbf{m}\in \mathbf{R}_{\mathrm{min}}}|\mathbf{z}^{\mathbf{m}}| \|\mathbf{w}\|_{L^2_{-\frac{a}{10}}}\lesssim c^{-1}\sum_{\mathbf{m}\in \mathbf{R}_{\mathrm{min}}}|\mathbf{z}^{\mathbf{m}}|^2 +c \|\mathbf{w}\|_{L^2_{-\frac{a}{10}}}^2.
\end{align}
Next, since $\|P_c D\boldsymbol{\phi}[\mathbf{z}]\|_{\Sigma}\lesssim \delta$ by Proposition \ref{prop:OrbStab}, we have
\begin{align}\label{eq:v11remain02}
|\<\mathbf{J}P_cD\boldsymbol{\phi}[\mathbf{z}](\dot{\mathbf{z}}-\widetilde{\mathbf{z}}),S_A\widetilde{\boldsymbol{\eta}}\>|\lesssim \delta |\dot{\mathbf{z}}-\widetilde{\mathbf{z}}| \|\mathbf{w}\|_{L^2_{-\frac{a}{10}}}\lesssim
\delta |\dot{\mathbf{z}}-\widetilde{\mathbf{z}}|^2 +\delta \|\mathbf{w}\|_{L^2_{-\frac{a}{10}}}^2.
\end{align}
Using Lemma \ref{lem:R} as well as $\|\zeta_A^{-1}\|_{L^2_{\frac{a}{10}}\to L^2_{\frac{a}{10}-\frac{1}{A}}}$, $\|S_A P_c \mathbf{L}_1\|_{\Sigma \to L^2_{\frac{a}{10} }  }$, $\|\zeta_A^{-1}\|_{ L^{2}_{-\frac{a}{10}} \to \Sigma^*  }\lesssim 1$, we have
\begin{align}\label{eq:v11remain03}
|\<\mathbf{J}P_c \mathbf{L}_1(R[\mathbf{z}]-1)\widetilde{\boldsymbol{\eta}},S_A\widetilde{\boldsymbol{\eta}}\>|\lesssim \delta \|\mathbf{w}\|_{L^2_{-\frac{a}{10}}}^2.
\end{align}
For   $E_1 := \begin{pmatrix}
0 & 0 \\ 1 & 0
\end{pmatrix}  $  and $\Delta _{W ''}(\mathbf{z}):=  W''(\phi_1[\mathbf{z}])-W''(H)$, we consider
\begin{align} \label{v11remainder1}& \<P_c (\mathbf{L}[\mathbf{z}]- \mathbf{L}_1)\boldsymbol{\eta} ,S_A\widetilde{\boldsymbol{\eta}}\>  =
\<\mathbf{J} P_c  \Delta _{W ''}(\mathbf{z})  E_1  \boldsymbol{\eta} ,S_A\widetilde{\boldsymbol{\eta}}\>  \\& \nonumber =\<\mathbf{J}E_1     \Delta _{W ''}(\mathbf{z})   \widetilde{\boldsymbol{\eta}} ,S_A\widetilde{\boldsymbol{\eta}}\>-\< \mathbf{J}  P_d  \Delta _{W ''}(\mathbf{z})  E_1  \widetilde{\boldsymbol{\eta}} ,S_A\widetilde{\boldsymbol{\eta}}\>  \\&  - \<\mathbf{J} E_1    \Delta _{W ''}(\mathbf{z})    (R[\mathbf{z}]-1) \widetilde{\boldsymbol{\eta}} ,S_A\widetilde{\boldsymbol{\eta}}\>  + \< \mathbf{J}  P_d  \Delta _{W ''}(\mathbf{z})  E_1   (R[\mathbf{z}]-1) \widetilde{\boldsymbol{\eta}} ,S_A\widetilde{\boldsymbol{\eta}}\>.\nonumber
\end{align}
The most significant term in the right is the first. Since $\mathbf{J}E_1=\begin{pmatrix}
1 & 0 \\ 0 & 0
\end{pmatrix} $, it follows that $\mathbf{J}E_1S_A$ is skew--symmetric, so that
\begin{align*} & |
 \<\mathbf{J}E_1     \Delta _{W ''}(\mathbf{z})   \widetilde{\boldsymbol{\eta}} ,S_A\widetilde{\boldsymbol{\eta}}\> |=   |2^{-1}\<    [ S_A ,\Delta _{W ''}(\mathbf{z})]   \widetilde{  \eta }_1 ,\widetilde{  \eta }_1\> | \lesssim \delta  \| {w}_1\|_{L^2_{-\frac{a}{10}}} ^2 \le \delta  \| \mathbf{w }\|_{L^2_{-\frac{a}{10}}} ^2 .
\end{align*}
The other terms in  \eqref{v11remainder1} satisfy the same estimate. For example, if we consider the 2nd term in the r.h.s. of \eqref{v11remainder1}, we have
\begin{align*} & |
 \<\mathbf{J} S_A P_d  \Delta _{W ''}(\mathbf{z})  E_1  \widetilde{\boldsymbol{\eta}} ,\widetilde{\boldsymbol{\eta}}\> | \le \|  S_A P_d \| _{L^2 \to \Sigma  }
  \|  \Delta _{W ''}(\mathbf{z})\widetilde{\boldsymbol{\eta}} \| _{L^2}  \|  \widetilde{\boldsymbol{\eta}} \| _{\Sigma ^*}  \lesssim \delta  \|\mathbf{w}\|_{L^2_{-\frac{a}{10}}} ^2.
\end{align*}
The other terms in the r.h.s. of \eqref{v11remainder1} can bounded similarly, so that we can conclude
\begin{align}\label{eq:v11remain05}
\<\mathbf{J} P_c  \Delta _{W ''}(\mathbf{z})  E_1  \boldsymbol{\eta} ,S_A\widetilde{\boldsymbol{\eta}}\> \lesssim   \delta \|\mathbf{w}\|_{L^2_{-\frac{a}{10}}}^2.
\end{align}
For  $F_1$  and $\mathbf{F}$ defined in \eqref{eq:defF},    consider
\begin{align}\label{v11remainder2}
\<     \mathbf{J} P_c \mathbf{J} \mathbf{F} ,S_A\widetilde{\boldsymbol{\eta}}\>=
 -\<    \mathbf{F} ,S_A \boldsymbol{\eta} \> - \<\mathbf{J} P_d \mathbf{J} \mathbf{F} ,S_A \boldsymbol{\eta}\>
-\<\mathbf{J} P_c  \mathbf{J}\mathbf{F} ,S_A      (R[\mathbf{z}]-1) \widetilde{\boldsymbol{\eta}}\> .
\end{align}
Using the pointwise bound $|F_1|\lesssim |\eta_1|^2$, the fact that $P_d:\Sigma ^* \to \Sigma$ and \eqref{eq:lem:R},  we can bound the 2nd and the 3rd term by
\begin{align}&
|\<  \mathbf{F} ,     P_d  S_A \boldsymbol{\eta} \>|  \lesssim   \| F_1 \|_{\Sigma ^*}
\| P_d S_A \| _{\Sigma ^*\to \Sigma }  \|  {\eta}_1\|_{\Sigma ^*} \lesssim \| \eta _1 \| _{L^\infty} \|  {\eta}_1\|_{\Sigma ^*}^2 \lesssim  \delta \| {w}_1\|_{L^2_{-\frac{a}{10}}}^2 ,\label{eq:v11remain06}\\&
|\< \mathbf{F} ,   P_c S_A      (R[\mathbf{z}]-1) \widetilde{\boldsymbol{\eta}}\>| \lesssim  \|  F_1\|_{\Sigma ^*}   \| \mathbf{J} P_c S_A      (R[\mathbf{z}]-1)  \| _{\Sigma ^*\to \Sigma }   \|  \widetilde{\boldsymbol{\eta}} \|_{\Sigma ^*} \lesssim \delta  ^2 \|\mathbf{w}\|_{L^2_{-\frac{a}{10}}}^2.\label{eq:v11remain07}
\end{align}
Finally, for the 1st term of the r.h.s.\ of \eqref{v11remainder2},   we have
\begin{align*} &
 \<       \mathbf{F} ,S_A \boldsymbol{\eta}\>  = \< F_1 ,S_A\eta_1\>  = 2^{-1}\< F_1 \eta_1 , \zeta _A ^2\>
\\& - \<    W (\phi_1[\mathbf{z}]+\eta_1) -W (\phi_1[\mathbf{z}] )      -W'(\phi_1[\mathbf{z}])\eta _1-2^{-1}W''(\phi_1[\mathbf{z}])\eta_1 ^2, \zeta _A ^2\>
\\& - \<    W '(\phi_1[\mathbf{z}]+\eta_1)  -W '(\phi_1[\mathbf{z}] )      -W''(\phi_1[\mathbf{z}])\eta _1 -2^{-1}W'''(\phi_1[\mathbf{z}])\eta_1 ^2, \phi_1'[\mathbf{z}]  \varphi  _A  \> .
\end{align*}
So \begin{align*} &
 |\<      \mathbf{F} ,S_A \boldsymbol{\eta}\> |   \lesssim \int |\eta_1|^3\zeta_A^2\,dx
\end{align*}
and, by Lemma 2.7 of \cite{CM19SIMA}, we have
\begin{align}
|\<   \mathbf{F} ,S_A \boldsymbol{\eta}\>|\lesssim \delta^{1/3}\|w_1'\|_{L^2}^2.\label{eq:v11remain08}
\end{align}
By \eqref{eq:RFRemainder} and \eqref{eq:Rtildeeta} we have bounded all terms in the l.h.s.\ of \eqref{eq:v11remainder}.
Combining, \eqref{eq:v11remain01}, \eqref{eq:v11remain02}, \eqref{eq:v11remain03},  \eqref{eq:v11remain05},   \eqref{eq:v11remain06}, \eqref{eq:v11remain07} and \eqref{eq:v11remain08} we have the conclusion.
\end{proof}

Combining Lemmas \ref{lem:v11main} and \ref{lem:v11remainder} we obtain \eqref{eq:1stv1}.

We next prove \eqref{eq:1stv2}.
As \eqref{eq:1stv1}, we start from examine the contribution of the 1st term in the r.h.s.\ of \eqref{eq:diffI2}.

\begin{lemma}\label{lem:v12main}
We have
\begin{align*}
\<\mathbf{J}  \mathbf{L}_1\widetilde{\boldsymbol{\eta}} ,\zeta_A^2 \sigma_3\widetilde{\boldsymbol{\eta}} \>=\|\zeta_{\frac{10}{a}} w_2\|_{L^2}^2 + r,
\end{align*}
with $r$ satisfying
\begin{align}\label{est:v12r}
|r|\lesssim  \|w_1\|_{ \widetilde{\Sigma}}^2 .
\end{align}
\end{lemma}

\begin{proof}
We have
\begin{align*}
\<\mathbf{J}  \mathbf{L}_1\widetilde{\boldsymbol{\eta}} ,\zeta_A^2 \sigma_3\widetilde{\boldsymbol{\eta}} \>=\|\zeta_{\frac{10}{a}} w_2\|_{L^2}^2-\<L_1 \widetilde{\eta}_1,\zeta_A^2\widetilde{\eta}_1\>=\|\zeta_{\frac{10}{a}} w_2\|_{L^2}^2+r.
\end{align*}
The remainder term $r$ can be expanded as
\begin{align*}
-r=\<L_1\widetilde{\eta}_1,\zeta_{\widetilde{A}}^2\widetilde{\eta}_1\>
=\|(\zeta_{\frac{10}{a}}w_1)'\|_{L^2}^2+\int \(\((\log \zeta_{\widetilde{A}})'\)^2+W''(H)\)|\zeta_{\frac{10}{a}} w_1|^2\,dx.
\end{align*}
Thus, we have \eqref{est:v12r}.
\end{proof}

The contribution of the remaining terms in the r.h.s.\ of \eqref{eq:diffI2} can be bounded as follows.

\begin{lemma}\label{lem:v12remainder}
For arbitrary $c\in (0,1)$ we have
\begin{align}
\left|\<\mathbf{J}\(   \sum_{\mathbf{m}\in \mathbf{R}_{\mathrm{min}}}\mathbf{z}^{\mathbf{m}} \mathcal{R}_{\mathbf{m}} + \boldsymbol{\mathcal{R}}_{\widetilde{\boldsymbol{\eta}}} \),\zeta_{\widetilde{A}}^2 \sigma_3\widetilde{\boldsymbol{\eta}} \>\right|\lesssim c\|\mathbf{w}\|_{L^2_{-\frac{a}{10}}}^2+c^{-1}\sum_{\mathbf{m}\in \mathbf{R}_{\mathrm{min}}}|\mathbf{z}^{\mathbf{m}}|^2+\delta|\dot{\mathbf{z}}-\widetilde{\mathbf{z}}|^2,
\end{align}
where the implicit constant is independent of $c$.
\end{lemma}

\begin{proof}
The proof is similar to the proof of Lemma \ref{lem:v11remainder}.
Thus, we omit it.
\end{proof}

Combining Lemmas \ref{lem:v12main} and \ref{lem:v12remainder} and the fact $\|w_2\|_{L^2_{-\frac{a}{10}}}\leq \|\zeta_{\frac{10}{a}} w_2\|_{L^2}$, we have \eqref{eq:1stv2}.
This completes the proof of Lemma \ref{lem:1stvdiff}.

\section{Technical estimates }\label{sec:Preliminary}

The following lemmas are proved in \cite{CM2109.08108}, to which we refer for proofs.

\begin{lemma}\label{lem:equiv_rho}  Let $U\ge 0$ be a non--zero   potential $U\in L^1(\R , \R )$. Then
   there exists a constant $C  _{U}>0$ such that
for any    function $0\le W$  such that $ \<x\> W\in L^1(\R )$ then
 \begin{align}\label{eq:lem:rhoequiv}
& \<  W f,f\> \le    C_U \| \<x\> W  \| _{L^1(\R )}  \< (-\partial _x^2+ U )f,f\>.\end{align}
In particular,  for  $a>0$  the constant in the norm $\|\cdot\|_{\widetilde{\Sigma}}$ in \eqref{eq:norm_rho},   there exists a constant $C (a )>0$ such that
 \begin{align}\label{eq:lem:rhoequiva}
& \<  W f,f\> \le    C(a ) \| \<x\> W  \| _{L^1(\R )}   \|f\|_{\widetilde{\Sigma}}^2.\end{align}
\end{lemma}
\qed

\begin{lemma}\label{claim:l2boundII}  Consider a Schwartz function   $\mathcal{V}\in \mathcal{S}(\R, \C)$. Then, for any $L\in \N\cup \{ 0\}$ there exists a constant $C_L$ s.t. we have for all $\varepsilon \in (0,1]$ and for $L^{2,s}(\R ) $  is defined  in Definition \ref{def:spaces},
\begin{align}\label{eq2stestJ21II}
 \|    \< \im \varepsilon  \partial _x \> ^{-\widetilde{N}}  \left [   \mathcal{V} , \< \im \varepsilon  \partial _x \> ^{\widetilde{N}}\right  ] \| _{L ^{2,-L}(\R ) \to L ^{2, L}(\R )} \le C_L \varepsilon . \end{align}
 where $L^{2,s}(\R ) $  is defined  in Definition \ref{def:spaces}.

  \end{lemma}
\qed

  \begin{lemma}\label{lem:coer5} There exist constants $C_0$ and  $C _{\widetilde{N}}$ such that for  $\varepsilon>$ small enough
we have
\begin{align}\label{eq:coer51}
\|\mathcal{T} \|_{L^2\to L^2}\le C_0 \varepsilon^{-\widetilde{N}} \text{   and }  \|\mathcal{T}  \|_{\Sigma ^{\widetilde{N}} \to \Sigma^0}\le C_{\widetilde{N}}.
\end{align}
Furthermore, let $K_{\varepsilon}(x,y) \in \mathcal{D}'(\R \times \R)$  be the Schwartz kernel of  $\mathcal{T} $. Then, we have
\begin{align}\label{eq:coer52}
| K_{\varepsilon}(x,y)    | \le C_0  e^{-\frac{|x-y|}{2\varepsilon}}  \text{  for all $x,y$ with $|x-y|\ge 1$.}
\end{align} \end{lemma}
\qed

  \begin{lemma}\label{lem:coer3}
We have \begin{align}&  \| w _1\| _{L^2(|x|\le 2B^2)}\lesssim B^2  \| w  _1\| _{\widetilde{\Sigma}}    \text{  for any $w$,} \label{eq:estimates111}
\\& \| \xi _1\| _{\widetilde{\Sigma}}^2\lesssim      \< (-\partial ^2_x- 2^{-2}  \chi _{B^2} ^{2}xV' _{D})\xi _1, \xi _1\> \lesssim   \| \xi _1 \| _{\widetilde{\Sigma}}^2  \  \text{  for any $\xi$},  \label{eq:estimates113}\\&
\| v _1\| _{L^2(\R )} \lesssim \varepsilon ^{-{\widetilde{N}}} B^2  \|  w    _1   \| _{\widetilde{\Sigma}}  \ , \label{eq:estimates115}\\&  \| v '_1 \| _{L^2(\R )} \lesssim    \varepsilon ^{-{\widetilde{N}}}\|  w  _1   \| _{\widetilde{\Sigma}} \ ,  \label{eq:estimates116} \\& \|   \< x \> ^{-M}    v_1  \|_{H^1(\R )}  \lesssim   \|      \xi _1   \|_{\widetilde{\Sigma}}    +\varepsilon ^{-\widetilde{N}} \< B \> ^{-M+3}    \|    w  _1      \| _{\widetilde{\Sigma}}  \text{ for $M\in \N$, $M\ge 4$}.\label{eq2stestJ21-4}
\end{align}
\end{lemma}
\qed

\begin{lemma}\label{lem:coer6}
We have the formula
\begin{align}\label{eq:Tinverse}
P_c\(\chi_{B^2}\widetilde{\eta} _1\)=\prod_{j=1}^{\widetilde{N}}R _{L_1}(\widetilde{\lambda} ^{2}_j) P_c \mathcal{A} \<   \im \varepsilon\partial_x\> ^{\widetilde{N}} v _1 .
\end{align}
\end{lemma}
\qed

We next consider a number of results on linear theory.
\begin{lemma}\label{lem:generic} $\omega ^2$ is neither an eigenvalue nor a resonance for the operator $L_{D} $, that is, if    $L_{D}f= \omega ^2f $  for $f\in L^\infty (\R )$, then $f=0$.
 \end{lemma}
    \textit{Proof(sketch)}   If the statement is false,   there exists a nonzero and bounded solution of  $L_{D}f= \omega ^2f $. We can assume $f$ is real valued. Now, let $[a,b]$ be an interval where $\left . -xV'_{D}\right | _{ [a,b] }>0$ and let $\psi \in C_c ^{ \infty}((a,b), [0,+\infty ) )$ be a nonzero function such that $-x\( V'_{D}- \alpha  \psi  '\)>0$ in $[a,b]$  for all $\alpha \in [0,1]$.  Then it can be shown that for small $\alpha >0$ the operator $L_{D}-\alpha \psi $ has exactly one negative eigenvalue. But it is elementary to see that this is incompatible with the fact that $V _{D}-\lambda \psi$   is repulsive.

\qed

Notice that we can apply Komech--Kopylova
   \cite[Proposition 3.3]{KomKop111}  and conclude the following.

\begin{lemma}\label{lem:lemg9}
Let $ \Lambda$ be a
finite subset of $( 0,\infty) $
and let $S >5/2$. Then  there exists a fixed $ c(S,\Lambda)$ s.t.
  for every $t \ge 0$ and  $\lambda \in  \Lambda$
\begin{equation} \label{eq:lemg91}
\|   e^{  \mathbf{L}_{D}   t} R_{  \im \mathbf{L}_{D}  }^{+}( \lambda )\mathbf{f} \|_{ \boldsymbol{\mathcal{H}}^{1,-S}   (\R )}\le c(S,\Lambda)\langle  t\rangle ^{-\frac 32} \|  \mathbf{f}  \|_{ \boldsymbol{\mathcal{H}}^{1, S}  (\R  )}  \text{  for all $\mathbf{f}\in  \boldsymbol{\mathcal{H}}^{1, S}  (\R )$} .
\end{equation}
\end{lemma}  \qed

We have the following resolvent identity, see   Komech--Kopylova \cite[formula (3.6)]{KomKop10},
 \begin{align}\label{eq:resolv1}
 R _{\im  \mathbf{L}_{D}} ( \varsigma ) =  \begin{pmatrix}
	 \varsigma R _{L_D} (\varsigma ^2-\omega ^2 )
	& \im R _{L_D} (\varsigma ^2-\omega ^2 ) \\
	-\im \( 1+ \varsigma ^2  R _{L_D} (\varsigma ^2-\omega ^2 ) \)
	&\varsigma  R_{L_D} (\varsigma ^2-\omega ^2 )
\end{pmatrix} .
\end{align}

For the following see   Komech--Kopylova \cite[Sect. 3]{KomKop10}.

\begin{lemma}\label{lem:Lemma3.5}
For  any preassigned $r>0$ and for $S>1/2$ we have
\begin{align}\label{eq:Lemma3.51}
  \text{$R _{\im  \mathbf{L}_{\mathrm{Tr}}} ( \varsigma \pm \im \epsilon )\xrightarrow{\epsilon \to 0^+}R _{\im  \mathbf{L}_{D}} ^{\pm}( \varsigma   )$ in $L^\infty \(  (-\infty ,-r-\omega ] \cup [ r+\omega ,\infty ),  \mathcal{L}\( \boldsymbol{\mathcal{H}}^{1, S} , \boldsymbol{\mathcal{H}}^{1, -S} \) \) $}
\end{align}
 \end{lemma}  \qed

Combining Lemma \ref{lem:Lemma3.5}  with Lemma 8.5 \cite{CM2109.08108} we have the following.
\begin{lemma} \label{lem:LAP} For  $S>5/2$ and $\tau >1/2$ we have
 \begin{align}&   \label{eq:LAP1}   \sup _{\varsigma \in \R } \|   R ^{\pm }_{\im  \mathbf{L}_{D}}(\varsigma  ) \| _{ \boldsymbol{\mathcal{H}}^{1, \tau}\to  \boldsymbol{\mathcal{H}}^{1, -S}} <\infty  .
\end{align}
\end{lemma}
\proof    An uniform upper bound in   $|\varsigma|\ge 1+\omega$   holds by \eqref{eq:Lemma3.51}.  So now we focus on  $|\varsigma|\le 1+\omega$. By  \eqref{eq:resolv1}
it is enough to bound, for $j=0,1$
\begin{align}&   \label{eq:LAP1b}   \sup _{|\varsigma|\le 1+\omega } \|   \< x \> ^{-S} \partial _x ^{j} R ^{\pm }_{  L_D-\omega ^2 }(\varsigma  ^2 -\omega ^2 )  \< y \> ^{-\tau} \| _{     L^2(\R )\to  L^2(\R ) } <\infty  .
\end{align}
 which in turn  will be  a consequence of the following bound, for $j=0,1$,  for the integral kernel,
 \begin{align}  &   \sup _{|\varsigma|\le 1+\omega } \int _{\R ^2}  \< x \> ^{-2S}    | \partial _x ^{j} R ^{\pm }_{ L_D-\omega ^2 }(x,y,\varsigma  ^2 -\omega ^2  )| ^2  \< y \> ^{-2\tau} dx dy   <\infty  . \label{eq:LAP1c}
\end{align}
For $j=0$, in the $+$ case (case $-$ is similar),  this is proved in Lemma 8.5 \cite{CM2109.08108}, and for $j=1$  the same is true. Indeed, recall that, say for  $x<y$, with an analogous formula for $x>y$, \begin{align}&   R _{L_{D} -\omega ^2   }^{+}(x,y,\varsigma  ^2-\omega ^2  ) =     \frac{T(\sqrt{\varsigma  ^2-\omega ^2})}{2\im \sqrt{\varsigma  ^2-\omega ^2}}     e^{\im \sqrt{\varsigma  ^2-\omega ^2} (x-y)}       m_- (x, \sqrt{\varsigma  ^2-\omega ^2})   m_+ (y, \sqrt{\varsigma  ^2-\omega ^2})  ,    \label{eq:resolvKern} \end{align}
where  the  Jost functions  $f_{\pm } (x,k )=e^{\pm \im k x}m_{\pm } (x,k )$  solve $ \( - \Delta +  V _{D}-\omega ^2   \)u=k^2 u$ with\begin{align*} \lim _{x\to +\infty }   {m_{ + } (x,k )}  =1 =\lim _{x\to -\infty }  {m_{- } (x,k)}  .\end{align*}  These functions  satisfy, see Lemma 1 p. 130 \cite{DT79CPAM},\begin{align}  \label{eq:kernel2} &  |m_\pm(x, k )-1|\le  C _1 \langle  \max \{ 0,\mp x \}\rangle\langle k \rangle ^{-1}  \ ,
\\&   | \partial _x m_\pm(x, k ) |\le  C _1   \<    k\>  ^{-1}  ,\label{eq:dejost}  \end{align} while $T(k) =\alpha k (1+o(1))$  near $k=0$ for some $\alpha \in \R$    and $T(k) = 1+O(1/k) $ for $k\to \infty$ and $T\in C^0(\R )$, see Theorem 1 \cite{DT79CPAM}.

\noindent  Now,
\begin{align*}    \partial _x R ^{+}_{L_{D-\omega ^2 }   }(x,y,\varsigma  ^2-\omega ^2   )& =   \im \sqrt{\varsigma  ^2-\omega ^2}  R _{L_{D}-\omega ^2    }(x,y,\varsigma  ^2-\omega ^2  ) \\ &+
\frac{T(\sqrt{\varsigma  ^2-\omega ^2})}{2\im \sqrt{\varsigma  ^2-\omega ^2}}     e^{\im \sqrt{\varsigma  ^2-\omega ^2} (x-y)}       m_-' (x, \sqrt{\varsigma  ^2-\omega ^2})   m_+ (y, \sqrt{\varsigma  ^2-\omega ^2})  . \end{align*}
The first term on the right, by $|\varsigma|\le 1+\omega$ is essentially like the kernel \eqref{eq:resolvKern}, so the corresponding contribution to  \eqref{eq:LAP1c}
is like the case $j=0$.  It is easy to see, following the discussion in Lemma 8.5 \cite{CM2109.08108},
that the bound  of the last line is simpler, basically  because \eqref{eq:dejost}  is better than \eqref{eq:kernel2}.  \qed

\begin{remark}\label{rem:ass:repuls1}  Lemma \ref{lem:LAP} is essential to us  to get the key inequality \eqref{eq:essential1}.  Notice that  Lemma \ref{lem:LAP} is true under the   repulsivity hypothesis of Kowalczyk and Martel \cite{KM22}, which we have recalled in Remark \ref{rem:ass:repuls}, if we further assume that $\omega ^2$ is not a resonance for $L_D$. But if it has a resonance,   then the status of the Lemma \ref{lem:LAP} is unclear.   For $L_D=-\partial _x ^2+\omega ^2$,  by $ R_{-\partial_x ^2} (x,y,\varsigma ) = \frac{\im}{\sqrt{\varsigma}} e^{\im \sqrt{\varsigma }|x-y| }$ and by  a cancelation due to the odd functions,
we are reduced to   the following opposite  of \eqref{eq:LAP1c}
\begin{align*}  &   \sup _{|\varsigma|\le 1  } \int _{\R ^2}  \< x \> ^{-2S}    \left | \frac{1}{ \sqrt{\varsigma   }} \(  e^{\im \sqrt{\varsigma  } |x-y|}  -1  \) \right |^2  \< y \> ^{-2\tau} dx dy =+\infty .
\end{align*}
 Notice that this follows from the fact that the above is larger than
 \begin{align*}  &   \sup _{|\varsigma|\le 1  } \int _{ |x|\ll |y| \ll 1/\sqrt{|\varsigma|}}  \< x \> ^{-2S}    \left | \frac{1}{ \sqrt{\varsigma   }} \(  e^{\im \sqrt{\varsigma  } |x-y|}  -1  \) \right |^2  \< y \> ^{-2\tau} dx dy \\& \sim  \sup _{|\varsigma|\le 1  } \int _{ |x|\ll |y| \ll 1/\sqrt{|\varsigma|}}  \< x \> ^{-2S}  \< y \> ^{2-2\tau} dx dy  \sim \sup _{|\varsigma|\le 1  } |\varsigma| ^{ \frac{ 2\tau  -3}{2}} =+\infty  \text{  for   $\tau \in (1/2, 3/2)$}
\end{align*}
and is infinite also for $\tau =3/2$.
So, even though   a resonance  of $L_D$      involves  even functions, this still seems to affect  the estimates for the resolvent acting  only on odd functions. See also the resolvent expansions   in Lemma 2.2 in Murata \cite{murata}  or  Lemma 2.2 in Jensen and Kato \cite{JK}, which require increasing weights.

  \end{remark}

The following formulas   can be proved following    Mizumachi \cite[Lemma 4.5]{mizu08}, to which we refer for the proof.
 \begin{lemma} \label{lem:lemma11} Let  for $\mathbf{g}\in \mathcal{S}(\R \times \R , \C ^2 )$ \begin{align*}&  \mathbf{U}(t,x) = \frac{1}{\sqrt{2\pi} \im }\int _\R e^{-\im \lambda t}\( R ^{-}_{\im \mathbf{L}_{D}}(\lambda )+R ^{+}_{\im \mathbf{L}_{D}}(\lambda )   \) \mathcal{F}_{t}^{-1}\mathbf{g}( \lambda , \cdot ) d\lambda  ,
\end{align*}
where $\mathcal{F}_{t}^{-1}$  is the inverse Fourier transform in $t$. Then
 \begin{align} \label{eq:lemma11}2\int _0^t e^{  (t-t') \mathbf{L}_{D} }{\im \mathbf{L}_{D}}(t') dt'   &=  \mathbf{U}(t,x)  - \int  _{\R _-} e^{  (t-t')\mathbf{L}_{D} }{\im \mathbf{L}_{D}}(t') dt'
 \\& + \int  _{\R _+} e^{  (t-t') \mathbf{L}_{D} }{\im \mathbf{L}_{D}}(t') dt' .\nonumber
\end{align}
\end{lemma}
\qed

By repeating verbatim the proof of  \cite[Lemma 8.7]{CM2109.08108},   the  last two lemmas give us the following smoothing estimate.
\begin{lemma} \label{lem:smoothest} For  $S>5/2$ and $\tau >1/2$ there exists a constant $C(S,\tau )$ such that we have
 \begin{align}&   \label{eq:smoothest1}   \left \|   \int   _{0} ^{t   }e^{(t-t') \mathbf{L}_{D} }\mathbf{g}(t') dt' \right \| _{L^2( \R ,\boldsymbol{\mathcal{H}}^{1, -S})  } \le C(S,\tau ) \|  \mathbf{g } \| _{L^2( \R , \boldsymbol{\mathcal{H}}^{1, \tau}) }.
\end{align}
\end{lemma}\proof We repeat verbatim the proof of  \cite[Lemma 8.7]{CM2109.08108},  which in turn is taken from Mizumachi \cite{mizu08}, that is,  can use formula \eqref{eq:lemma11} and bound $\mathbf{U}$, with the bound on the last two terms in the right hand side of \eqref{eq:lemma11} similar. So we have, taking Fourier transform in $t$  and by Plancherel,
 \begin{align*}&
\| \mathbf{U}\|_{
L_{t}^2\boldsymbol{\mathcal{H}}^{1, -S}} \le   2 \sup _{\pm}\| R_{ \im \mathbf{L}_{D} }^\pm (\lambda )
   \widehat{ \mathbf{g}}(\lambda,\cdot )\|_{L_{\lambda  }^2 \boldsymbol{\mathcal{H}}^{1, -S}}   \le \\& \le
2\sup _{\pm}   \sup _{\lambda \in \R }
\|  R_{ \im \mathbf{L}_{D} }^\pm (\lambda )  \| _{ \boldsymbol{\mathcal{H}}^{1, \tau}  \to \boldsymbol{\mathcal{H}}^{1, -S}   } \|
     \widehat{\mathbf{g}} (\lambda,x) \|_{   L_{\lambda  }^2 \boldsymbol{\mathcal{H}}^{1, \tau}}\,
 \lesssim           \| \mathbf{g}\|_{L_{t }^2\boldsymbol{\mathcal{H}}^{1, \tau} }.
\end{align*}

\qed

\section{Second virial estimate, for $\boldsymbol{\xi}$}
We set
\begin{align} \label{eq:2nviqf}&
\mathcal{J}_1(\mathbf{v}):=\frac{1}{2}\<\mathbf{J}\mathbf{v},\widetilde{S}_B \mathbf{v}\>,\quad \mathcal{J}_2(\mathbf{v}):=\frac{1}{2}\<\mathbf{J}\mathbf{v},\(\chi_{B^2}\zeta_{\widetilde{B}}\)^2\sigma_3 \mathbf{v}\> \\& \text{where }\widetilde{B}^{-1}=B^{-1}+\frac{a}{10} \label{eq:2nviqf2}
\end{align}
 and
\begin{align}
\widetilde{S}_B:=\frac{\psi_B'}{2}+\psi_B\partial_x,\ \psi_B:=\chi_{B^2}^2\varphi_B. \label{eq:2nviqf3}
\end{align}

The main result of the section is the following.
\begin{lemma}\label{lem:2ndv}
\begin{align}
&\frac{d}{dt}\mathcal{J}_1(\mathbf{v})+\frac{1}{2}\|\xi_1\|_{\widetilde{\Sigma}}^2\lesssim  (c+\varepsilon)\|\xi_2\|_{L^2_{-\frac{a}{10}}}^2+c^{-1}
\sum_{\mathbf{m}\in \mathbf{R}_{\mathrm{min}}}|\mathbf{z}^{\mathbf{m}}|^2+\delta |\dot{\mathbf{z}}-\widetilde{\mathbf{z}}|^2
+\varepsilon \| \mathbf{w} \|_{\boldsymbol{ \widetilde{\Sigma}}}^2  ,\label{eq:2ndv1}\\
&-\frac{d}{dt}\mathcal{J}_2(\mathbf{v})+\frac{1}{2}\|\xi_2\|_{L^2_{-\frac{a}{10}}}^2\lesssim  \|\xi_1\|_{\widetilde{\Sigma}}^2+\sum_{\mathbf{m}\in \mathbf{R}_{\mathrm{min}}}|\mathbf{z}^{\mathbf{m}}|^2+\delta|\dot{\mathbf{z}}-\widetilde{\mathbf{z}}|^2
 +\varepsilon   \| \mathbf{w} \|_{\boldsymbol{ \widetilde{\Sigma}}}^2 .\label{eq:2ndv2}
\end{align}
\end{lemma}

\begin{proof}[Proof of Proposition \ref{prop:2ndvirial} assuming Lemma \ref{lem:2ndv}]
We have
\begin{align*}
|\mathcal{J}_1(\mathbf{v})|  &= | \< v_2,\widetilde{S}_B v_1\>  |      \le \|  {v}_2\| _{L^2}
|\widetilde{S}_B v_1 | \\& \lesssim  B \| \< \im \varepsilon \partial _x\> ^{-\widetilde{N}}\| _{L^2 \to L^2}   \| \< \im \varepsilon \partial _x\> ^{-\widetilde{N}}\| _{H^1 \to H^1} \|   \widetilde{ \eta}_2 \| _{L^2} \| \widetilde{ \eta}_1 \| _{H^1} \lesssim B \varepsilon ^{-2\widetilde{N}} \delta^2.
\end{align*}
Similarly, we have   $|\mathcal{J}_2(\mathbf{v})|\lesssim \varepsilon^{-2\widetilde{N} }\delta^2$.
Integrating \eqref{eq:2ndv1} and \eqref{eq:2ndv2} we have
\begin{align*}
\|\xi_1\|_{L^2\widetilde{\Sigma}}^2\lesssim& B\varepsilon^{-2\widetilde{N}}\delta^2+(c+\varepsilon)\|\xi_2\|_{L^2 L^2_{-\frac{a}{10}}}^2+c^{-1}\sum_{\mathbf{m}\in \mathbf{R}_{\mathrm{min}}}\|\mathbf{z}^{\mathbf{m}}\|_{L^2}^2+\delta\|\dot{\mathbf{z}}-\widetilde{\mathbf{z}}\|_{L^2}^2 +\varepsilon \|  \mathbf{w } \|_{L^2\boldsymbol{ \widetilde{\Sigma}}}^2 ,\\
\|\xi_2\|_{L^2L^2_{-\frac{a}{10}}}^2\lesssim& \varepsilon^{-2\widetilde{N}}\delta^2+\|\xi_1\|_{L^2\widetilde{\Sigma}}^2+\sum_{\mathbf{m}\in \mathbf{R}_{\mathrm{min}}}\|\mathbf{z}^{\mathbf{m}}\|_{L^2}^2+\delta\|\dot{\mathbf{z}}-\widetilde{\mathbf{z}}\|_{L^2}^2 +\varepsilon\|  \mathbf{w } \|_{L^2\boldsymbol{ \widetilde{\Sigma}}}^2 .
\end{align*}
Thus, as for the proof of Proposition \ref{prop:1stvirial}, we have the conclusion.
\end{proof}

As in \eqref{eq:diffI1} and \eqref{eq:diffI2}, we have
\begin{align}
\frac{d}{dt} \mathcal{J}_1(\mathbf{v})
&=\<\mathbf{J}\(  \mathbf{L}_{D}\mathbf{v} + \sum_{\mathbf{m}\in \mathbf{R}_{\mathrm{min}}}\mathbf{z}^{\mathbf{m}}\widetilde{ \mathcal{R}}_{\mathbf{m}} + \mathbf{R}_{\mathbf{v}} \),\widetilde{S}_B \mathbf{v} \>,\label{eq:diffJ1}\\
\frac{d}{dt} \mathcal{J}_2(\mathbf{v})
&=\<\mathbf{J}\(  \mathbf{L}_{D}\mathbf{v} + \sum_{\mathbf{m}\in \mathbf{R}_{\mathrm{min}}}\mathbf{z}^{\mathbf{m}}\widetilde{\mathcal{R}}_{\mathbf{m}} + \mathbf{R}_{\mathbf{v}} \),\chi_{B^2}^2\zeta_{\widetilde{B}}^2 \sigma_3\mathbf{v} \>.\label{eq:diffJ2}
\end{align}

\begin{lemma}
\label{lem:2ndv1m}
We have
\begin{align}
-\<\mathbf{J}  \mathbf{L}_{D}\mathbf{v} ,\widetilde{S}_B \mathbf{v} \>\geq \frac{1}{2}\<\(-\partial_x^2-\frac{1}{2}\chi_{B^2}^2 x V_{D}'\)\xi_1,\xi_1\> + B^{-1/2}O(\|\xi_1\|_{\widetilde{\Sigma}}^2+\|w_1\|_{\widetilde{\Sigma}}^2). \label{eq:2ndv1m}
\end{align}
\end{lemma}

\begin{proof}
Since $\widetilde{S}_B$ is skew-adjoint, we have the following
\begin{align}
-\<\mathbf{J}  \mathbf{L}_{D}\mathbf{v} ,\widetilde{S}_B \mathbf{v} \>=\<L_Dv_1, \widetilde{S}_B v_1\>,
\end{align}
where by    \cite[Lemma 6.1]{CM2109.08108} the very last term has the lower bound in the right hand side of \eqref{eq:2ndv1m}.
\end{proof}

\begin{lemma}\label{lem:Gv1}
We have
\begin{align}\label{eq:Gv11}
\left| \<\mathbf{z}^{\mathbf{m}} \mathbf{J}  \widetilde{\mathcal{R}}_{\mathbf{m}},\widetilde{S}_B \mathbf{v} \>\right| \lesssim |\mathbf{z}^{\mathbf{m}}|    \(
\| \boldsymbol{ \xi}  \|_{\boldsymbol{ \widetilde{\Sigma}}}   + e^{-B/2} \| \mathbf{w }   \| _{\boldsymbol{ \widetilde{\Sigma}} }    \)   .
\end{align}
\end{lemma}
\proof  We have
\begin{align}\label{eq:Gv12} & \<\mathbf{z}^{\mathbf{m}} \mathbf{J}  \widetilde{\mathcal{R}}_{\mathbf{m}},\widetilde{S}_B \mathbf{v} \> = \<\mathbf{z}^{\mathbf{m}}   \widetilde{\mathcal{R}}_{\mathbf{m}2},\widetilde{S}_B v_1 \> +  \<\mathbf{z}^{\mathbf{m}} \widetilde{S}_B  \widetilde{\mathcal{R}}_{\mathbf{m}1}, v_2 \> .
\end{align}
The following inequality is the content of Lemma 6.3 in \cite{CM2109.08108},
\begin{align*} & \left |  \<\mathbf{z}^{\mathbf{m}}   \widetilde{\mathcal{R}}_{\mathbf{m}2},\widetilde{S}_B v_1 \>  \right | \le  |\mathbf{z}^{\mathbf{m}}|    \(     \|\xi_1\|_{\widetilde{\Sigma}}   + e^{-B/2} \| w_1   \| _{\widetilde{\Sigma}   }         \) .
\end{align*}
We turn to the second term in the right in \eqref{eq:Gv12}. Using   $1=(1-\chi_{B^2}) + \chi_{B^2}$  we split in two and  bound separately the two terms.  Using $ \xi_2 = \chi_{B^2}  \zeta _{B}v_2$
the contribution from $ \chi_{B^2}$   is
\begin{align*} & \left | \<\mathbf{z}^{\mathbf{m}}   e ^{\frac{a}{10}   |x|} \zeta _{B}^{-1} \widetilde{S}_B  \widetilde{\mathcal{R}}_{\mathbf{m}1}, e ^{-\frac{a}{10}   |x|}  \xi_2 \>  \right | \le  |\mathbf{z}^{\mathbf{m}}|      \|   \xi_2 \| _{L^2 _{-\frac{a}{10}}}
\|        \widetilde{S}_B  \widetilde{\mathcal{R}}_{\mathbf{m}1} \| _{L^2_{ \frac{a}{10} + B ^{-1}}}    .
\end{align*}
We show now that the last factor is $\lesssim 1$. The term we need to bound is
\begin{align*} &        e ^{\frac{a}{10}   |x|} \zeta _{B}^{-1}       (\chi_{B^2}^2\varphi_B) '   \<\im \varepsilon \partial_x\>^{-\widetilde{N} }\mathcal{A}^*   \chi_{B^2} {\mathcal{R}}_{\mathbf{m}1} +2 e ^{\frac{a}{10}   |x|} \zeta _{B}^{-1}  \chi_{B^2}^2\varphi_B            \<\im \varepsilon \partial_x\>^{-\widetilde{N} }\partial_x  \mathcal{A}^*   \chi_{B^2} {\mathcal{R}}_{\mathbf{m}1}.
\end{align*}
 We bound only the second term, since the first can be bounded similarly and in fact is smaller.  Let us set $f:=\partial_x  \mathcal{A}^*   \chi_{B^2} {G}_{\mathbf{m}1}$.  We have
 \begin{align} &                \<\im \varepsilon \partial_x\>^{-\widetilde{N} } f  (x) = \int  f(y)    e ^{\frac{a}{5}   |y|}  I  dy \text{ , where }     I:=\int \frac{e ^{-\frac{a}{5}   |y|} e ^{\im (x-y)  ( \tau _1+\im \tau _2)}}{ \( 1+ \varepsilon ^2 \tau _1^2 -\varepsilon ^2 \tau _2^2  + 2\im \varepsilon ^2 \tau _1   \tau _2\) ^{ \widetilde{N}/2}   } d \tau _1 \label{eq:Gv13}
\end{align}
 is a generalized  integral in $\tau _1$,   well defined, using integration by parts,  also  for $ \widetilde{N}=1$, when it is not absolutely convergent. $I$  is constant in $ |\tau _2|< \varepsilon ^{-1} $. Then, for
 $  \tau _2 = 2 ^{-1/2} \varepsilon ^{-1}   \sign (x-y)$ we have
 \begin{align*} &              I =      e ^{-\frac{a}{5}   |y|}     e ^{- \frac{|x-y|}{2\varepsilon}}    II\text{ , where } II:=  \int \frac{ e ^{\im (x-y)   \tau _1 }}{ \( 1/2+ \varepsilon ^2 \tau _1^2   + \sqrt{2} \im \varepsilon   \tau _1   \sign (x-y)\) ^{ \widetilde{N}/2}   } d \tau _1.
\end{align*}
 For $ \widetilde{N}>1$  we have $|II|\lesssim \varepsilon ^{-1}$. Standard arguments show  $|II|\lesssim  \varepsilon ^{-1} \log \(  2+ \varepsilon /|x-y| \)     $    for $ \widetilde{N}=1$.
 Since $ e ^{-\frac{a}{5}   |y|}     e ^{- \frac{|x-y|}{4\varepsilon}} \le  e ^{-\frac{a}{5}   |x|}$,  we conclude
 \begin{align*} &            \left |     \<\im \varepsilon \partial_x\>^{-\widetilde{N} } f  (x)    \right | \lesssim   e ^{-\frac{a}{5}   |x|}  \int     \varepsilon ^{-1}   K _{\widetilde{N}}\(   \frac{x-y }{\varepsilon}   \)
  | f(y)   |  e ^{\frac{a}{5}   |y|} dy \text{  with }  K _{\widetilde{N}}\(   x   \) =  e ^{- \frac{|x|}{4 }}     \log  ^{\sigma _{\widetilde{N}}}\(  2+   /|x| \)  ,
\end{align*}
where $\sigma _{\widetilde{N}} =0$  for $\widetilde{N}>1$  and   $\sigma _{\widetilde{N}} =1$  for $\widetilde{N}=1$.
Then
  \begin{align*} &      \|        \chi_{B^2}^2\varphi_B            \<\im \varepsilon \partial_x\>^{-\widetilde{N} }f  \| _{L^2_{ \frac{a}{10} + B ^{-1}}} \lesssim \|    e ^{\frac{a}{10}   |x|} \zeta _{B}^{-1}   \varphi_B     e ^{-\frac{a}{5}   |x|} \| _{L^\infty} \| K _{\widetilde{N}} \| _{L^1}  \|  e ^{\frac{a}{5}   |x|}\partial_x  \mathcal{A}^*   \chi_{B^2} \mathcal{R}_{\mathbf{m}1} \| _{L^2_{ \frac{a}{10}  }} \lesssim 1.
\end{align*}
To finish the proof we consider the following, which completes the proof,
\begin{align*}     \left |    \<\mathbf{z}^{\mathbf{m}}  e ^{ \frac{a}{10}} \zeta _{A}^{-1}   (1-\chi_{B^2})   \widetilde{S}_B  \widetilde{\mathcal{R}}_{\mathbf{m}1},   e ^{-\frac{a}{10}   |y|}\zeta _{A}   v_2 \>    \right  | &  \lesssim
 | \mathbf{z}^{\mathbf{m}}  | \  \| w_2  \| _{L^2_{-\frac{a}{10} }}     \|      (1-\chi_{B^2})   \widetilde{S}_B  \widetilde{\mathcal{R}}_{\mathbf{m}1}  \| _{L^2_{ \frac{a}{10} + A ^{-1}}}  \\& \le e^{-B}   | \mathbf{z}^{\mathbf{m}}  | \  \| w_2  \| _{L^2_{-\frac{a}{10} }}  .
\end{align*}

 \qed

\begin{lemma}
\label{lem:2ndv1r}
We have
\begin{align}
\left |\<\mathbf{J}  \mathbf{R}_{\mathbf{v}}  ,\widetilde{S}_B \mathbf{v} \>   \right |
\le o _
\varepsilon (1)& \left [
  \| \boldsymbol{ \xi}  \|_{\boldsymbol{ \widetilde{\Sigma}}}^2+\|  \mathbf{w}  \|_{\boldsymbol{ \widetilde{\Sigma}}}^2    +
	|\dot{\mathbf{z}}-\widetilde{\mathbf{z}}|^2 + \sum_{\mathbf{m}\in \mathbf{R}_{\mathrm{min}}} |\mathbf{z}^{\mathbf{m}}| ^2   \right ] .
 \label{eq:v21r0}
\end{align}

\end{lemma}

\begin{proof}
As the proof of Lemma \ref{lem:v11remainder} we estimate each term.
First,
\begin{align}
&\<\mathbf{J}\begin{pmatrix} 0 \\ -\mathcal{T}(2\chi_{B^2}'\partial_x +\chi_{B^2}'')\widetilde{\eta}_1+\<\im \varepsilon\partial_x\>^{-\widetilde{N}}[V_{D},\<\im \varepsilon\partial_x\>^{\widetilde{N}}]v_1
\end{pmatrix},\widetilde{S}_B \mathbf{v} \>\label{eq:v21r1}\\&=
\<-\mathcal{T}(2\chi_{B^2}'\partial_x +\chi_{B^2}'')\widetilde{\eta}_1+\<\im \varepsilon\partial_x\>^{-\widetilde{N}}[V_{D},\<\im \varepsilon\partial_x\>^{\widetilde{N}}]v_1, \widetilde{S}_Bv_1\>\nonumber
\end{align}
The last term is bounded in  (6.19) and (6.20) of \cite{CM2109.08108},
and in particular  we have
\begin{align}\label{eq:v21r2}
|\eqref{eq:v21r1}|\lesssim (\varepsilon+\varepsilon^{-\widetilde{N}}B^{-1})\(\|\xi_1\|_{\widetilde{\Sigma}}^2+\|w_1\|_{\widetilde{\Sigma}}^2\).
\end{align}
We next, we have
\begin{align*} & \left |
\<\mathbf{J}\mathcal{T}\chi_{B^2}P_c\boldsymbol{\mathcal{R}}_1[\mathbf{z}],\widetilde{S}_B\mathbf{v}\> \right |   +    \left |
    \<
	\mathbf{J}\mathcal{T}\chi_{B^2}P_cD_{\mathbf{z}}\boldsymbol{\phi}[\mathbf{z}](\dot{\mathbf{z}}-\widetilde{\mathbf{z}}),
	\widetilde{S}_B\mathbf{v}
\>       \right |     \\&  \lesssim \delta   \(   |\dot{\mathbf{z}}-\widetilde{\mathbf{z}}| +  \sum_{\mathbf{m}\in  \mathbf{R}_{\mathrm{min}}} |\mathbf{z}^{\mathbf{m}}|   \)   \(     \| \boldsymbol{ \xi}  \|_{\boldsymbol{ \widetilde{\Sigma}}} + e^{-B/2}  \| \mathbf{w }\|_{\boldsymbol{ \widetilde{\Sigma}}}       \) ,
\end{align*}
because the first term in the left    can be treated   like \eqref{eq:Gv11}, except that it is   smaller  because of
the bound  \eqref{eq:R1FRemainder}  on $\boldsymbol{\mathcal{R}}_1[\mathbf{z}]$, and a similar argument holds for the second term on the left, where additionally we use $  \|    P_cD_{\mathbf{z}}\boldsymbol{\phi}[\mathbf{z}]    \| _{\C   ^{ {N}}   \to \Sigma }\lesssim |\mathbf{z}|.$

\noindent Next, proceeding as above
\begin{align}
 \left |\<
	\mathbf{J}\mathcal{T}\chi_{B^2}P_c  \mathbf{L}_1(R[\mathbf{z}]-1)\widetilde{\boldsymbol{\eta}},
	\widetilde{S}_B\mathbf{v}
\>  \right | &\lesssim
\|
	\widetilde{S}_B  \mathcal{T}\chi_{B^2}P_c \mathbf{L}_1(R[\mathbf{z}]-1)\widetilde{\boldsymbol{\eta}}
\|_{L^2_{\frac{a}{10}+B^{-1}}}
\nonumber
\\& \times
 \(     \| \boldsymbol{ \xi}  \|_{\boldsymbol{ \widetilde{\Sigma}}} + e^{-B/2}  \| \mathbf{w }\|_{\boldsymbol{ \widetilde{\Sigma}}}       \) .
\label{eq:v21r7}
\end{align}
By
\begin{align}
&\|\mathcal{T}\chi_{B^2} P_c \widetilde{S}_B[\mathbf{z}]\|_{\Sigma^{\widetilde{N}+1}\to L^2_{\frac{a}{10}+B^{-1}}}\lesssim \varepsilon^{-\widetilde{N}},\\
& \|R[\mathbf{z}]-1\|_{\Sigma^* \to \Sigma^{\widetilde{N}+1}  }\lesssim \delta,\\
& \|\widetilde{\boldsymbol{\eta}}\|_{\Sigma^*}\lesssim \|\mathbf{w}\|_{L^2_{-\frac{a}{10}}},
\end{align}
we conclude
\begin{align*}
\|
	\widetilde{S}_B  \mathcal{T}\chi_{B^2}P_c \mathbf{L}_1(R[\mathbf{z}]-1)\widetilde{\boldsymbol{\eta}}
\|_{L^2_{\frac{a}{10}+B^{-1}}}  \lesssim \delta \varepsilon^{-\widetilde{N}} \|\mathbf{w}\|_{L^2_{-\frac{a}{10}}} .
\end{align*}
Next, following the notation in Lemma \ref{lem:v11remainder},   we consider
\begin{align}
\<\mathbf{J} \mathcal{T}\chi_{B^2} P_c  (L[\mathbf{z}]-\mathbf{L}_1)  \boldsymbol{\eta} ,\widetilde{S}_B\mathbf{v}\> =\<    \mathcal{T}\chi_{B^2}   \Delta _{W ''}(\mathbf{z})    \widetilde{\eta} _1 ,\widetilde{S}_B {v}_1\>-\< \mathbf{J} \mathcal{T}\chi_{B^2} P_d (L[\mathbf{z}]-\mathbf{L}_1)  \widetilde{\boldsymbol{\eta}} ,\widetilde{S}_B\mathbf{v}\>  \nonumber\\  - \<\mathbf{J}\mathcal{T}\chi_{B^2} (L[\mathbf{z}]-\mathbf{L}_1)     (R[\mathbf{z}]-1) \widetilde{\boldsymbol{\eta}} ,\widetilde{S}_B\mathbf{v}\>  + \< \mathbf{J}\mathcal{T}\chi_{B^2}  P_d (L[\mathbf{z}]-\mathbf{L}_1)   (R[\mathbf{z}]-1) \widetilde{\boldsymbol{\eta}} ,\widetilde{S}_B\mathbf{v}\>.\label{eq:v21r7--}
\end{align}
Like in Lemma \ref{lem:v11remainder},  the most significant term in the right is the first, which for brevity is the only one we bound explicitly, since the other ones are simpler.
We have
\begin{align*}
  | \<    \mathcal{T}\chi_{B^2}   \Delta _{W ''}(\mathbf{z})    \widetilde{\eta} _1 ,\widetilde{S}_B {v}_1\> |\le  \varepsilon ^{-\widetilde{N}}  \|  \chi_{B^2}   \Delta _{W ''}(\mathbf{z})     \widetilde{\eta} _1 \| _{L^1} \|  \widetilde{S}_Bv_1  \| _{L^2}  \\ \lesssim  \varepsilon ^{-\widetilde{N}} \| e^{\frac{a}{10}  |x|} \Delta _{W ''}(\mathbf{z}) \| _{L^\infty} \| w_1\| _{L^2_{-\frac{a}{10}  }}  \(    \varepsilon ^{-\widetilde{N}}B\| w _1 \| _{\widetilde{\Sigma}}  + B\| \xi  _1   \| _{\widetilde{\Sigma}} \) \\  \lesssim \varepsilon ^{-2\widetilde{N}} B \delta   \| w_1\| _{L^2_{-\frac{a}{10}  }}    \(     \| w _1 \| _{\widetilde{\Sigma}}  + | \xi  _1   \| _{\widetilde{\Sigma}} \) ,
\end{align*}
where we  used,
 \begin{align}
 \|  \widetilde{S}_Bv_1  \| _{L^2} \lesssim    \varepsilon ^{-N}B\| w _1 \| _{\widetilde{\Sigma}}  + B\| \xi  _1   \| _{\widetilde{\Sigma}} .  \label{eq:sbv}
\end{align}
The proof of \eqref{eq:sbv} is   in \cite{KMM3} and for completeness we write the proof in \cite{CM2109.08108}. By  \eqref{eq:estimates116} and $\|  \psi _{B}   \|  _{L^\infty} \lesssim B$
\begin{align*}
 \|   \widetilde{S}_Bv_1  \| _{L^2}\lesssim \| \psi _{B}' v  _1 \|  _{L^2} + \| \psi _{B}  v '_1  \|  _{L^2} \lesssim \| \psi _{B}' v_1  \|  _{L^2} +  \varepsilon ^{-N}B\| w _1   \| _{\widetilde{\Sigma}}.
\end{align*}
Next, we have
\begin{align}&
| \psi ' _B| = |2\chi ' _{B^2} \chi _{B^2} \varphi _B + \chi _{B^2}^2 \zeta ^2 _B|\lesssim B ^{-1}\chi _{B^2}+  \chi _{B^2}^2 \zeta ^2 _B.\label{eq2stestJ33}
\end{align}
Then
\begin{align*}
 B ^{-1} \|   v _1  \|  _{L^2}\lesssim B  \varepsilon ^{-N}B\| w _1   \| _{\widetilde{\Sigma}}
\end{align*}
by \eqref{eq:estimates115}. By Lemma \ref{lem:equiv_rho} we have
\begin{align*}
   \|   \chi _{B^2}^2 \zeta ^2 _B v_1   \|  _{L^2} =    \|   \chi _{B^2}  \zeta  _B \xi  _1 \|  _{L^2}       \lesssim
    \sqrt{\| \<   x \> \chi _{B^2}  \zeta  _B \| _{L^1} }\| \xi    \| _{\widetilde{\Sigma}} \sim B \| \xi   _1 \| _{\widetilde{\Sigma}}
\end{align*}
and, finally  the following by \eqref{eq:estimates116}, which completes the proof of \eqref{eq:sbv},
\begin{align*}
   \|   \chi _{B^2}^2 \varphi _B   v  '_1  \|  _{L^2} \lesssim B   \|      v  '_1  \|  _{L^2}    \lesssim  B \varepsilon ^{-N} \| w_1  \| _{\widetilde{\Sigma}}.
\end{align*}
Next, following again the notation in Lemma \ref{lem:v11remainder},   we consider
\begin{align}
\<     \mathbf{J}    \mathcal{T}\chi_{B^2}  P_c\mathbf{ JF} ,\widetilde{S}_B \mathbf{v}\>=
 \<     \mathcal{T}\chi_{B^2}  \mathbf{F} ,\widetilde{S}_B  \mathbf{v}  \> - \<\mathbf{J} \mathcal{T}\chi_{B^2} P_d \mathbf{JF },\widetilde{S}_B \mathbf{v}\> .\label{eq:2vinl}
\end{align}
The main term in the right is the first, which by \eqref{eq:sbv}  can be treated as
\begin{align*}
 |
 \<     \mathcal{T}\chi_{B^2} F_1 ,\widetilde{S}_B   {v}  _1 \> |   \le \|  \mathcal{T}\chi_{B^2} F_1 \| _{L^2}\|  \widetilde{S}_B   {v}  _1\| _{L^2} \lesssim    \varepsilon ^{-\widetilde{N}} \|   \chi_{B^2} \eta _1 ^2 \| _{L^2}  \(  \varepsilon ^{-N}B\| w _1 \| _{\widetilde{\Sigma}}  + B\| \xi  _1   \| _{\widetilde{\Sigma}} \)  \\  \lesssim \varepsilon ^{-2\widetilde{N}} B \| \eta _1  \| _{H^1}
 \(  \|   \chi_{B^2} \widetilde{\eta} _1  \| _{L^2}  +  \|    ( R[\mathbf{z}]-1)  \widetilde{\eta} _1  \| _{L^2}   \) \(  \| w _1 \| _{\widetilde{\Sigma}}  +  \| \xi  _1   \| _{\widetilde{\Sigma}} \)
\\   \lesssim  \varepsilon ^{-2\widetilde{N}} B \delta
 \(  \|  w _1  \| _{L^2(|x|\le 2B^2)}  +\delta        \| w  _1   \| _{\widetilde{\Sigma}}    \) \(  \| w _1 \| _{\widetilde{\Sigma}}  +  \| \xi  _1   \| _{\widetilde{\Sigma}} \)  \\  \lesssim     \varepsilon ^{-2\widetilde{N}} B ^3 \delta   \| w _1 \| _{\widetilde{\Sigma}}
  \(   \| w _1 \| _{\widetilde{\Sigma}} +  \| \xi  _1   \| _{\widetilde{\Sigma}} \)   ,
\end{align*}
where we used \eqref{eq:estimates111},  Lemma \ref{lem:R}.

\noindent Turning to the second term in the right of \eqref{eq:2vinl}, it is bounded from above by
\begin{align*}
 \| \widetilde{S}_B  \mathcal{T}\chi_{B^2} P_d\mathbf{J} \mathbf{F} \| _{\Sigma} \|  \mathbf{v}  \| _{\Sigma ^*}\lesssim \| F_1\| _{\Sigma ^*} \|  \mathbf{w}  \| _{ L^2_{-\frac{a}{10}}} \lesssim  \| \eta _1  \| _{H^1}  \| \eta _1  \| _{\Sigma ^*} \|  \mathbf{w}  \| _{ L^2_{-\frac{a}{10}}}\lesssim \delta  \|  \mathbf{w}  \| _{ L^2_{-\frac{a}{10}}}^2.
\end{align*}

\end{proof}

\begin{lemma}
\label{lem:2ndv2m}
For the   $\widetilde{B}$  defined in \eqref{eq:2nviqf2}, we have
\begin{align}
\label{eq:v22m0}
\<
	\mathbf{J}
	\(
		\mathbf{L}_{D}\mathbf{v}
	\),
	\chi_{B^2}^2
	\zeta_{\widetilde{B}}^2
	\sigma_3\mathbf{v}
\>
=
	\int \zeta_{\frac{10}{a}}^2|\xi_2|^2\,dx
-
	\<
		(L_D+\omega ^2) v_1,
		\chi_{B^2}^2 \zeta_{\widetilde{B}}^2 v_1
	\> ,
\end{align}
with
\begin{align}
\label{eq:v22m1}
|\<
		(L_D+\omega ^2)v_1,
		\chi_{B^2} \zeta_{\widetilde{B}}^2 v_1
\>|
\lesssim
	\|\xi_1\|_{\widetilde{\Sigma}}
	\(
		\|\xi_1\|_{\widetilde{\Sigma}}
			+
		e^{-B}\varepsilon^{\widetilde{N}}\|w_1\|_{\widetilde{\Sigma}}.
	\)
\end{align}
\end{lemma}

\begin{proof}
Formula \eqref{eq:v22m0} follows from direct computation.
We prove \eqref{eq:v22m1}.
First,
\begin{align}
\label{eq:v22m02}
\<
		(L_D+\omega ^2)v_1,
		\chi_{B^2} \zeta_{\widetilde{B}}^2 v_1
\>
=
	\<
		-v_1'',
		\chi_{B^2}\zeta_B\zeta_{\frac{10}{a}}^2\xi_1
	\>
		+
	\<
		(V_{D}+\omega ^2)\xi_1,
		\zeta_{\frac{10}{a}}^2 \xi_1
	\>.
\end{align}
The 2nd term in r.h.s.\ of \eqref{eq:v22m02} can be bounded by
\begin{align}
\label{eq:v22m03}
|\<
		(V_{D}+\omega ^2)\xi_1,
		\zeta_{\frac{10}{a}}^2 \xi_1
\>|
	\lesssim
\| \xi_1 \|_{L^2_{-\frac{a}{10}}}^2.
\end{align}
For the 1st term of r.h.s.\ of \eqref{eq:v22m02} we write
\begin{align} \nonumber &
\<
	-v_1'',
	\chi_{B^2} \zeta_B \zeta_{\frac{10}{a}}^2\xi_1
\>
	 =
-\<
	v_1,
	(
		\chi_{B^2} \zeta_B \zeta_{\frac{10}{a}}^2\xi_1
	)''
\>
	\\&  \nonumber =
-\<
	v_1,
	(
		\chi_{B^2} \zeta_B \zeta_{\frac{10}{a}}^2
	)''
	\xi_1
		+
	2(
		\chi_{B^2} \zeta_B \zeta_{\frac{10}{a}}^2
	)'
	\xi_1'
		+
	\chi_{B^2}  \zeta_B  \zeta_{\frac{10}{a}}^2
	\xi_1''
\>\\&
	=
-\<
	v_1,
	 \chi_{B^2}  '' \zeta_B  \zeta_{\frac{10}{a}}^2 \xi_1
\>
	-
\<
	v_1,
	2 \chi_{B^2} '
		(\zeta_B \zeta_{\frac{10}{a}}
	)'
	\xi_1
\>
	-
\<
	v_1,
	\chi_{B^2}
	(
		\zeta_B \zeta_{\frac{10}{a}}^2
	)''
	\xi_1
\>
	\label{eq:v22m04}\\&\quad-
\<
	v_1,
	2\chi_{B^2}'
	\zeta_B \zeta_{\frac{10}{a}}^2
	\xi_1'
\>
	-
\<
	v_1,
	2\chi_{B^2}
	(
		\zeta_B \zeta_{\frac{10}{a}}^2
	)'
	\xi_1'
\>
	-
\<
	v_1,
	\chi_{B^2} \zeta_B \zeta_{\frac{10}{a}}^2 \xi_1''v
\> . \label{eq:v22m04-}
\end{align}
For the 1st term of line \eqref{eq:v22m04}, we have
\begin{align}
\nonumber
-\<
	v_1,
	\chi_{B^2}'' \zeta_B \zeta_{\frac{10}{a}}^2 \xi_1
\>
	&=
-\<
	\zeta_B v_1,
	\chi_{B^2}'' \zeta_B \zeta_{\frac{10}{a}}^2 \xi_1
\>
	\\&=
-\<
	\xi_1,
	\chi_{B^2}'' \zeta_B \zeta_{\frac{10}{a}}^2 \xi_1
\>
	-
\<
	(1-\chi_{B^2})\zeta_B \mathcal{T} \chi_{B^2} \zeta_A^{-1} w_1,
	\chi_{B^2}'' \zeta_B \zeta_{\frac{10}{a}}^2 \xi_1
\> . \label{eq:v22m05}
\end{align}
For the 1st term of line \eqref{eq:v22m05},
\begin{align}
\label{eq:v22m06}
|\<
	\xi_1,
	\chi_{B^2}'' \zeta_B \zeta_{\frac{10}{a}}^2 \xi_1
\>|
	\lesssim
B^{-4}
\| \xi_1 \|_{L^2_{-\frac{a}{10}}}^2.
\end{align}
For the 2nd term of line \eqref{eq:v22m05},
\begin{align}
\label{eq:v22m07}
|\<
	(1-\chi_{B^2})\zeta_B \mathcal{T} \chi_{B^2} \zeta_A^{-1} w_1,
	\chi_{B^2}'' \zeta_B \zeta_{\frac{10}{a}}^2 \xi_1
\>|
	\lesssim
e^{-B} \varepsilon^{\widetilde{N}} B^{-4}
\|w_1\|_{L^2_{-\frac{a}{10}}}
 \|\xi_1\|_{L^2_{-\frac{a}{10}}}
\end{align}
Combining \eqref{eq:v22m06} and \eqref{eq:v22m07} we have
\begin{align}
\label{eq:v22m071}
|\<
	v_1,
	\chi_{B^2}'' \zeta_B \zeta_{\frac{10}{a}}^2 \xi_1
\>|
	\lesssim
B^{-4} \| \xi_1 \|_{L^2_{-\frac{a}{10}}}^2
	+
e^{-B} \varepsilon^{\widetilde{N}} B^{-4}
\|w_1\|_{L^2_{-\frac{a}{10}}}
 \|\xi_1\|_{L^2_{-\frac{a}{10}}}.
\end{align}
Next, for the 2nd term of   line \eqref{eq:v22m04} we have
\begin{align}
\nonumber
&-\<
	v_1,
	2\chi_{B^2}'
	(
		\zeta_B\zeta_{\frac{10}{a}}
	)'
	\xi_1
\>
	=
-\<
	\zeta_Bv_1,
	2\chi_{B^2}' \zeta_B^{-1}
	(
		\zeta_B\zeta_{\frac{10}{a}}
	)'
	\xi_1
\>
	\\&=
-\<
	\xi_1,
	2\chi_{B^2}' \zeta_B^{-1}
	(
		\zeta_B\zeta_{\frac{10}{a}}
	)'
	\xi_1
\>	
	-
\<
	(1-\chi_{B^2})\zeta_B \mathcal{T} \chi_{B^2} \zeta_A^{-1} w_1,
	2\chi_{B^2}' \zeta_B^{-1}
	(
		\zeta_B\zeta_{\frac{10}{a}}
	)'
	\xi_1
\>	 . \label{eq:v22m08}
\end{align}
The 1st term of  line \eqref{eq:v22m08} can be bounded as
\begin{align*}
|-\<
	\xi_1,
	2\chi_{B^2}' \zeta_B^{-1}
	(
		\zeta_B\zeta_{\frac{10}{a}}
	)'
	\xi_1
\>|
	\lesssim
B^{-2}
\|\xi_1 \|_{L^2_{-\frac{a}{10}}}^2
\end{align*}
and the 2nd term of line  \eqref{eq:v22m08} can be bounded as
\begin{align*}
|\<
	(1-\chi_{B^2})\zeta_B \mathcal{T} \chi_{B^2} \zeta_A^{-1} w_1,
	2\chi_{B^2}' \zeta_B^{-1}
	(
		\zeta_B\zeta_{\frac{10}{a}}
	)'
	\xi_1
\>|
	\lesssim
e^{-B}\varepsilon^{-\widetilde{N}} B^{-2}
\|w_1\|_{L^2_{-\frac{a}{10}}}
\|\xi_1 \|_{L^2_{-\frac{a}{10}}}.
\end{align*}
Thus, we have
\begin{align}
\label{eq:v22m081}
|\<
	v_1,
	2\chi_{B^2}'
	(
		\zeta_B\zeta_{\frac{10}{a}}
	)'
	\xi_1
\>|
	\lesssim
B^{-2}
\|\xi_1 \|_{L^2_{-\frac{a}{10}}}^2
	+	
e^{-B}\varepsilon^{-\widetilde{N}} B^{-2}
\|w_1\|_{L^2_{-\frac{a}{10}}}
\|\xi_1 \|_{L^2_{-\frac{a}{10}}}.
\end{align}
For the 3rd term of  line \eqref{eq:v22m04}, we have
\begin{align}
\label{eq:v22m09}
|\<
	v_1,
	\chi_{B^2}
	(
		\zeta_B \zeta_{\frac{10}{a}}^2
	)''
	\xi_1
\>|
	=
|\<
	\xi_1,
	\zeta_B^{-1}(\zeta_B\zeta_{\frac{10}{a}}^2)'' \xi_1
\>|
	\lesssim
\|\xi_1\|_{L^2_{-\frac{a}{10}}}^2
\end{align}
For the 1st term of  line \eqref{eq:v22m04-}  we have
\begin{align}
\label{eq:v22m10}
-
\<
	v_1,
	2\chi_{B^2}'
	\zeta_B \zeta_{\frac{10}{a}}^2
	\xi_1'
\>
	=
-\<
	\xi_1,
	2\chi_{B^2}'
	\zeta_B \zeta_{\frac{10}{a}}^2
	\xi_1'
\>
	-
\<
	(1-\chi_{B^2})\zeta_B \mathcal{T} \chi_{B^2} \zeta_A^{-1} w_1,
	2\chi_{B^2}'
	\zeta_B \zeta_{\frac{10}{a}}^2
	\xi_1'
\>
\end{align}
For the 1st term of the r.h.s.\ of \eqref{eq:v22m10},
\begin{align}
\label{eq:v22m11}
|\<
	\xi_1,
	2\chi_{B^2}'
	\zeta_B \zeta_{\frac{10}{a}}^2
	\xi_1'
\>|
	\lesssim
B^{-2}\|\xi_1\|_{L^2_{-\frac{a}{10}}} \|\xi_1'\|_{L^2_{-\frac{a}{10}}},
\end{align}
and for the 2nd term of the  r.h.s.\ of \eqref{eq:v22m10},
\begin{align}
\label{eq:v22m12}
|\<
	(1-\chi_{B^2})\zeta_B \mathcal{T} \chi_{B^2} \zeta_A^{-1} w_1,
	2\chi_{B^2}'
	\zeta_B \zeta_{\frac{10}{a}}^2
	\xi_1'
\>|
	\lesssim
e^{-B}B^{-1}\varepsilon^{-\widetilde{N}} \|w_1\|_{L^2_{-\frac{a}{10}}} \|\xi_1' \|_{L^2_{-\frac{a}{10}}}.
\end{align}
Combining \eqref{eq:v22m11} and \eqref{eq:v22m12}, we have
\begin{align}
\label{eq:v22m13}
|\<
	v_1,
	2\chi_{B^2}'
	\zeta_B \zeta_{\frac{10}{a}}^2
	\xi_1'
\>|
	\lesssim
\(
	B^{-2}\|\xi_1\|_{L^2_{-\frac{a}{10}}}
		+
	B^{-1}e^{-B}\varepsilon^{\widetilde{N}}
	\|w_1\|_{L^2_{-\frac{a}{10}}}
\)
\|\xi_1' \|_{\widetilde{\Sigma}}.
\end{align}
For the 2nd term of line \eqref{eq:v22m04-}, we have
\begin{align}
\label{eq:v22m14}
|\<
	v_1,
	2\chi_{B^2}
	(
		\zeta_B \zeta_{\frac{10}{a}}^2
	)'
	\xi_1'
\>|
	=
2|\<
	\xi_1,
	\zeta_B^{-1}
	(
		\zeta_B \zeta_{\frac{10}{a}}^2
	)'
	\xi_1'
\>|
	\lesssim
\|\xi_1\|_{L^2_{-\frac{a}{10}}}\|\xi_1\|_{\widetilde{\Sigma}}.
\end{align}
For the last term of line \eqref{eq:v22m04-}, we have
\begin{align}
\label{eq:v22m15}
|\<
	v_1,
	\chi_{B^2} \zeta_B \zeta_{\frac{10}{a}}^2 \xi_1''
\>|
	=
|\<
	\xi_1,
	 \zeta_{\frac{10}{a}}^2 \xi_1''
\>|
	\leq
|\<
	\xi_1',
	\zeta_{\frac{10}{a}}^2 \xi_1'
\>|
	+
|\<
	\xi_1,
	 \(\zeta_{\frac{10}{a}}^2\)' \xi_1'
\>|
	\lesssim
\|\xi_1\|_{\widetilde{\Sigma}}^2.
\end{align}
Collecting \eqref{eq:v22m071}, \eqref{eq:v22m081}, \eqref{eq:v22m09}, \eqref{eq:v22m13}, \eqref{eq:v22m14} and \eqref{eq:v22m15} we have \eqref{eq:v22m1}.
\end{proof}

\begin{lemma}
\label{lem:2ndv2r}
We have
\begin{align}
\label{eq:v22r00}
|\<
	\mathbf{J}
	\(
		\sum_{\mathbf{m}\in\boldsymbol{\mathcal{R}}_{\mathrm{min}}}
			\mathbf{z}^{\mathbf{m}}\widetilde{\mathbf{G}}_{\mathbf{m}}
			+
		\boldsymbol{\mathcal{R}}_{\mathbf{v}}
	\),
	\chi_{B^2}^2\zeta_{\widetilde{B}}^2 \sigma_3\mathbf{v}
\>|	
	\lesssim&
c^{-1}\sum_{\mathbf{m}\in \mathbf{R}_{\mathrm{min}}}
	|\mathbf{z}^{\mathbf{m}}|^2
	+
\delta |\dot{\mathbf{z}}-\widetilde{\mathbf{z}}|^2
	+
\|\xi_1\|_{\widetilde{\Sigma}}^2
	\\&+
 (c+\varepsilon) \|\xi_2\|_{L^2_{-\frac{a}{10}}}^2
	+
\varepsilon \|w_1\|_{\widetilde{\Sigma}}^2
	+
\varepsilon \|w_2\|_{L^2_{-\frac{a}{10}}}^2 \nonumber
\end{align}
\end{lemma}

\begin{proof}
First, recalling \eqref{eq:RFRemainder},  we have
\begin{align}
\label{eq:v22r01}
|\<
	\mathbf{J} \mathcal{T} \chi_{B^2}
	P_c
	\boldsymbol{\mathcal{R}}[\mathbf{z}],
	\chi_{B^2}^2\zeta_{\widetilde{B}}^2 \sigma_3\mathbf{v}
\>|
	\lesssim
\sum_{\mathbf{m}\in\boldsymbol{\mathcal{R}}_{\mathrm{min}}} |\mathbf{z}^{\mathbf{m}}|
 \|\boldsymbol{\xi}\|_{L^2_{-\frac{a}{10}}}.
\end{align}
Next,
\begin{align}
\label{eq:v22r02} &
|\<
	\mathbf{J} \mathcal{T} \chi_{B^2}
	P_c
	D\boldsymbol{\phi}[\mathbf{z}](\dot{\mathbf{z}}-\widetilde{\mathbf{z}}),
	\chi_{B^2}^2\zeta_{\widetilde{B}}^2 \sigma_3\mathbf{v}
\>|
	\lesssim
\delta |\dot{\mathbf{z}}-\widetilde{\mathbf{z}}|
 \|\boldsymbol{\xi}\|_{L^2_{-\frac{a}{10}}}, \\&
\label{eq:v22r03}
|\<
	\mathbf{J} \mathcal{T} \chi_{B^2}
	P_c
	\mathbf{L}[\mathbf{z}]\(R[\mathbf{z}]-1\)\widetilde{\boldsymbol{\eta}},
	\chi_{B^2}^2\zeta_{\widetilde{B}}^2 \sigma_3\mathbf{v}
\>|
	\lesssim
\delta
\|\mathbf{w} \|_{L^2_{-\frac{a}{10}}}
 \|\boldsymbol{\xi}\|_{L^2_{-\frac{a}{10}}} .\end{align}
 We have, using the notation of Lemma \ref{lem:v11remainder},
 \begin{align}
\label{eq:v22r04} &
|\<
	\mathbf{J} \mathcal{T} \chi_{B^2}
	 P_c  \Delta _{W ''}(\mathbf{z})  E_1  \boldsymbol{\eta},
	\chi_{B^2}^2\zeta_{\widetilde{B}}^2 \sigma_3\mathbf{v}
\>|
	\lesssim
\delta \varepsilon^{-\widetilde{N}}
\|\mathbf{w} \|_{L^2_{-\frac{a}{10}}}
 \|\boldsymbol{\xi}\|_{L^2_{-\frac{a}{10}}}\\&
\label{eq:v22r05}
|\<
	\mathbf{J}   \mathcal{T}\chi_{B^2}  P_c  \mathbf{JF},
	\chi_{B^2}^2\zeta_{\widetilde{B}}^2 \sigma_3\mathbf{v}
\>|
	\lesssim
\varepsilon^{-\widetilde{N}}
\|\chi_{B^2} |\boldsymbol{\eta}|^2 \|_{L^2_{-\frac{a}{10}}}
\|\boldsymbol{\xi}\|_{L^2_{-\frac{a}{10}}}
 	\lesssim
\delta \varepsilon^{-\widetilde{N}}
\|\mathbf{w} \|_{L^2_{-\frac{a}{10}}}
 \|\boldsymbol{\xi}\|_{L^2_{-\frac{a}{10}}} .
\end{align}
Finally,
\begin{align}
\nonumber
&|\<
	\mathbf{J}
	\begin{pmatrix}
		0 \\
		 -\mathcal{T}(2\chi_{B^2}'\partial_x +\chi_{B^2}'')\widetilde{\eta}_1+\<\im \varepsilon\partial_x\>^{-\widetilde{N}}\left [V_{D},\<\im \varepsilon\partial_x\>^{\widetilde{N}}\right ]v_1
	\end{pmatrix},
	\chi_{B^2}^2\zeta_{\widetilde{B}}^2 \sigma_3\mathbf{v}
\>|
	\\&\nonumber=
|\<
	-\mathcal{T}(2\chi_{B^2}'\partial_x +\chi_{B^2}'')\widetilde{\eta}_1+\<\im \varepsilon\partial_x\>^{-\widetilde{N}}\left [V_{D},\<\im \varepsilon\partial_x\>^{\widetilde{N}}\right ] v_1,
	\chi_{B^2}^2\zeta_{\widetilde{B}}^2 v_1
\>|
	\\&\nonumber\lesssim
\varepsilon^{-\widetilde{N}}e^{-B}
\|w_1\|_{\widetilde{\Sigma}}
\|\xi_1 \|_{\widetilde{\Sigma}}
	+
\varepsilon
\|\<x\>^{-10} v_1\|_{L^2}
\|\xi_1\|_{L^2_{-\frac{a}{10}}}
	\\&\lesssim
\varepsilon^{-\widetilde{N}}e^{-B}
\|w_1\|_{\widetilde{\Sigma}}
\|\xi_1 \|_{\widetilde{\Sigma}}
	+
\varepsilon
\(
	\|\xi_1\|_{\widetilde{\Sigma}}
		+
	\varepsilon^{-\widetilde{N}}
	B^{-1} \|w_1\|_{\widetilde{\Sigma}}
\)
\|\xi_1\|_{L^2_{-\frac{a}{10}}}
	\nonumber\\&\lesssim
\varepsilon
\(
	\|w_1\|_{\widetilde{\Sigma}}
	\|\xi_1 \|_{\widetilde{\Sigma}}
		+
	\|\xi_1\|_{\widetilde{\Sigma}}^2
\).\label{eq:v22r06}
\end{align}
Combining \eqref{eq:v22r01}--\eqref{eq:v22r06}, we obtain \eqref{eq:v22r00}.
\end{proof}

\begin{proof}[Proof of Lemma \ref{lem:2ndv}]
The inequalities \eqref{eq:2ndv1} and \eqref{eq:2ndv2} follows from Lemmas \ref{lem:2ndv1m}, \ref{lem:2ndv1r}, \ref{lem:2ndv2m} and \ref{lem:2ndv2r}.
\end{proof}

\section{Proof of Proposition \ref{prop:modp}}
\label{sec:propdmodes1}

 \textit{ Proof of Proposition \ref{prop:modp}.} Recalling  equation \eqref{eq:eta}, which we rewrite in an equivalent form
\begin{align*}
\partial_t \boldsymbol{\eta} + D_{\mathbf{z}}\boldsymbol{\phi}[\mathbf{z}](\dot{\mathbf{z}}-\widetilde{\mathbf{z}})= \mathbf{L}[\mathbf{z}]\boldsymbol{\eta}    + \mathbf{J} \mathbf{F}[\mathbf{z},\boldsymbol{\eta}] + \boldsymbol{\mathcal{R}}[\mathbf{z}],
\end{align*}
and taking the inner product between  of this equation with  and $\mathbf{J}D_{\mathbf{z}}\boldsymbol{\phi}[\mathbf{z}]\boldsymbol{\zeta}$, for any fixed $\boldsymbol{\zeta} \in \C  ^{\widetilde{N}}$, we have
\begin{align}&
\<\partial_t \boldsymbol{\eta} ,\mathbf{J}D_{\mathbf{z}}\boldsymbol{\phi}[\mathbf{z}]\boldsymbol{\zeta}\>+ \< D_{\mathbf{z}}\boldsymbol{\phi}[\mathbf{z}](\dot{\mathbf{z}}-\widetilde{\mathbf{z}}),\mathbf{J}D_{\mathbf{z}}\boldsymbol{\phi}[\mathbf{z}]\boldsymbol{\zeta})\> =
\< \mathbf{L}[\mathbf{z}]\boldsymbol{\eta}  ,\mathbf{J}D_{\mathbf{z}}\boldsymbol{\phi}[\mathbf{z}]\boldsymbol{\zeta}\> + \<\mathbf{F} , D_{\mathbf{z}}\boldsymbol{\phi}[\mathbf{z}]\boldsymbol{\zeta}\> , \label{eq:inneqe}
\end{align}
where we exploited the orthogonality  \eqref{eq:R1FRemainder--},    $\<\mathbf{J} \boldsymbol{\mathcal{R}}[\mathbf{z}],D_{\mathbf{z}}\boldsymbol{\phi}[\mathbf{z}]\boldsymbol{\zeta}\>=0$.
By Leibnitz and the orthogonality condition  $\<\boldsymbol{\eta}, \mathbf{J} D_{\mathbf{z}}\boldsymbol{\phi}[\mathbf{z}]\boldsymbol{\zeta}\>=0$,  we have
\begin{align*}&
\<\partial_t \boldsymbol{\eta} ,\mathbf{J}D_{\mathbf{z}}\boldsymbol{\phi}[\mathbf{z}]\boldsymbol{\zeta}\> = -\<  \boldsymbol{\eta} ,\mathbf{J}D^2_{\mathbf{z}}\boldsymbol{\phi}[\mathbf{z}]( \dot {\mathbf{z}} , \boldsymbol{\zeta})\> .
\end{align*}
Next, differentiating \eqref{eq:RF} w.r.t.\ $\mathbf{z}$, we have
\begin{align} \mathbf{L}[\mathbf{z}]D_{\mathbf{z}}\boldsymbol{\phi}[\mathbf{z}]\boldsymbol{\zeta}  = -
D_{\mathbf{z}}^2 \boldsymbol{\phi}[\mathbf{z}](\widetilde{\mathbf{z}},\boldsymbol{\zeta}) +D_{\mathbf{z}} \boldsymbol{\phi}[\mathbf{z}](D_{\mathbf{z}}\widetilde{\mathbf{z}}[\mathbf{z}]\boldsymbol{\zeta})   - D_{\mathbf{z}}\boldsymbol{\mathcal{R}}[\mathbf{z}]\boldsymbol{\zeta}.\label{eq:RFdiff}
\end{align}
Inserting the information in \eqref{eq:inneqe}, we obtain
\begin{align*}
\<D_{\mathbf{z}}\boldsymbol{\phi}[\mathbf{z}]
(\dot{\mathbf{z}}-\widetilde{\mathbf{z}}),\mathbf{J}D_{\mathbf{z}}\boldsymbol{\phi}[\mathbf{z}]\boldsymbol{\zeta}\>+\<\mathbf{J}\boldsymbol{\eta},
D^2_{\mathbf{z}}\boldsymbol{\phi}[\mathbf{z}](\dot{\mathbf{z}}-\widetilde{\mathbf{z}},\boldsymbol{\zeta})\>=
\<\mathbf{J}\boldsymbol{\eta},D_\mathbf{z}\boldsymbol{\mathcal{R}}[\mathbf{z}]\boldsymbol{\zeta}\>+\<\mathbf{F}[\mathbf{z},\boldsymbol{\eta}],D_{\mathbf{z}}\boldsymbol{\phi}[\mathbf{z}]\boldsymbol{\zeta}\>.
\end{align*}
Now,
\begin{align*}
\<D_{\mathbf{z}}\boldsymbol{\phi}[\mathbf{z}](\dot{\mathbf{z}}-\widetilde{\mathbf{z}}),\mathbf{J}D_{\mathbf{z}}\boldsymbol{\phi}[\mathbf{z}]\boldsymbol{\zeta}\>
=\<D_{\mathbf{z}}\boldsymbol{\phi}[0](\dot{\mathbf{z}}-\widetilde{\mathbf{z}}),\mathbf{J}D_{\mathbf{z}}\boldsymbol{\phi}[0]\boldsymbol{\zeta}\>+O(|\mathbf{z}||\dot{\mathbf{z}}-\widetilde{\mathbf{z}}|),
\end{align*}
and
\begin{align*}
\<D_{\mathbf{z}}\boldsymbol{\phi}[0](\dot{\mathbf{z}}-\widetilde{\mathbf{z}}),\mathbf{J}D\boldsymbol{\phi}[0]\mathbf{e}^j\>=4\<\mathrm{Re}\boldsymbol{\phi}_j(\dot{\mathbf{z}}-\widetilde{\mathbf{z}}),\mathbf{J}\mathrm{Re}\boldsymbol{\phi}_j\>=-2\mathrm{Im}(\dot{\mathbf{z}}-\widetilde{\mathbf{z}}),\\
\<D_{\mathbf{z}}\boldsymbol{\phi}[0](\dot{\mathbf{z}}-\widetilde{\mathbf{z}}),\mathbf{J}D\boldsymbol{\phi}[0]\im \mathbf{e}^j\>=-4\<\mathrm{Re}\boldsymbol{\phi}_j(\dot{\mathbf{z}}-\widetilde{\mathbf{z}}),\mathbf{J}\mathrm{Im}\boldsymbol{\phi}_j\>=-2\mathrm{Re}(\dot{\mathbf{z}}-\widetilde{\mathbf{z}}).
\end{align*}
Thus, by
\begin{align}
|\<\mathbf{J}\boldsymbol{\eta},D_\mathbf{z}\boldsymbol{\mathcal{R}}[\mathbf{z}]\boldsymbol{\zeta}\>|\lesssim \delta\|\mathbf{w}\|_{L^2_{-\frac{a}{10}}},\\
|\<\mathbf{F}[\mathbf{z},\boldsymbol{\eta}],D\boldsymbol{\phi}[\mathbf{z}]\boldsymbol{\zeta}\>|\lesssim \|\eta_1^2\|_{L^2_{-\frac{a}{10}-A^{-1}}}\lesssim \delta \|w_1\|_{L^2_{-\frac{a}{10}}},
\end{align}
for $\boldsymbol{\zeta}=\mathbf{e}^j, \im \mathbf{e}^j$, we have the conclusion.

\qed

Our next task, is to examine the terms $\mathbf{z} ^{\mathbf{m}}$ and show  $\mathbf{z}   \xrightarrow {t  \to  + \infty  }0$, that is the discrete modes are damped   by nonlinear interaction with the radiation.
In order to do so, we   expand the variable $\mathbf{v}$, defined in \eqref{eq:vBg}, in a part resonating with the discrete modes $\mathbf{z}$, which will yield the damping,  and a remainder which we denote by $\mathbf{g}$. Notice that this additional variable $\mathbf{g}$, is standard in the field, starting from \cite{BP2,SW3}.

\section{Smoothing estimate  for $\mathbf{g}$}
\label{sec:smoothing}

Looking at the equation for $\mathbf{v}$, \eqref{eq:vBg}, we introduce the functions \begin{align}&  \label{eq:def_zetam}
 \boldsymbol{\rho}  _ {\mathbf{m}}:=  R ^{+}_{\im \mathbf{L}_{D} }(\boldsymbol{\lambda} \cdot \mathbf{m})  \im \widetilde{\mathcal{R}}_{ \mathbf{m}},
\end{align}
which solve
\begin{align}&  \label{eq:eq_zetam}
  (\im \mathbf{L}_{D}-\boldsymbol{\lambda} \cdot \mathbf{m}  )\boldsymbol{\rho} _{ \mathbf{m}} =  \im   \widetilde{\mathcal{R}}_{ \mathbf{m}}
\end{align}
and we set
\begin{align} \label{eq:expan_v}
 \mathbf{g}= \mathbf{v}+ Z(\mathbf{z}) \text{  where } Z(\mathbf{z}):=-  \sum_{\mathbf{m}\in \mathbf{R}_{\mathrm{min}}}   \mathbf{z} ^{\mathbf{m}} \boldsymbol{\rho} _{ \mathbf{m}} .
\end{align}
  An elementary computation yields
\begin{align*}  &
\im \partial_t     \mathbf{g }   = \im \mathbf{L}_{D} \mathbf{g} - \sum_{\mathbf{m}\in \mathbf{R}_{\min}}\(\im \partial_t \(\mathbf{z}^{\mathbf{m}}\)-\boldsymbol{\lambda}\cdot \mathbf{m} \, \mathbf{z}^{\mathbf{m}} \) \boldsymbol{\rho} _{ \mathbf{m}}+ \im  \mathbf{{R}}_{v}
\end{align*}
or, equivalently,
\begin{align} \label{eq:equation_g11}& \mathbf{g}(t) = e^{t \mathbf{L}_{D} }\mathbf{v}(0) + \im \sum_{\mathbf{m}\in \mathbf{R}_{\min}}\mathbf{z}^{\mathbf{m}}(0) e^{  t\mathbf{L}_{D} }R ^{+}_{\im \mathbf{L}_{D} }(\boldsymbol{\lambda} \cdot \mathbf{m})  \widetilde{\mathcal{R}}_{ \mathbf{m}} \\&  \label{eq:equation_g12}-  \sum_{\mathbf{m}\in \mathbf{R}_{\min}}\im \int _0 ^t e^{  (t-t') \mathbf{L}_{D} }\(  \partial_t \(\mathbf{z}^{\mathbf{m}}\)+\im \boldsymbol{\lambda}\cdot \mathbf{m}  \, \mathbf{z}^{\mathbf{m}} \) \boldsymbol{\rho} _{ \mathbf{m}} dt' \\& \label{eq:equation_g13}- \im \int _0 ^t e^{  (t-t') \mathbf{L}_{D} }\mathcal{T} \(2\chi_{B^2}'\partial_x + \chi_{B^2}''\)\tilde{\eta}_1 \mathbf{j}  dt'
\\& \label{eq:equation_g14}- \im \int _0 ^t e^{  (t-t') \mathbf{L}_{D} } \( \<\im \varepsilon \partial_x\>^{-\widetilde{N}}[V_{D},\< \im \varepsilon  \partial_x\> ^{\widetilde{N}}]v _1 \mathbf{j }  + \mathcal{T}  \chi_{B^2} \mathbf{{R}}_{\widetilde{\eta}} \) .
\end{align}

We will prove the following, where we use the weighted spaces defined in  Definition \ref{def:spaces}.
\begin{proposition}\label{prop:estg1} For $S>4$ we have  \begin{align}   \label{eq:estg11}&
  \| \mathbf{g}   \| _{L^2(I,  \boldsymbol{\mathcal{H}}  ^{1,-S}(\R )) }       \le   o_\varepsilon (1) \epsilon  .
\end{align}
 \end{proposition}
To prove Proposition \ref{prop:estg1}   we will need to bound one by one the terms in \eqref{eq:equation_g11}--\eqref{eq:equation_g14}.

   \begin{lemma}\label{lem:smooth}
For any $S >5/2$  there exists a fixed $ c(S)$ s.t.
 \begin{equation} \label{eq:smooth1a}
\|   e^{  t\mathbf{L}_{D}     }
\mathbf{f} \|_{L^{ 2 }(\R ,  \boldsymbol{\mathcal{H}}  ^{1,-S} )} \le c(S)  \| \mathbf{f}   \|_{ \boldsymbol{\mathcal{H}}  ^{1}}  \text{  for all $f \in \boldsymbol{\mathcal{H}}  ^{1}(\R )$} .
\end{equation}
\end{lemma}

By Lemma \ref{lem:smooth}  we have
\begin{align*}&        \| \text{r.h.s. of \eqref{eq:equation_g11}}   \| _{L^2(I,  \boldsymbol{\mathcal{H}}  ^{1,-S}(\R )) }  \lesssim \| v(0) \| _{\boldsymbol{\mathcal{H}}  ^{1}}+ \|\mathbf{z}(0) \| ^2
\end{align*}

\textit{Proof of Lemma  \ref{lem:smooth}}.   Recall that, as a consequence of  \eqref{eq:LAP1}, we have  \begin{align*}&      \sup _{0<\epsilon \le 1}  \sup _{\varsigma \in \R } \|   R  _{\im  \mathbf{L}_{D}}(\varsigma  \pm \im \epsilon ) \| _{ \boldsymbol{\mathcal{H}}^{1, S}\to  \boldsymbol{\mathcal{H}}^{1, -S}} <\infty  .
\end{align*}
This easily  implies that
\begin{align*}&     a_3 :=     \sup _{0<\epsilon \le 1}\sup _{\varsigma , \mathbf{f}}    \< \<x\> ^{-S} \( R  _{\im  \mathbf{L}_{D}}(\varsigma   + \im \epsilon )  - R  _{\im  \mathbf{L}_{D}}(\varsigma   - \im \epsilon ) \)   (\<x\> ^{-S}) ^*\mathbf{f}, \mathbf{f}\> _{ \boldsymbol{\mathcal{H}}^{1 }} ,
\end{align*}
where $ (\<x\> ^{-S}) ^*$  is the adjoint of the multiplicative operator $   \<x\> ^{-S} $   in $\boldsymbol{\mathcal{H}}^{1 }$.  Then by  Lemma 3.6  and Lemma 5.5  \cite{kato66} we have \eqref{eq:smooth1a}  with $C(S)= \sqrt{2\pi  a_3}$. \qed

By Lemma \ref{lem:lemg9}  we have
 \begin{align*}&     \| \text{\eqref{eq:equation_g12}}   \| _{L^2(I,  \boldsymbol{\mathcal{H}}  ^{1,-S}(\R )) }\lesssim \sum_{\mathbf{m}\in \mathbf{R}_{\min}}  \int _0^t     \langle  t -t'\rangle ^{-\frac 32 }     |\partial_t \(\mathbf{z}^{\mathbf{m}}\)+\im \boldsymbol{\lambda}\cdot \mathbf{m} \mathbf{z}^{\mathbf{m}}|    dt' \|   \widetilde{\mathcal{R}}_{ \mathbf{m}}  \|_{ \boldsymbol{\mathcal{H}}^{1, S}  }  \\& \lesssim  \sum_{\mathbf{m}\in \mathbf{R}_{\min}} \| \partial_t \(\mathbf{z}^{\mathbf{m}}\)+\im \boldsymbol{\lambda}\cdot \mathbf{m} \mathbf{z}^{\mathbf{m}}\| _{L^2(I)}  \|   \mathcal{R}_{ \mathbf{m}}  \|_{ \Sigma ^{\widetilde{N}+1}  }  \lesssim  \sum_{\mathbf{m}\in  \mathbf{R}_{\min}} \| \partial_t \(\mathbf{z}^{\mathbf{m}}\)+\im \boldsymbol{\lambda}\cdot \mathbf{m} \mathbf{z}^{\mathbf{m}}\| _{L^2(I)} \\& \lesssim \sum_{\mathbf{m}\in\mathbf{R}_{\min}} \(\| D_{\mathbf{z}}\mathbf{z}^{\mathbf{m}} (\dot {\mathbf{z}}- \widetilde{{\mathbf{z}}}) \| _{L^2(I)}  + \| D_{\mathbf{z}}\mathbf{z}^{\mathbf{m}} (  \widetilde{{\mathbf{z}}} +\im    \boldsymbol{\lambda} \mathbf{z} ) \| _{L^2(I)} \) \\& \lesssim
 \| \mathbf{z} \| _{L^\infty(I)} \|  \dot {\mathbf{z}}- \widetilde{{\mathbf{z}}}  \| _{L^2(I)} + \| \mathbf{z} \| _{L^\infty(I)}  \sum_{\mathbf{m}\in \mathbf{R}_{\min}}\| \mathbf{z}^{\mathbf{m}}\| _{L^2(I)}\lesssim \delta ^2 \| \mathbf{w } \| _{L^{2}_{-\frac{a}{10}}} + \delta   \sum_{\mathbf{m}\in \mathbf{R}_{\min}}\| \mathbf{z}^{\mathbf{m}}\| _{L^2(I)} .
\end{align*}
By Lemma \ref{lem:smoothest}  we have
\begin{align}&     \| \text{\eqref{eq:equation_g13}}   \| _{L^2(I,  \boldsymbol{\mathcal{H}}  ^{1,-S} ) }\lesssim
 \|  \mathcal{T} \(2\chi_{B^2}'\partial_x + \chi_{B^2}''\)\tilde{\eta}_1 \| _{L^2( I , L^{2, \tau}) }\lesssim B ^{-\frac{1}{2}}\varepsilon =o_\varepsilon (1) \epsilon  \label{eq:essential1}
\end{align}
where the last inequality is proved in \S 8, in particular formulas (8.23)--(8.25), \cite{CM2109.08108}.

We next look at  \eqref{eq:equation_g14}. Again by   Lemma \ref{lem:smoothest}  we have
\begin{align*}&     \| \int _0 ^t e^{  (t-t') \mathbf{L}_{D} } \<\im \varepsilon \partial_x\>^{-\widetilde{N}}[V_{D},\< \im \varepsilon  \partial_x\> ^{\widetilde{N}}]v _1 \mathbf{j }    \| _{L^2(I,  \boldsymbol{\mathcal{H}}  ^{1,-S} ) } \lesssim
 \|   \<\im \varepsilon \partial_x\>^{-\widetilde{N}}[V_{D},\< \im \varepsilon  \partial_x\> ^{\widetilde{N}}]v _1    \| _{L^2(I,  L  ^{1,\tau} ) } \lesssim  \varepsilon   \epsilon ,
\end{align*}
 where the last inequality is proved in formula (8.26) \cite{CM2109.08108}.  Finally,  we consider
\begin{align*}&     \| \int _0 ^t e^{  (t-t') \mathbf{L}_{D} }  \mathcal{T}  \chi_{B^2} \mathbf{R}_{\widetilde{\eta}}   \| _{L^2(I,  \boldsymbol{\mathcal{H}}  ^{1,-S} ) } \lesssim
 \|  \mathcal{T}  \chi_{B^2} \mathbf{R}_{\widetilde{\eta}}   \| _{L^2(I,  \boldsymbol{\mathcal{H}}  ^{1,\tau} ) }  .
\end{align*}
The right hand side is less than $I+II$ where
\begin{align*} \nonumber & I =    \|    \chi _{8B^2}     \mathcal{T}\chi _{B^2} \mathbf{R}_{\widetilde{\eta}}   \| _{L^2(I,  \boldsymbol{\mathcal{H}}  ^{1,\tau} ) } \\& II =    \|    (1- \chi _{8B^2})     \mathcal{T}\chi _{B^2} \mathbf{R}_{\widetilde{\eta}}   \| _{L^2(I,  \boldsymbol{\mathcal{H}}  ^{1,\tau} ) }
\end{align*}
We have
\begin{align*}   & I   \lesssim B ^{2\tau }   \|         \mathcal{T}\chi _{B^2} \boldsymbol{\mathcal{R}}_{\widetilde{\eta}}   \| _{L^2(I,  \boldsymbol{\mathcal{H}}  ^{1 } ) }
\end{align*}
with
\begin{align*}   & \|         \mathcal{T}\chi _{B^2} \boldsymbol{\mathcal{R}}_{\widetilde{\eta}}   \| _{L^2(I,  \boldsymbol{\mathcal{H}}  ^{1 } ) } \le
 ( I_1  +I_2+I_3) \text{ where}  \\& \nonumber I_1 =    \| \mathcal{T} \chi _{B^2}  P_c\(\boldsymbol{\mathcal{R}}_1[\mathbf{z}]     -  D\boldsymbol{\phi}[\mathbf{z}](\dot{\mathbf{z}}-\widetilde{\mathbf{z}})
+  \mathbf{L}_1(R[\mathbf{z}]-1)\widetilde{\boldsymbol{\eta}}   \)  \| _{L^2 (I,  \boldsymbol{\mathcal{H}}  ^{1 })}   ,\\& I_2= \|  \mathcal{T} \chi _{B^2}  P_c (\mathbf{L}[\mathbf{z}]- \mathbf{L}_1)\boldsymbol{\eta}     \| _{L^2 (I,  \boldsymbol{\mathcal{H}}  ^{1 })},  \\&  I_3=  \| \mathcal{T} \chi _{B^2}   P_c\mathbf{J}\mathbf{F} \| _{L^2 (I,  \boldsymbol{\mathcal{H}}  ^{1 })}.
\end{align*}
We have
\begin{align*}    & I_1 \le    \|  \boldsymbol{\mathcal{R}}_1[\mathbf{z}] \| _{L^2 (I,  \Sigma ^{\widetilde{N}+1})} + \| \dot{\mathbf{z}}-\widetilde{\mathbf{z}} \| _{L^2 (I)}
+ \|   (R[\mathbf{z}]-1)\widetilde{\boldsymbol{\eta}} \| _{L^2 (I,  \Sigma ^{\widetilde{N}+1} )}  \\& \lesssim   \| \mathbf{z} \| _{L^\infty(I)} [   \sum_{\mathbf{m}\in \boldsymbol{\mathcal{R}}_{\min}}\| \mathbf{z}^{\mathbf{m}}\| _{L^2(I)} +
  \|  \dot {\mathbf{z}}- \widetilde{{\mathbf{z}}}  \| _{L^2(I)} + \| \boldsymbol{\eta} \| _{L^2 (I,  \Sigma ^* )}]\\&  \lesssim \delta  [   \sum_{\mathbf{m}\in \boldsymbol{\mathcal{R}}_{\min}}\| \mathbf{z}^{\mathbf{m}}\| _{L^2(I)} +
  \|  \dot {\mathbf{z}}- \widetilde{{\mathbf{z}}}  \| _{L^2(I)} + \| \mathbf{w} \| _{L^2 (I,  L ^{2} _{-\frac{a}{10}})}] .
\end{align*}
We have
\begin{align*}    & I_2 \le   \|  \mathcal{T} \chi _{B^2}  (\mathbf{L}[\mathbf{z}]- \mathbf{L}_1)\boldsymbol{\eta}     \| _{L^2 (I,  \boldsymbol{\mathcal{H}}  ^{1 })}
+  \| \mathcal{T} \chi _{B^2}  P_d (\mathbf{L}[\mathbf{z}]- \mathbf{L}_1)\boldsymbol{\eta}     \| _{L^2 (I,  \boldsymbol{\mathcal{H}}  ^{1 })}\\&
 \lesssim \varepsilon ^{-\widetilde{N}}\|    (\mathbf{L}[\mathbf{z}]- \mathbf{L}_1)\boldsymbol{\eta}     \| _{L^2 (I,  \boldsymbol{\mathcal{H}}  ^{1 })} \\& \lesssim
\varepsilon ^{-\widetilde{N}}\|    (\mathbf{L}[\mathbf{z}]- \mathbf{L}_1)\widetilde{\boldsymbol{\eta}}     \| _{L^2 (I,  \boldsymbol{\mathcal{H}}  ^{1 })}
+ \varepsilon ^{-\widetilde{N}}\|    (\mathbf{L}[\mathbf{z}]- \mathbf{L}_1)(R[\mathbf{z}]-1)\widetilde{\boldsymbol{\eta}}     \| _{L^2 (I,  \boldsymbol{\mathcal{H}}  ^{1 })}
\\&\lesssim
\|    (W''(\phi_1[\mathbf{z}])-W''(H)) \widetilde{{\eta}} _2     \| _{L^2 (I,  L  ^{2 })}  +  \delta\| \mathbf{w} \| _{L^2 (I,  L ^{2} _{-\frac{a}{10}})}   \\& \le   \( \|    (W''(\phi_1[\mathbf{z}])-W''(H))  \zeta _A  ^{-1} e^{|x|\frac{a}{10}}  \| _{L^\infty (I,  L  ^{\infty })} + \delta \)  \| \mathbf{w} \| _{L^2 (I,  L ^{2} _{-\frac{a}{10}})} \lesssim   \delta   \| \mathbf{w} \| _{L^2 (I,  L ^{2} _{-\frac{a}{10}})}.
\end{align*}
We have
\begin{align*}   & I_3   \lesssim     \varepsilon ^{-\widetilde{N}}  \|   \chi _{B^2}  F _1     \| _{L^2 (I,  L  ^{2 })}  \lesssim  \varepsilon ^{-\widetilde{N}}  \|   \chi _{B^2}   \eta _1^2    \| _{L^2 (I,  L  ^{2 })} \\& \lesssim    \varepsilon ^{-\widetilde{N}}    \| \eta _1 \| _{L^\infty (I,  H  ^{1}  )}
\(  \|        w  _1    \| _{L^2 (I,  L  ^{2 } (|x|\le 2B^2 ) )}  + \|    R[\mathbf{z}]-1)   \eta _1    \| _{L^2 (I,  L  ^{2 })}    \) \\& \lesssim    \varepsilon ^{-\widetilde{N}} \delta  \(    B^2 \|        w  _1    \| _{L^2 (I,  \widetilde{\Sigma })}  +  \| \mathbf{z} \| _{L^\infty (I   )}     \|       w _1    \| _{L^2 (I, \widetilde{\Sigma })}    \)  \lesssim  \varepsilon ^{-\widetilde{N}} \delta       B^2 \|        w  _1    \| _{L^2 (I,  \widetilde{\Sigma })} .
\end{align*}
We conclude that
\begin{align} \label{eq:smooth7} &    \|         \mathcal{T}\chi _{B^2} \boldsymbol{\mathcal{R}} _{\widetilde{\eta}}   \| _{L^2(I,  \boldsymbol{\mathcal{H}}  ^{1 } ) } \lesssim       \text{    and} \\& I \lesssim B ^{2\tau + 2} \delta ^2 \epsilon =o_\varepsilon (1) \epsilon.\nonumber
\end{align}
Turning to the analysis of $II$, we have
\begin{align} \nonumber  & II   \lesssim     \|   (1- \chi _{8B^2})    \< x \> ^{\tau }  \mathcal{T}    \< x \> ^{-\tau } \chi _{2B^2}   \| _{  \boldsymbol{\mathcal{H}}  ^{1 } \to \boldsymbol{\mathcal{H}}  ^{1 } }     \|        \chi _{B^2} \mathbf{R}_{\widetilde{\eta}}   \| _{L^2 (I,  \boldsymbol{\mathcal{H}}  ^{1,\tau} )}   \\& \lesssim \|        \chi _{B^2} \mathbf{R}_{\widetilde{\eta}}   \| _{L^2 (I,  \boldsymbol{\mathcal{H}}  ^{1,\tau} )} \lesssim B ^{2\tau + 2} \delta ^2 \epsilon = o_\varepsilon (1) \epsilon\label{eq:smooth8}
\end{align}
where we used
\begin{align*}    &  \|   (1- \chi _{8B^2})    \< x \> ^{\tau }  \mathcal{T}    \< x \> ^{-\tau } \chi _{2B^2}   \| _{  \boldsymbol{\mathcal{H}}  ^{1 } \to \boldsymbol{\mathcal{H}}  ^{1 } }  \lesssim 1 .
\end{align*}
 Notice this will be a consequence of
\begin{align}    &  \|   (1- \chi _{8B^2})    \< x \> ^{\tau }  \mathcal{T}    \< x \> ^{-\tau } \chi _{2B^2}   \| _{ H  ^{1 } \to H ^{1 } }  \lesssim 1 \label{eq:smooth81} \\& \|   (1- \chi _{8B^2})    \< x \> ^{\tau }  \mathcal{T}    \< x \> ^{-\tau } \chi _{2B^2}   \| _{ L  ^{2 } \to L ^{2 } }  \lesssim 1 .\label{eq:smooth82}
\end{align}
Inequality \eqref{eq:smooth82} is proved in \S 8 \cite{CM2109.08108}. We turn to \eqref{eq:smooth81}. It is enough to bound the  operator norm  of
\begin{align}  &  \partial _x (1- \chi _{8B^2})    \< x \> ^{\tau }  \mathcal{T}    \< x \> ^{-\tau } \chi _{2B^2}     = [\partial _x  ,(1- \chi _{8B^2})  \< x \> ^{\tau }]     \mathcal{T}    \< x \> ^{-\tau } \chi _{2B^2} \nonumber\\&  +   (1- \chi _{8B^2})    \< x \> ^{\tau }    [\partial _x ,\mathcal{T} ]    \< x \> ^{-\tau } \chi _{2B^2}  \label{eq:smooth83}\\& +      (1- \chi _{8B^2})    \< x \> ^{\tau }    \mathcal{T}   [ \partial _x ,\< x \> ^{-\tau } \chi _{2B^2}] \nonumber\\&  +      (1- \chi _{8B^2})    \< x \> ^{\tau }    \mathcal{T}    \< x \> ^{-\tau } \chi _{2B^2}  \partial _x .  \label{eq:smooth84}
\end{align}
All   terms  except the one in line \eqref{eq:smooth84}
 are similar to the operator in \eqref{eq:smooth81}. The most interesting is the one in line \eqref{eq:smooth83}.  This operator equals
\begin{align*}    &   (1- \chi _{8B^2})    \< x \> ^{\tau }  \< \im \varepsilon \partial  _x \> ^{-\widetilde{N} }   [\partial _x , \mathcal{A} ^*]    \< x \> ^{-\tau } \chi _{2B^2} \\& =  \sum _{j=1}^{\widetilde{N}}  (1- \chi _{8B^2})    \< x \> ^{\tau }  \< \im \varepsilon \partial  _x \> ^{-\widetilde{N} }\( \prod _{i=0} ^{\widetilde{N}-1-j} A^{*}_{\widetilde{N}-i} \) \( \log \psi _j \) ^{\prime\prime}  \(\prod _{i= 1}^{j-1} A^{*}_{j-i} \) \< x \> ^{-\tau } \chi _{2B^2}
\end{align*}
with  the convention $ \prod _{i=0} ^{l}B_i =B_0\circ ...\circ B_l$  and where $\psi _j$ is a ground state of $L_j$, see
 \S \ref{sec:intmodes}. The  operators  in the last line summation are similar to the one in  \eqref{eq:smooth81} and satisfy the same estimate.  Obviously for the operator in line \eqref{eq:smooth84} we have
\begin{align*}    &   \|   (1- \chi _{8B^2})    \< x \> ^{\tau }  \mathcal{T}    \< x \> ^{-\tau } \chi _{2B^2}   \partial _x  \| _{ H  ^{1 } \to L ^{2 } } \le \|   (1- \chi _{8B^2})    \< x \> ^{\tau }  \mathcal{T}    \< x \> ^{-\tau } \chi _{2B^2}      \| _{ L  ^{2 } \to L ^{2 } }\lesssim 1.
\end{align*}

\section{Proof of  Proposition \ref{prop:FGR}: the Fermi Golden Rule} \label{sec:FGR}

We can aptly name   $E(\boldsymbol{\phi}[\mathbf{z}])$  \textit{localized} energy, since   $\boldsymbol{\eta}(t)$ is expected to disperse to infinity as $t\to +\infty$ and what remains locally of the solution is $\boldsymbol{\phi}[\mathbf{z}(t)]$. In our analysis of the FGR,    $E(\boldsymbol{\phi}[\mathbf{z}])$  is like a Lyapunov function. So we compute, recall $  \< \mathbf{f},\mathbf{g}\> :=\Re \int {^t\mathbf{f}} \mathbf{g} dx$,
\begin{align}
\frac{d}{dt} E(\boldsymbol{\phi}[\mathbf{z}]) &= \< \nabla E(\boldsymbol{\phi}[\mathbf{z}]),D_{\mathbf{z}}\boldsymbol{\phi}[\mathbf{z}]\dot{\mathbf{z}}\> \nonumber\\&
=-\<\mathbf{J}\(\boldsymbol{\mathcal{R}}[\mathbf{z}]+D_{\mathbf{z}}\boldsymbol{\phi}[\mathbf{z}]\widetilde{\mathbf{z}}\),D_{\mathbf{z}}\boldsymbol{\phi}[\mathbf{z}]\dot{\mathbf{z}}\>\nonumber\\&
=\<\mathbf{J} D_{\mathbf{z}}\boldsymbol{\phi}[\mathbf{z}]\(\dot{\mathbf{z}}-\widetilde{\mathbf{z}}\),D_{\mathbf{z}}\boldsymbol{\phi}[\mathbf{z}]\widetilde{\mathbf{z}}\>\nonumber\\&
=\<\mathbf{J}\( \mathbf{L}[\mathbf{z}]\boldsymbol{\eta} + \mathbf{J} \mathbf{F}[\mathbf{z},\boldsymbol{\eta}] + \boldsymbol{\mathcal{R}}[\mathbf{z}]-\partial_t \boldsymbol{\eta} \),D_{\mathbf{z}}\boldsymbol{\phi}[\mathbf{z}]\widetilde{\mathbf{z}}\>.\label{eq:FGR1}
\end{align}
where we have used \eqref{eq:RF} in the 2nd equality, the cancelation \eqref{eq:R1FRemainder--} and $\<\mathbf{J}f,f\>=0$ in the 3rd equality and \eqref{eq:eta} in the 4th inequality.
Now, from $\boldsymbol{\eta}\in \mathcal{H}_c[\mathbf{z}]$,
\begin{align}
-\<\mathbf{J}\partial_t \boldsymbol{\eta} ,D_{\mathbf{z}}\boldsymbol{\phi}[\mathbf{z}]\widetilde{\mathbf{z}}\> = \<\mathbf{J} \boldsymbol{\eta} ,D_{\mathbf{z}}^2\boldsymbol{\phi}[\mathbf{z}]\(\dot{\mathbf{z}},\widetilde{\mathbf{z}}\)\>,\label{eq:FGR2}
\end{align}
and from \eqref{eq:RFdiff}, the fact that $\mathbf{J}\mathbf{L}[\mathbf{z}] $ is self-adjoint and, again, the fact that $\boldsymbol{\eta}\in \mathcal{H}_c[\mathbf{z}]$,
\begin{align}
\<\mathbf{J} \mathbf{L}[\mathbf{z}]\boldsymbol{\eta} ,D_{\mathbf{z}}\boldsymbol{\phi}[\mathbf{z}]\widetilde{\mathbf{z}}\>&=
\<\boldsymbol{\eta},\mathbf{J} \mathbf{L}[\mathbf{z}] D_{\mathbf{z}}\boldsymbol{\phi}[\mathbf{z}]\widetilde{\mathbf{z}}\>\nonumber\\&=
\<\boldsymbol{\eta},\mathbf{J} \(D_{\mathbf{z}}^2 \boldsymbol{\phi}[\mathbf{z}](\widetilde{\mathbf{z}},\widetilde{\mathbf{z}}) + D_{\mathbf{z}}\boldsymbol{\mathcal{R}}[\mathbf{z}]\widetilde{\mathbf{z}}\)\>.\label{eq:FGR3}
\end{align}
Substituting, \eqref{eq:FGR2} and \eqref{eq:FGR3} into \eqref{eq:FGR1}, we have
\begin{align}\label{eq:compdtEphi}
\frac{d}{dt} E(\boldsymbol{\phi}[\mathbf{z}]) =
-\<\mathbf{J}\boldsymbol{\eta},  D_{\mathbf{z}}\boldsymbol{\mathcal{R}}[\mathbf{z}]\widetilde{\mathbf{z}}\>
+
 \<\mathbf{J} \boldsymbol{\eta} ,D_{\mathbf{z}}^2\boldsymbol{\phi}[\mathbf{z}]\(\dot{\mathbf{z}}-\widetilde{\mathbf{z}},\widetilde{\mathbf{z}}\)\>
 -
 \< \mathbf{F}[\mathbf{z},\boldsymbol{\eta}] ,D_{\mathbf{z}}\boldsymbol{\phi}[\mathbf{z}]\widetilde{\mathbf{z}}\>
\end{align}
\begin{claim}\label{claim:1} For all $t\in I$
\begin{align}\label{eq:claim:1}
 \left | \int _{0} ^{t}\<\mathbf{J}\boldsymbol{\eta},  D_{\mathbf{z}}\boldsymbol{\mathcal{R}}[\mathbf{z}]\widetilde{\mathbf{z}}\> dt' \right |
= o_{\varepsilon}(1) \epsilon ^2.
\end{align}
\end{claim}
\proof Indeed  we have $\left .E(\boldsymbol{\phi}[\mathbf{z}]) \right ] _{0}^{t}=O(\delta ^2)$  from  Proposition \ref{prop:OrbStab}. we have
\begin{align*}
| \<\mathbf{J} \boldsymbol{\eta} ,D_{\mathbf{z}}^2\boldsymbol{\phi}[\mathbf{z}]\(\dot{\mathbf{z}}-\widetilde{\mathbf{z}},\widetilde{\mathbf{z}}\)\>|\lesssim \delta\|\mathbf{w}\|_{L^2_{-\frac{a}{10}}} |\dot{\mathbf{z}}-\widetilde{\mathbf{z}}|,
\end{align*}
and
\begin{align*}
|\< \mathbf{F}[\mathbf{z},\boldsymbol{\eta}] ,D_{\mathbf{z}}\boldsymbol{\phi}[\mathbf{z}]\widetilde{\mathbf{z}}\>|\lesssim \delta \|w_1\|_{L^2_{-\frac{a}{10}}}^2 ,
\end{align*}
and integrating in time, we obtain the desired bound  \eqref{eq:claim:1}.   \qed

Let us focus now on the term in the left hand side of \eqref{eq:claim:1}.
By   the expansion  \eqref{eq:RFRemainder}   of $\boldsymbol{\mathcal{R}}[\mathbf{z}]$,          we have
\begin{align}
 \<\mathbf{J}\boldsymbol{\eta},  D_{\mathbf{z}}\boldsymbol{\mathcal{R}}[\mathbf{z}]\widetilde{\mathbf{z}}\>=
& \sum_{\mathbf{m}\in \mathbf{R}_{\mathrm{min}}}\<\mathbf{J}\boldsymbol{\eta},  D_{\mathbf{z}}\mathbf{z}^{\mathbf{m}}(-\im \boldsymbol{\lambda}\mathbf{z}) \mathcal{R}_{\mathbf{m}}\>    \label{main1} \\&
+\sum_{\mathbf{m}\in \mathbf{R}_{\mathrm{min}}}\<\mathbf{J}\boldsymbol{\eta}, D\mathbf{z}^{\mathbf{m}}(\widetilde{\mathbf{z}}+\im \boldsymbol{\lambda}\mathbf{z}) \mathcal{R}_{\mathbf{m}}\>
-\<\mathbf{J}\boldsymbol{\eta}, D_{\mathbf{z}}\boldsymbol{\mathcal{R}}_1[\mathbf{z}]\widetilde{\mathbf{z}}\>, \label{nomain1}
\end{align}
where $\boldsymbol{\lambda}\mathbf{z}:=(\lambda_1z_1,\cdots,\lambda_N z_N)$.
The 2nd line can be bounded as
\begin{align*}
\sum_{\mathbf{m}\in \mathbf{R}_{\mathrm{min}}}|\<\mathbf{J}\boldsymbol{\eta}, D\mathbf{z}^{\mathbf{m}}(\widetilde{\mathbf{z}}+\im \boldsymbol{\lambda}\mathbf{z}) \mathcal{R}_{\mathbf{m}}\>|
+|\<\mathbf{J}\boldsymbol{\eta}, D_{\mathbf{z}}\boldsymbol{\mathcal{R}}_1[\mathbf{z}]\widetilde{\mathbf{z}}\>|\lesssim \delta \|\mathbf{w}\|_{L^2_{-\frac{a}{10}}} \sum_{\mathbf{m}\in \mathbf{R}_{\mathrm{min}}}|\mathbf{z}^{\mathbf{m}}|.
\end{align*}
Notice that the time integral of the last formula is of the form $o_{\varepsilon}(1) \epsilon ^2.$

\noindent Now we focus on the term in the right in line \eqref{main1}.   Using the identity  $D_{\mathbf{z}}\mathbf{z}^{\mathbf{m}} (\im \boldsymbol{\lambda}\mathbf{z})=\im \mathbf{m}\cdot \boldsymbol{\lambda} \  \mathbf{z}^{\mathbf{m}}$,  this term equals the sum
\begin{align}&
 -\sum_{\mathbf{m}\in \mathbf{R}_{\mathrm{min}}}\boldsymbol{\lambda}\cdot \mathbf{m}\<\mathbf{J}P_c  \boldsymbol{\eta},\im \mathbf{z}^{\mathbf{m}} \mathcal{R}_{\mathbf{m}}\> \label{main2}\\& -
\sum_{\mathbf{m}\in \mathbf{R}_{\mathrm{min}}}\boldsymbol{\lambda}\cdot \mathbf{m}\<\mathbf{J}P_d  (R[\mathbf{z}]-1) P_c\boldsymbol{\eta},\im \mathbf{z}^{\mathbf{m}} \mathcal{R}_{\mathbf{m}}\>  ,\label{nomain2}
\end{align}
where, by Lemma \ref{lem:R},
\begin{align*}
 |\text{\eqref{nomain2}}|\lesssim \delta \| \mathbf{w} \| _{\widetilde{\boldsymbol{\Sigma}}} \sum_{\mathbf{m}\in \mathbf{R}_{\mathrm{min}}} |\mathbf{z}^{\mathbf{m}}|,
\end{align*}
so that its time integral is of the form $o_{\varepsilon}(1) \epsilon ^2.$

So now let us focus on the term in line \eqref{main2}. It equals the sum
\begin{align} & -\sum_{\mathbf{m}\in \mathbf{R}_{\mathrm{min}}}\boldsymbol{\lambda}\cdot \mathbf{m}\<\mathbf{J}P_c \chi_{B^2}\boldsymbol{\eta},\im \mathbf{z}^{\mathbf{m}} \mathcal{R}_{\mathbf{m}}\> \label{eq:line1}\\& -
\sum_{\mathbf{m}\in \mathbf{R}_{\mathrm{min}}}\boldsymbol{\lambda}\cdot \mathbf{m}\<\mathbf{J}P_c (1-\chi_{B^2})\boldsymbol{\eta},\im \mathbf{z}^{\mathbf{m}} \mathcal{R}_{\mathbf{m}}\>  ,\label{eq:line2}
\end{align}
where the terms in  line  \eqref {eq:line2} can be bounded as follows,
\begin{align*}
\sum_{\mathbf{m}\in \mathbf{R}_{\mathrm{min}}}|\boldsymbol{\lambda}\cdot \mathbf{m}\<\mathbf{J}P_c (1-\chi_{B^2})\boldsymbol{\eta},\im \mathbf{z}^{\mathbf{m}} \mathcal{R}_{\mathbf{m}}\>|\lesssim
B^{-1}\|\mathbf{w}\|_{L^2_{-\frac{a}{10}}} \sum_{\mathbf{m}\in \mathbf{R}_{\mathrm{min}}}|\mathbf{z}^{\mathbf{m}}|,
\end{align*}
and so again the time integral is of the form $o_{\varepsilon}(1) \epsilon ^2.$

Now let us focus on \eqref{eq:line1}.     By Lemma  \ref{lem:coer6}, we have
\begin{align*}
\sum_{\mathbf{m}\in \mathbf{R}_{\mathrm{min}}}\boldsymbol{\lambda}\cdot \mathbf{m}\<\mathbf{J}P_c \chi_{B^2}\boldsymbol{\eta},\im \mathbf{z}^{\mathbf{m}} \mathcal{R}_{\mathbf{m}}\>=
\sum_{\mathbf{m}\in \mathbf{R}_{\mathrm{min}}}\boldsymbol{\lambda}\cdot \mathbf{m}\<\mathbf{J}\prod_{j=1}^{\widetilde{N}}R_{L_1}(\widetilde{\lambda}_j^2)P_c \mathcal{A}\<\im \varepsilon\partial_x\>^{\widetilde{N}}\mathbf{v},\im \mathbf{z}^{\mathbf{m}} \mathcal{R}_{\mathbf{m}}\>
\end{align*}
We substitute $\mathbf{v}=\mathbf{g}-Z(\mathbf{z})$ using \eqref{eq:def_zetam} and \eqref{eq:expan_v}. Then the above term becomes
\begin{align} &
\sum_{\mathbf{m}\in \mathbf{R}_{\mathrm{min}}}\boldsymbol{\lambda}\cdot \mathbf{m}  |\mathbf{z}^{\mathbf{m}}| ^2    \<\mathbf{J}\prod_{j=1}^{\widetilde{N}}R_{L_1}(\widetilde{\lambda}_j^2)P_c \mathcal{A}\<\im \varepsilon\partial_x\>^{\widetilde{N}}R ^{+}_{\im \mathbf{L}_{D} }(\boldsymbol{\lambda} \cdot \mathbf{m})  \im \widetilde{\mathcal{R}}_{ \mathbf{m}} ,\im  \mathcal{R}_{\mathbf{m}}\> \label{eq:line11} \\& +\sum_{ \substack{\mathbf{m},\mathbf{n}\in \mathbf{R}_{\mathrm{min}} \\ \mathbf{m}\neq \mathbf{n}}}\boldsymbol{\lambda}\cdot \mathbf{m}      \< \mathbf{z}^{\mathbf{n}} \mathbf{J}\prod_{j=1}^{\widetilde{N}}R_{L_1}(\widetilde{\lambda}_j^2)P_c \mathcal{A}\< \im \varepsilon\partial_x\>^{\widetilde{N}}R ^{+}_{\im \mathbf{L}_{D} }(\boldsymbol{\lambda} \cdot \mathbf{n})  \im \widetilde{\mathcal{R}}_{ \mathbf{n}} ,\im \mathbf{z}^{\mathbf{m}} \mathcal{R}_{\mathbf{m}}\> \label{eq:line12}
\\& +
\sum_{\mathbf{m}\in \mathbf{R}_{\mathrm{min}}}\boldsymbol{\lambda}\cdot \mathbf{m}\<\mathbf{J}\prod_{j=1}^{\widetilde{N}}R_{L_1}(\widetilde{\lambda}_j^2)P_c \mathcal{A}\<\im \varepsilon\partial_x\>^{\widetilde{N}}\mathbf{g},\im \mathbf{z}^{\mathbf{m}} \mathcal{R}_{\mathbf{m}}\> . \label{eq:line13}
\end{align}
The main term is the one in line \eqref{eq:line11} which we leave aside for a moment. We have
\begin{align*} &
    | \text{\ref{eq:line13}} | \lesssim \sum_{\mathbf{m}\in \mathbf{R}_{\mathrm{min}}} |\mathbf{z}^{\mathbf{m}}|   \| \mathbf{g}\| _{L ^{2,-S}}
    \| \<\im \varepsilon\partial_x\>^{\widetilde{N}} \mathcal{A} ^* P_c  \prod_{j=1}^{\widetilde{N}}R_{L_1}(\widetilde{\lambda}_j^2)    \mathcal{R}_{\mathbf{m}} \| _{L ^{2, S}} \\&   \nonumber \lesssim  \| \mathbf{g}\| _{L ^{2,-S}} \sum_{\mathbf{m}\in \mathbf{R}_{\mathrm{min}}} |\mathbf{z}^{\mathbf{m}}|,
\end{align*}
so that, using Proposition \ref{prop:estg1} and the continuation hypothesis \eqref{eq:main11}, we have
\begin{align} &  \label{eq:line131}
    \| \text{\ref{eq:line13}} \| _{L^1 _t} \lesssim      \| \mathbf{g}\| _{L^2L ^{2,-S}} \sum_{\mathbf{m}\in \mathbf{R}_{\mathrm{min}}} \|\mathbf{z}^{\mathbf{m}}\|_{L^2}\le o_\epsilon  (1) \epsilon ^2.
\end{align}
The generic bracket in line  \eqref{eq:line12} is   of the form
 \begin{align*} &   \< \mathbf{z^{n}}\mathbf{z}^{\overline{\mathbf{m}}}, A \> =    \frac{1}{\boldsymbol{\lambda} \cdot (\mathbf{n}-\mathbf{m})} \< -\im  \boldsymbol{\lambda} \cdot (\mathbf{n}-\mathbf{m}) \mathbf{z^{n}}\mathbf{z}^{\overline{\mathbf{m}}}, -\im  A     \> \\&=  \frac{1}{\boldsymbol{\lambda} \cdot (\mathbf{n}-\mathbf{m})} \frac{d}{dt}  \<   \mathbf{z^{n}}\mathbf{z}^{\overline{\mathbf{m}}}, -\im  A     \>-\frac{1}{\boldsymbol{\lambda} \cdot (\mathbf{n}-\mathbf{m})}  \<  D_{\mathbf{z}}( \mathbf{z^{n}}\mathbf{z}^{\overline{\mathbf{m}}} )   \( \dot {\mathbf{z}} +\im \boldsymbol{\lambda}\mathbf{z}\)    , -\im  A     \> ,
\end{align*}
where, for $B^* = \overline{{^tB}  }$,   $A$ is defined as
  \begin{align*} &   A =   ( R ^{+}_{\im \mathbf{L}_{D} }(\boldsymbol{\lambda} \cdot \mathbf{n})    \widetilde{\mathcal{R}}_{ \mathbf{n}} )^*
  \< \im \varepsilon\partial_x\>^{\widetilde{N}} \mathcal{A}^* \mathbf{J}\prod_{j=1}^{\widetilde{N}}R_{L_1}(\widetilde{\lambda}_j^2)P_c \mathcal{R}_{\mathbf{m}} .
\end{align*}
 So we have
\begin{align*} &  \left | \int _0 ^t  \< \mathbf{z^{n}}\mathbf{z}^{\overline{\mathbf{m}}}, A \>  dt'  \right | \lesssim       \left | \left . \<   \mathbf{z^{n}}\mathbf{z}^{\overline{\mathbf{m}}}, -\im  A     \>   \right ] _{0}^{t} \right |  +    \|   D_{\mathbf{z}}( \mathbf{z^{n}}\mathbf{z}^{\overline{\mathbf{m}}} )   \( \dot {\mathbf{z}} +\im \boldsymbol{\lambda}\mathbf{z}\)    \| _{L^1(0,t)}   \| A \| _{L^1_x}.
\end{align*}
We have $ \| A \| _{L^1_x}\lesssim 1$ uniformly in $\varepsilon \in (0,1]$. So the first term on the right is $O( \delta ^2)$. We bound the second term
\begin{align*} &   \|   D_{\mathbf{z}}( \mathbf{z^{n}}\mathbf{z}^{\overline{\mathbf{m}}} )   \( \dot {\mathbf{z}} +\im \boldsymbol{\lambda}\mathbf{z}\)    \| _{L^1 } \le \|   D_{\mathbf{z}}( \mathbf{z^{n}}\mathbf{z}^{\overline{\mathbf{m}}} )   \( \dot {\mathbf{z}} - \widetilde{\mathbf{z}}\)    \| _{L^1 } + \|   D_{\mathbf{z}}( \mathbf{z^{n}}\mathbf{z}^{\overline{\mathbf{m}}} )   \(   \widetilde{\mathbf{z}} +\im \boldsymbol{\lambda}\mathbf{z}    \)    \| _{L^1 } \\& \lesssim  \|   D_{\mathbf{z}}( \mathbf{z^{n}}\mathbf{z}^{\overline{\mathbf{m}}} )       \| _{L^2 }\|    \dot {\mathbf{z}} - \widetilde{\mathbf{z}}     \| _{L^2 } +  \|   \mathbf{z}    \| _{L^\infty } \|    \mathbf{z^{n}}  \| _{L^2 }\|    \mathbf{z^{m}}  \| _{L^2 } =o_\varepsilon (1) \epsilon ^2,
\end{align*}
and so we conclude
\begin{align} &  \left | \int _0 ^t   \text{\eqref{eq:line12} }  dt'  \right | =o_\varepsilon (1) \epsilon ^2. \label{eq:line121}
\end{align}
Now we focus on line    \eqref{eq:line11}, which represents the main term of formula \eqref{eq:line11}--\eqref{eq:line13}.
Using $\widetilde{\mathcal{R}}_{ \mathbf{m}} = \mathcal{T} \chi _{B^2}  P_c\mathcal{R}_{\mathbf{m}}$, the bracket in line \eqref{eq:line11} can be rewritten
 \begin{align} &    \label{mainFGR4}
   \<\mathbf{J}\prod_{j=1}^{\widetilde{N}}R_{L_1}(\widetilde{\lambda}_j^2)P_c \mathcal{A} R ^{+}_{\im \mathbf{L}_{D} }(\boldsymbol{\lambda} \cdot \mathbf{m})  \mathcal{A}^*\chi _{B^2}\mathcal{R}_{ \mathbf{m}} ,  \mathcal{R}_{\mathbf{m}}\> \\& +  \<\mathbf{J}\prod_{j=1}^{\widetilde{N}}R_{L_1}(\widetilde{\lambda}_j^2)P_c \mathcal{A} \< \im \varepsilon\partial_x\>^{\widetilde{N}} [R ^{+}_{\im \mathbf{L}_{D} }(\boldsymbol{\lambda} \cdot \mathbf{m}) , \< \im \varepsilon\partial_x\>^{-\widetilde{N}}] \mathcal{A}^*\chi _{B^2}\mathcal{R}_{ \mathbf{m}} ,  \mathcal{R}_{\mathbf{m}}\>    \tag{$a_{\mathbf{m}}$}
\end{align} where we will show   now that the quantity in $(a_{\mathbf{m}})$ is the form $o ( \varepsilon )$.  This will imply that
 \begin{align}&  \left \| \sum_{\mathbf{m}\in \mathbf{R}_{\mathrm{min}}}\boldsymbol{\lambda}\cdot \mathbf{m}  |\mathbf{z}^{\mathbf{m}}| ^2    (a_{\mathbf{m}})\right \| _{L^1(0,t)}=o ( \varepsilon ) \epsilon ^2.  \label{0remainFGR4-}
\end{align}
For $E_1$ the matrix in \eqref{v11remainder1},   the quantity in $(a_{\mathbf{m}})$  can be bounded by the product $\mathfrak{A}\cdot \mathfrak{B}$, where
  \begin{align*}
  	& \mathfrak{A}=
   \|    \< \im \varepsilon  \partial _x \> ^{ \widetilde{N}}   \mathcal{A}^* \prod_{j=1}^{\widetilde{N}}R_{L_1}(\widetilde{\lambda}_j^2) P_c  \mathcal{R}_{\mathbf{m}}\| _{L^{2,\ell}} \text{  and} \\& \mathfrak{B}= \|   R ^{+}_{\im \mathbf{L}_{D} }(\boldsymbol{\lambda} \cdot \mathbf{m})  E_1  \left [ V _D   ,\< \im \varepsilon  \partial _x \> ^{-\widetilde{N}} \right ]  R ^{+}_{\im \mathbf{L}_{D} }(\boldsymbol{\lambda} \cdot \mathbf{m})   \mathcal{A}^*\chi  _{B^2} \mathcal{R}_{\mathbf{m}} \| _{L^{2,-\ell}} ,
\end{align*}
  for $\ell \ge 2$.   We have
  \begin{align*}    \mathfrak{B}\le &  \|  R ^{+}_{\im \mathbf{L}_{D} }(\boldsymbol{\lambda} \cdot \mathbf{m})   \| _{ \boldsymbol{\mathcal{H}}  ^{1,\ell }\to L^{2,-\ell}}^2
    \| \< \im \varepsilon  \partial _x \> ^{-\widetilde{N}}  \left [ V _{D}   ,\< \im \varepsilon  \partial _x \> ^{ \widetilde{N}} \right ]
     \| _{\boldsymbol{\mathcal{H}}  ^{1,-\ell }\to\boldsymbol{\mathcal{H}}  ^{1,\ell }}   \\& \times \|   \< \im \varepsilon  \partial _x \> ^{-\widetilde{N}} \| _{\boldsymbol{\mathcal{H}}  ^{1,-\ell }\to \boldsymbol{\mathcal{H}}  ^{1,-\ell }}
     \| \mathcal{A}^*\chi  _{B^2}\mathcal{R}_{\mathbf{m}} \| _{\boldsymbol{\mathcal{H}}  ^{1,\ell }}  \lesssim \varepsilon ,
\end{align*}
where the $\varepsilon$ comes from the commutator term in the first line,   by   a simple adaptation of Lemma \ref{claim:l2boundII}, while the other terms are uniformly bounded, with $\|   \< \im \varepsilon  \partial _x \> ^{-\widetilde{N}} \| _{\boldsymbol{\mathcal{H}}  ^{1,-\ell }\to \boldsymbol{\mathcal{H}}  ^{1,-\ell }}  \lesssim 1$  uniformly in $\varepsilon \in (0,1]$,  by the proof of  the bound on (10.23) in  \cite{CM2109.08108}.  Uniformly in $\varepsilon \in (0,1]$, we have
\begin{align*}    \mathfrak{A}\le &  \|  \< \im \varepsilon  \partial _x \> ^{ \widetilde{N}}   \< \im   \partial _x \> ^{-2 \widetilde{N}}  \| _{\boldsymbol{\mathcal{H}}  ^{1,\ell } \to \boldsymbol{\mathcal{H}}  ^{1,\ell }}  \|      \< \im   \partial _x \> ^{ 2 \widetilde{N}}  \mathcal{A}^* \prod_{j=1}^{\widetilde{N}}R_{L_1}(\widetilde{\lambda}_j^2) P_c \mathcal{R}_\mathbf{m} \| _{\boldsymbol{\mathcal{H}}  ^{1,\ell } } \lesssim 1.
\end{align*}
We have thus proved what was needed to obtain  \eqref{0remainFGR4-}.

\noindent We consider  \eqref{mainFGR4},  the main term. Using
\begin{align*} &    R ^{+}_{\im \mathbf{L}_{D} }(\boldsymbol{\lambda} \cdot \mathbf{m})  \mathcal{A}^* = \mathcal{A}^* R ^{+}_{\im \mathbf{L}_{1} }(\boldsymbol{\lambda} \cdot \mathbf{m}),
\end{align*}
which follows from \eqref{eq:DarConj2} and \eqref{eq:resolv1}, using   the formula
\begin{align*} &   \mathcal{A}  \mathcal{A}^*=  \prod_{j=1}^{\widetilde{N}}(L_1-\widetilde{\lambda}_j^2),
\end{align*}
which is an elementary consequence of the discussion in \S \ref{sec:darboux} and is proved in \cite{CM2109.08108}, and finally using the fact that $L_1$ commutes with $P_c$, see Remark \ref{rem:comm}, we conclude that
line \eqref{mainFGR4}  equals
\begin{align}    \label{eq:line1111}
      \<\mathbf{J} P_c   R ^{+}_{\im \mathbf{L}_{1} }(\boldsymbol{\lambda} \cdot \mathbf{m})  \chi _{B^2}\mathcal{R}_{ \mathbf{m}} ,  \mathcal{R}_{\mathbf{m}}\>  = & \<\mathbf{J} P_c   R ^{+}_{\im \mathbf{L}_{1} }(\boldsymbol{\lambda} \cdot \mathbf{m})   \mathcal{R}_{ \mathbf{m}} ,  \mathcal{R}_{\mathbf{m}}\> \\& -  \<\mathbf{J} P_c   R ^{+}_{\im \mathbf{L}_{1} }(\boldsymbol{\lambda} \cdot \mathbf{m}) (1- \chi _{B^2})\mathcal{R}_{ \mathbf{m}} ,  \mathcal{R}_{\mathbf{m}}\>     . \label{eq:line1112}
\end{align}
Is is elementary to show that the last line is $O(B ^{-1})$, so that
 \begin{align}&  \left \| \sum_{\mathbf{m}\in \mathbf{R}_{\mathrm{min}}}\boldsymbol{\lambda}\cdot \mathbf{m}  |\mathbf{z}^{\mathbf{m}}| ^2    \<\mathbf{J} P_c   R ^{+}_{\im \mathbf{L}_{1} }(\boldsymbol{\lambda} \cdot \mathbf{m}) (1- \chi _{B^2})\mathcal{R}_{ \mathbf{m}} ,  \mathcal{R}_{\mathbf{m}}\>\right \| _{L^1(0,t)}=o ( B ^{-1} ) \epsilon ^2.  \label{eq:line111-}
\end{align}
  Using an obvious  analogue of   \eqref{eq:resolv1}, the term in the right hand side in line \eqref{eq:line1111}  can be rewritten as
 \begin{align} &  \label{eq:line1111a}
- \<    \im \mathbf{A}_{\mathbf{m}}    P.V. (L_1-r_{\mathbf{m}} )^{-1} P_c\mathcal{R}_{ \mathbf{m}} ,  P_c\mathcal{R}_{\mathbf{m}}\> \\& +\pi \<      \mathbf{A}_{\mathbf{m}}    \delta (L_1-r_{\mathbf{m}} ) P_c\mathcal{R}_{ \mathbf{m}} ,  P_c\mathcal{R}_{\mathbf{m}}\>  \text{  ,   }    \label{eq:line1111b}\\& \text{where }   \mathbf{A}_{\mathbf{m}}= \begin{pmatrix}
	 r_{\mathbf{m}}^2
	& \im r_{\mathbf{m}} \\
	-\im r_{\mathbf{m}}
	&1
\end{pmatrix}  \ , \quad r_{\mathbf{m}}=\sqrt{(\boldsymbol{\lambda} \cdot \mathbf{m})^2-\omega ^2}  . \nonumber
\end{align}
 By antisymmetry, line \eqref{eq:line1111a} is equal  to  0.   We have
 \begin{align*} & \mathbf{B}_{\mathbf{m}}^{-1}\mathbf{A}_{\mathbf{m}} \mathbf{B}_{\mathbf{m}}=\text{diag}\(0, 1+r_{\mathbf{m}}^2\) \text{ where }
 \mathbf{B}_{\mathbf{m}}=  \begin{pmatrix}
	  1
	&  \im r_{\mathbf{m}} \\
	 \im r_{\mathbf{m}}
	&1
\end{pmatrix} .
\end{align*}
Noticing that $ \mathbf{B}_{\mathbf{m}} ^*=(1+ r_{\mathbf{m}}^2)\mathbf{B}_{\mathbf{m}}^{-1}$,
line \eqref{eq:line1111b}  equals
\begin{align} & \nonumber \pi \<   \mathbf{B}_{\mathbf{m}} ^*   \mathbf{A}_{\mathbf{m}}  \mathbf{B}_{\mathbf{m}}   \delta (L_1-r_{\mathbf{m}} )  \mathbf{B}_{\mathbf{m}}^{-1} P_c\mathcal{R}_{ \mathbf{m}} ,  \mathbf{B}_{\mathbf{m}}^{-1}P_c\mathcal{R}_{\mathbf{m}}\> \\&=\pi
\<      \delta (L_1-r_{\mathbf{m}} )    ,|   -\im r_{\mathbf{m}} (P_c\mathcal{R}_{\mathbf{m}}) _{1}+(P_c\mathcal{R}_{\mathbf{m}}) _{2}|^2 \>  \nonumber \\& = \frac{\pi}{2\sqrt{r_{\mathbf{m}}} }   \sum _{\pm} \left |  \left [-\im r_{\mathbf{m}} \widehat{(P_c\mathcal{R}}_{\mathbf{m}}) _{1}+\widehat{(P_c\mathcal{R}_{\mathbf{m}})} _{2} \right ]  ( \pm \sqrt{r_{\mathbf{m}}})   \right | ^2 \ge 0.  \nonumber
\end{align}
where $(P_c\mathcal{R}_{\mathbf{m}}) _{j}$ are the two components of  $P_c\mathcal{R}_{\mathbf{m}}$ for $j=1,2$ and we are taking the distorted Fourier transform associated to operator $L_1$,   for  which we refer to   Weder \cite{weder}  and  Deift and Trubowitz \cite{DT79CPAM}.
By Assumption  \ref{ass:FGR}  there is a fixed $\Gamma >0$ such that
\begin{align} &  \frac{\pi \boldsymbol{\lambda}\cdot \mathbf{m}   }{2\sqrt{r_{\mathbf{m}}} }   \sum _{\pm} \left |  \left [-\im r_{\mathbf{m}} \widehat{(P_c\mathcal{R}}_{\mathbf{m}}) _{1}+\widehat{(P_c\mathcal{R}_{\mathbf{m}})} _{2} \right ]  ( \pm \sqrt{r_{\mathbf{m}}})   \right | ^2 \ge \Gamma >0   \text{  for all } \mathbf{m} \in \mathbf{R}_{\min} .  \label{eq:line1111aa}
\end{align}
Hence we conclude
\begin{align} &
 \sum_{\mathbf{m}\in \mathbf{R}_{\mathrm{min}}}  \boldsymbol{\lambda}\cdot \mathbf{m}  |\mathbf{z}^{\mathbf{m}}| ^2  \<\mathbf{J} P_c   R ^{+}_{\im \mathbf{L}_{1} }(\boldsymbol{\lambda} \cdot \mathbf{m})  \chi _{B^2}\mathcal{R}_{ \mathbf{m}} ,  \mathcal{R}_{\mathbf{m}}\> \ge \Gamma \sum_{\mathbf{m}\in \mathbf{R}_{\mathrm{min}}}  \boldsymbol{\lambda}\cdot \mathbf{m}  |\mathbf{z}^{\mathbf{m}}| ^2.  \label{eq:line1111aa}
\end{align}
So we have expanded the integral in the left hand side of \eqref{eq:claim:1} as a sum of terms which are $o_{\varepsilon}(1) \epsilon ^2$ plus  the integral in $(0,t)$ of the left hand side of  \eqref{eq:line1111aa}.
 We conclude
\begin{align*} & \sum_{\mathbf{m}\in \mathbf{R}_{\mathrm{min}}}  \|\mathbf{z}^{\mathbf{m}}\| ^2 _{L^2(I)} =o_{\varepsilon}(1) \epsilon ^2,
\end{align*}
completing the proof of  Proposition \ref{prop:FGR}. \qed

\section{Repulsivity of the $\phi^8$ model near the $\phi^4$ model} \label{sec:phi8}
In this section, we study that the following nonlinear potential,
\begin{align*}
W_\epsilon(u):=
	\frac{1}{4}(1+\epsilon)^2\(u^2-1\)^2
	\(\epsilon u^2-1\)^2,\ \epsilon\in[0,1),
\end{align*}
which appears in the $\phi^8$ theory.
Notice that when $\epsilon=0$, $W_0$ is the nonlinear potential of the $\phi^4$ theory.
It was shown by \cite{KMMvdB21AnnPDE} that for $2-\sqrt{3}\leq \epsilon<1$, $L_{2,\epsilon}$ has repulsive potential, in the sense of the definition in \cite{KMMvdB21AnnPDE}.
Here, $L_{1,\epsilon}$ is given by $-\partial_x^2+W_{\epsilon}''(H_\epsilon)$ with $H_\epsilon$   the odd kink satisfying $H_{\epsilon}''=W'(H_\epsilon)$ and $L_{j,\epsilon}$  given by Darboux transformations in section \ref{sec:darboux}.

Recall that the potential $V_{2,\epsilon}$ of the 1st transformed operator $L_{2,\epsilon}=-\partial_x^2+V_{2,\epsilon}$ is given by $-W_\epsilon''(H_\epsilon)+\frac{(W_\epsilon'(H_\epsilon))^2}{W_\epsilon(H_\epsilon)}$.
So, to check the repulsivity of $V_{2,\epsilon}$, one only needs to study the function $-W_\epsilon''(x)+\frac{(W_\epsilon'(x))^2}{W_\epsilon(x)}$ in the domain $x\in[-1,1]$ because $H_\epsilon$ is monotone.
This was the very nice observation of \cite{KMMvdB21AnnPDE}.

On the other hand, when $\epsilon=0$, $L_{1,0}$ has two eigenvalues ($0$ and $\frac{3}{2}$), so $L_{2,0}$ is not   repulsive and the 2nd transformed operator $L_{3,0}=-\partial_x^2+2$ has a flat potential, which lies in the boundary of repulsive potential and is not a repulsive potential  in our definition, Assumption \ref{ass:repuls}).

Since it seems that as $\epsilon$ increases, the number of eigenvalues decreases,  it is natural to expect that $V_{3,\epsilon}$ is repulsive for $\epsilon \in (0, \epsilon _*)$ for the first  $\epsilon_*>0$ when $L_{1,\epsilon_*}$ stops to have two eigenvalues.
We will confirm this observation by computing the 1st order expansion of $V_{3,\epsilon}=2+\epsilon \widetilde{V}_{3}+O(\epsilon^2)$ and by numerical computation.
First, $\widetilde{V}_3$ can be computed explicitly.
\begin{proposition}\label{prop:phi8}
We have
\begin{align}\label{eq:tildeV3epsilon}
\widetilde{V}_3=\frac{6}{5}\sech^{2} \(\frac{x}{\sqrt{2}}\)+\frac{3}{5}\sech^{4} \(\frac{x}{\sqrt{2}}\).
\end{align}
In particular, we have $x\widetilde{V}_3'(x)<0$ for $x\neq 0$.
\end{proposition}
Next, the result of numerically computation of $V_{3,\epsilon}$ is given by the following graph.

\begin{figure*}[t]
  \begin{center}
    \includegraphics[height=9cm,keepaspectratio
    ]{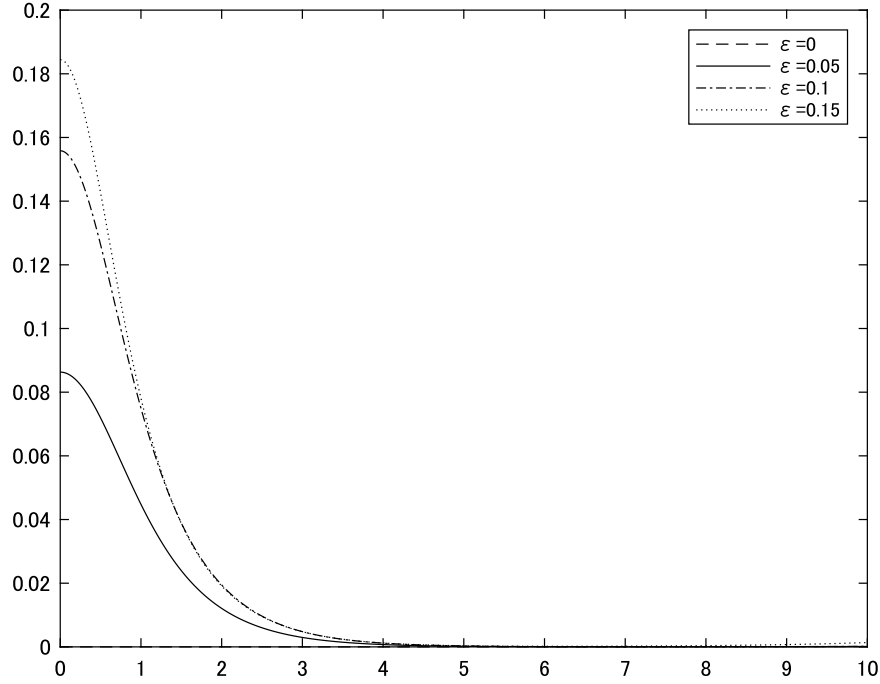}
  \end{center}
  \caption{Graph of $V_{3,\epsilon}-2+4\epsilon^2-2\epsilon^4$ generated by numerically computing $H_{\epsilon}$ and $\varphi_{\epsilon}$.
  The case $\epsilon=0$ is not visible because it is flat.}
  \label{fig:image_label}
\end{figure*}

\begin{proof}[Proof of Proposition \ref{prop:phi8}]
First, by multiplying $H_\epsilon'$ to $H_\epsilon''=W_\epsilon'(H_\epsilon)$ and integrating it, we have
\begin{align*}
H_{\epsilon}'=\sqrt{2W_{\epsilon}(H_{\epsilon})},
\end{align*}
which gives a implicit representation of the kink $H_\epsilon$ by
\begin{align*}
x=\int_0^{H_\epsilon} \frac{dh}{\sqrt{2W_\epsilon(h)}}.
\end{align*}
The above formula holds for any nonlinear potential $W$.
In our case, we can compute the integral in the right hand side and obtain
\begin{align}\label{kinkcompute1}
\sqrt{2}(1-\epsilon^2)x=\log(1+H) - \log(1-H) +\sqrt{\epsilon} \log(1-\sqrt{\epsilon}H) -\sqrt{\epsilon} \log(1+\sqrt{\epsilon}H)
\end{align}
When $\epsilon=0$, we can solve \eqref{kinkcompute1} w.r.t.\ $H_0$ and obtain the $\phi^4$-kink:
\begin{align}\label{phi4kink}
H_0=\tanh\(\frac{x}{\sqrt{2}}\)
\end{align}
Differentiating, \eqref{kinkcompute1} w.r.t.\ $\epsilon$, we have
\begin{align*}
\partial_\epsilon H_\epsilon = \frac{\(1-H_\epsilon^2\)\(1-\epsilon H_\epsilon^2\)}{2(1-\epsilon)} \(-\frac{1}{2\sqrt{\epsilon}}\(\log(1-\sqrt{\epsilon}H_\epsilon)-\log(1+\sqrt{\epsilon}H_\epsilon)\)+\frac{H_\epsilon}{1-\epsilon H_\epsilon^2}-2\sqrt{2}\epsilon x\),
\end{align*}
and by $\lim_{h\to 0}\frac{\log(1+hH_\epsilon)}{h}=H_\epsilon \left.\frac{d}{dh}\right|_{h=0}\log(1+h )=H_\epsilon$, we have
\begin{align}\label{depH0}
\left.\partial_{\epsilon}\right|_{\epsilon=0}H_\epsilon=H_0(1-H_0^2).
\end{align}

We set $\psi_{\epsilon}:=\psi_0+\epsilon\widetilde{\psi}_{\epsilon}$, with  $\widetilde{\psi}_{\epsilon}\perp \psi_0$,  to be the eigenfunction of
\begin{align*}
L_{\epsilon}:=-\partial_x^2 + W_{\epsilon}''(H_{\epsilon}),
\end{align*}
associated to the eigenvalue $\lambda_{\epsilon}=\frac{3}{2}+\epsilon \widetilde{\lambda}_{\epsilon}$, where
\begin{align}\label{phi4ev}
\psi_0(x):=\(\frac{9}{8}\)^{\frac{1}{4}}\frac{\sinh(\frac{x}{\sqrt{2}})}{\cosh^2(\frac{x}{\sqrt{2}})}
\end{align}
is the normalized eigenvector of $L_0$ satisfying $L_0\psi_{0}=\frac{3}{2}\psi_{0}$.
\begin{remark}
By the stability of eigenvalues, $L_{\epsilon}$ has a unique eigenvalue near $\frac{3}{2}$.
\end{remark}
We set $A_{1,\epsilon}=(H_\epsilon')^{-1}\partial_x \(H_\epsilon'\cdot\)$ and $\varphi_{\epsilon}=A_{1,\epsilon}^* \psi_{\epsilon}$.
Since the 2nd transformed potential $V_{3,\epsilon}$ is given by
\begin{align}\label{eq:V3epsilon}
V_{3,\epsilon}=\tilde{V}_{2,\epsilon}(H_\epsilon)-2\(\frac{\varphi_\epsilon'}{\varphi_{\epsilon}}\)',
\end{align}
with $\tilde{V}_{2,\epsilon}$ is given by $V_{2,\epsilon}=\tilde{V}_{2,\epsilon}(H_\epsilon)$, which can be explicitly written as
\begin{align*}
\tilde{V}_{2,\epsilon}(x)&=-W_\epsilon(x) (\log W_{\epsilon}(x))''
\\&=(1+3\epsilon+3\epsilon^2+\epsilon^3)+(1-2\epsilon-6\epsilon^2-2\epsilon^3+\epsilon^4)x^2-\epsilon(1+\epsilon)^3x^4+2\epsilon^2(1+\epsilon)^2x^6,
\end{align*}
it suffices to compute $\left.\partial_{\epsilon}\right|_{\epsilon=0}\varphi_{\epsilon}=A_{1,0}^*\widetilde{\psi}+\(\left.\partial_{\epsilon}\right|_{\epsilon=0}A_{1,\epsilon}^*\) \psi_{0}$.

Expanding $L_{\epsilon}\psi_{\epsilon}=\lambda_{\epsilon}\psi_{\epsilon}$, we have
\begin{align*}
\(L_0-\frac{3}{2}\)\widetilde{\psi}_{\epsilon}= -\frac{W_{\epsilon}''(H_\epsilon)-W_0''(H_0)}{\epsilon} \psi_0 + \widetilde{\lambda}_{\epsilon}\psi_0 + \epsilon\(-\frac{W_{\epsilon}''(H_\epsilon)-W_0''(H_0)}{\epsilon}\widetilde{\psi}_{\epsilon} +\widetilde{\lambda}_{\epsilon}\widetilde{\psi}_{\epsilon}\).
\end{align*}
Thus, taking $\epsilon\to 0$, we have
\begin{align}\label{3}
	\(L_0-\frac{3}{2}\)\widetilde{\psi}_{0}= -\(W_0'''(H_0)\left.\partial_{\epsilon}\right|_{\epsilon=0}H_{\epsilon}+\left.\partial_{\epsilon}\right|_{\epsilon=0}W_{\epsilon}''(H_0)\) \psi_0 + \widetilde{\lambda}_{0}\psi_0 .
\end{align}
Here, $\widetilde{\lambda}_0$ is determined from the orthogonality condition:
\begin{align*}
\widetilde{\lambda}_0=\<\(W'''(H_0)\left.\partial_{\epsilon}\right|_{\epsilon=0}H_{\epsilon}+\left.\partial_{\epsilon}\right|_{\epsilon=0}W_{\epsilon}''(H_0)\)\psi_0,\psi_0\>.
\end{align*}
From \eqref{depH0},
we have
\begin{align}\label{W'''+W''ep}
	W_0'''(H_0)\left.\partial_{\epsilon}\right|_{\epsilon=0}H_{\epsilon}+\left.\partial_{\epsilon}\right|_{\epsilon=0}W_{\epsilon}''(H_0)=-3+24H_0^2-21H_0^4.
\end{align}
Therefore, from  \eqref{phi4kink} and \eqref{phi4ev},
\begin{align}
	\widetilde{\lambda}_0
	=-3+\frac{3}{2}\int \(24\frac{\sinh^2x }{\cosh^2x}-21\frac{\sinh^4x }{\cosh^4 x}\)\frac{\sinh^2x}{\cosh^4x}\,dx
	=\frac{12}{5}.\label{lambda0}
\end{align}
From \eqref{lambda0} and \eqref{W'''+W''ep}, \eqref{3} can be written as
\begin{align}\label{phi8:tildepsi}
\(L_0-\frac{3}{2}\)\widetilde{\psi}_{0}= \(\frac{27}{5}-24H_0^2+21H_0^4\) \psi_0.
\end{align}
Let $A_{2,\epsilon}=\varphi_{\epsilon}^{-1}\partial_x\(\varphi_{\epsilon}\cdot\)$.
Applying $A_{1,0}^*$ to \eqref{phi8:tildepsi}, from $A_{1,0}^*(L_0-3/2)=A_{2,0}A_{2,0}^* A_{1,0}^*$we have
\begin{align}
A_{2,0}A_{2,0}^*A_{1,0}^*\widetilde{\psi}_0=A_0^* \(\frac{27}{5}-24H_0^2+21H_0^4\) \psi_0.
\end{align}
Solving this, we have
\begin{align}\label{A0tildepsi}
	A_0^* \widetilde{\psi}_0=
	-\(\frac{9}{8}\)^{\frac{1}{4}}\sqrt{2}\frac{1}{\cosh(\frac{x}{\sqrt{2}})}
	\(
		\frac{6}{5}\log \(\cosh \frac{x}{\sqrt{2}}\)
		-\frac{27}{10} \frac{1}{\cosh^2(\frac{x}{\sqrt{2}})}
		+3 \frac{1}{\cosh^4(\frac{x}{\sqrt{2}})}
	\).
\end{align}
This provides all the ingredients for the computation of $\widetilde{V}_3$ by differentiating \eqref{eq:V3epsilon}.
After elementary but somewhat long computation, we obtain \eqref{eq:tildeV3epsilon}.
\end{proof}

\begin{remark}
In \cite{KM22}, the asymptotic stability in the odd setting for the odd kink of $\phi^8$ model near the $\phi^4$ model is shown.
They show this result by proving $\phi^4$ model is asymptotically stable and all models near $\phi^4$ model are also asymptotically stable.
\end{remark}

\section*{Acknowledgments}
C. was supported by a FRA of the University of Trieste and by the Prin 2020 project \textit{Hamiltonian and Dispersive PDEs} N. 2020XB3EFL.
M. was supported by the JSPS KAKENHI Grant Number 19K03579, G19KK0066A and JP17H02853.

Department of Mathematics and Geosciences,  University
of Trieste, via Valerio  12/1  Trieste, 34127  Italy.
{\it E-mail Address}: {\tt scuccagna@units.it}

Department of Mathematics and Informatics,
Graduate School of Science,
Chiba University,
Chiba 263-8522, Japan.
{\it E-mail Address}: {\tt maeda@math.s.chiba-u.ac.jp}

\end{document}